\documentclass{amsart}
\usepackage{amssymb}
\usepackage{amsfonts}
\usepackage{geometry}
\usepackage{texdraw}
\usepackage{tikz}

\usepackage[allcolors=blue, pagebackref=true]{hyperref}

\hypersetup{colorlinks=True}

\newtheorem{theorem}{Theorem}
\theoremstyle{plain}

\newtheorem{definition}[theorem]{Definition}

\newtheorem{lemma}[theorem]{Lemma}

\newtheorem{proposition}[theorem]{Proposition}
\newtheorem{remark}[theorem]{Remark}

\numberwithin{equation}{section}
\input{tcilatex}
\allowdisplaybreaks

\begin{document}
\title[Two weight inequality]{A two weight inequality for Calder\'{o}%
n--Zygmund operators on spaces of homogeneous type with applications}
\author[X. T. Duong]{Xuan Thinh Duong}
\address{Xuan Thinh Duong, Department of Mathematics\\
Macquarie University\\
NSW, 2109, Australia.}
\email{xuan.duong@mq.edu.au}
\thanks{X. Duong's research supported by ARC DP 190100970.}
\author[J. Li]{Ji Li}
\address{Ji Li, Department of Mathematics\\
Macquarie University\\
NSW, 2109, Australia.}
\email{ji.li@mq.edu.au}
\thanks{J. Li's research supported by ARC DP 170101060.}
\author[E. Sawyer]{Eric T. Sawyer}
\address{Eric T. Sawyer, Department of Mathematics, McMaster University, 
Hamilton, Ontario, Canada.}
\email{sawyer@mcmaster.ca }
\thanks{E. T. Sawyer's research supported by NSERC}
\author[M. N. Vempati]{Manasa N. Vempati}
\address{Manasa N. Vempati, Department of Mathematics\\
Washington University -- St. Louis\\
One Brookings Drive\\
St. Louis, MO USA 63130-4899}
\email{m.vempati@wustl.edu}
\author[B. D. Wick]{Brett D. Wick}
\address{Brett D. Wick, Department of Mathematics\\
Washington University -- St. Louis\\
One Brookings Drive\\
St. Louis, MO USA 63130-4899}
\email{wick@math.wustl.edu}
\thanks{B. D. Wick's research supported in part by NSF grant DMS-1800057 as well as ARC DP190100970.}
\author[D. Yang]{Dongyong Yang}
\address{Dongyong Yang, Department of Mathematics\\
Xiamen University\\
Xiamen 361005,  China.}
\email{dyyang@xmu.edu.cn}
\thanks{D. Yang's research supported by NNSF of China \#11971402 and \#11871254.}
\date{\today}

\setcounter{tocdepth}{1}
\begin{abstract}
Let $(X,d,\mu )$ be a space of homogeneous type in the sense of Coifman and
Weiss, i.e. $d$ is a quasi metric on $X$ and $\mu $ is a positive measure
satisfying the doubling condition. Suppose that $u$ and $v$ are two
locally finite positive Borel measures on $(X,d,\mu )$. Subject to the pair
of weights satisfying a side condition, we characterize the boundedness of a
Calder\'{o}n--Zygmund operator $T$ from $L^{2}(u)$ to $L^{2}(v)$ in terms of
the $A_{2}$ condition and two testing conditions. For every cube $B\subset X$
, we have the following testing conditions, with $\mathbf{1}_{B}$ taken as
the indicator of $B$ 
\begin{equation*}
\Vert T(u\mathbf{1}_{B})\Vert _{L^{2}(B, v)}\leq \mathcal{T}\Vert 1_{B}\Vert
_{L^{2}(u)},
\end{equation*}
\begin{equation*}
\Vert T^{\ast }(v\mathbf{1}_{B})\Vert _{L^{2}(B, u)}\leq \mathcal{T}\Vert
1_{B}\Vert _{L^{2}(v)}.
\end{equation*}
The proof uses stopping cubes and corona decompositions originating in
work of Nazarov, Treil and Volberg, along with the pivotal side condition.
\end{abstract}

\maketitle
\tableofcontents


\section{Introduction and Statement of Main Results}
\label{s1}

The two weight conjecture for Calder\'on--Zygmund operators $T$ was first raised by Nazarov, Treil and Volberg on finding the necessary and sufficient conditions on the two weights $u$ and $v$ so that $T$ is bounded from $L^2(u)$ to $L^2(v)$. 
The third author, in \cite{s88-0}  first introduced the testing conditions (which are referred to as the Sawyer-type testing conditions) into the two weight setting on the maximal function, and later \cite{s88} on the fractional and Poisson integral operators; serving as motivation for the investigation by Nazarov, Treil and Volberg, see for example \cite{NTV}.

This conjecture in the special case that the pair of weights $u$ and $v$ do not share a common point mass was completely solved only recently when $T$ is the Hilbert transform on $\mathbb R$. This was supplied in two papers,  Lacey--Sawyer--Shen--Uriarte-Tuero \cite{LaSaShUr}, and Lacey \cite{L} that built on pioneering work of Nazarov, Treil, and Volberg \cite{NTV}.
The central question is providing a real-variable characterisation of the inequality
\begin{align}\label{two weight H}
\sup_{0<\alpha<\beta<\infty} \| H_{\alpha,\beta}(f\, v)\|_{L^2(w)}\leq \mathcal N \|f\|_{L^2(u)},
\end{align}
where $\mathcal N$ is the best constant such that the above inequality holds, $H_{\alpha,\beta}(f\, v)(x)$ is the standard truncation of the usual Hilbert transform, and $u, v$ are non-negative Borel locally finite measures on $\mathbb R$. The full solution is as follows.

\noindent {\bf Theorem A.}\ \ {\it Suppose that for all $x\in \mathbb R$, $u(\{x\})\cdot v(\{x\})=0$ for the pair of weights $u$ and $v$. Define two positive constants $\mathcal A_2$ and $\mathcal T$ as the best constants in the inequalities below, uniform over intervals $I$:
\begin{align}
&P(u,I)\cdot P(v,I) \leq \mathcal A_2;\label{Poisson}\\
&\int_I H(1_I u)^2 dv\leq \mathcal T^2 u(I),\qquad \int_I H(1_I v)^2 du\leq \mathcal T^2 v(I).\label{testing}
\end{align}
Then \eqref{two weight H} holds if and only if both \eqref{Poisson} and \eqref{testing} hold,  moreover, $ \mathcal N \approx \mathcal A_2^{1/2} + \mathcal T $.
}

In the theorem above, the term $P(u,I)$ is the Poisson integral with respect to the measure $u$ at the scaling level $|I|$ and centred at $x_I$, that is, 
$$ P(u,I) :=\int_{\mathbb R} { |I| \over (|I| + {\rm dist}(x,I))^2 } du(x). $$
The restriction regarding common point masses was removed by Hyt\"onen in \cite{Hyt2}.


The aim of this paper is to provide sufficient conditions for the two-weight
inequality for general Calder\'{o}n--Zygmund operators on spaces of homogeneous
type. Since we are working in a very general setting, the best we can hope
for at the moment is to provide a collection of sufficient conditions on the
weights that guarantee two weight estimates for Calder\'{o}n--Zygmund
operators. Our main approach is a suitable version of stopping cubes and corona decompositions originating in
work of Nazarov, Treil and Volberg \cite{NTV}, along with the pivotal\ side condition. 

Spaces of homogeneous type were introduced by Coifman and Weiss\footnote{%
As Yves Meyer remarked in his preface to \cite{DH}, \textquotedblleft One is
amazed by the dramatic changes that occurred in analysis during the
twentieth century. In the 1930s complex methods and Fourier series played a
seminal role. After many improvements, mostly achieved by the Calder\'on--Zygmund school, the action takes place today on
spaces of homogeneous type. No group structure is available, the Fourier
transform is missing, but a version of harmonic analysis is still present.
Indeed the geometry is conducting the analysis.\textquotedblright } in the
early 1970s, in \cite{CW1}, see also \cite{CW}. %
%
We say that $(X,d,\mu )$ is a {space of homogeneous type} in the sense of
Coifman and Weiss if $d$ is a quasi-metric on~$X$ and $\mu $ is a nonzero
measure satisfying the doubling condition. A \emph{quasi-metric}~$d$ on a
set~$X$ is a function $d:X\times X\longrightarrow \lbrack 0,\infty )$
satisfying

\begin{enumerate}
\item[(i)] $d(x,y) = d(y,x) \geq 0$ for all $x$, $y\in X$;

\item[(ii)] $d(x,y) = 0$ if and only if $x = y$; and

\item[(iii)] the \emph{quasi-triangle inequality}: there is a constant $%
A_{0}\in \lbrack 1,\infty )$ such that for all $x$, $y$, $z\in X$, \vspace{%
0.25cm} 
\begin{equation*}
d(x,y)\leq A_{0}[d(x,z)+d(z,y)].
\end{equation*}
\end{enumerate}

We say that a nonzero measure $\mu$ satisfies the \emph{doubling condition}
if there is a constant $C_{\mu}$ such that for all $x\in X$ and $r > 0$,
\begin{equation}
\mu (B(x,2r))\leq C_{\mu }\mu (B(x,r))<\infty ,  \label{doubling condition}
\end{equation}%
where $B(x,r)$ is the quasi-metric ball defined by $B(x,r):=\{y\in
X:d(x,y)<r\}$ for $x\in X$ and $r>0$. \vspace{0.25cm}

\noindent Recall that the doubling condition (\ref{doubling condition})
implies that there exists a positive constant $n$ (the \emph{upper dimension}
of~$\mu $) 
such that for all $x\in X$, $m\geq 1$ and $r>0$, 
\begin{equation*}
\mu (B(x,mr))\leq C_{\mu }m^{n}\mu (B(x,r)).
\end{equation*}%
Throughout this paper we assume that $\mu (X)=\infty $ and that $\mu
(\{x_{0}\})=0$ for every $x_{0}\in X$.

We now recall the definition of Calder\'{o}n--Zygmund operators on
spaces of homogeneous type.

\begin{definition}
\label{def 1} We say that $T$ is a Calder\'{o}n--Zygmund operator on $%
(X,d,\mu )$ if $T$ is bounded on $L^{2}(X)$ and has an associated kernel $%
\mathfrak K(x,y)$ such that $T(f)(x)=\int_{X}\mathfrak K(x,y)f(y)d\mu (y)$ for any $x\not\in 
\mathrm{supp}\,f$, and $\mathfrak  K(x,y)$ satisfies the following estimates: for all $%
x\not=y$, 
\begin{equation}
|\mathfrak K(x,y)|\leq {\frac{{C}}{{V(x,y)}}},  \label{size of C-Z-S-I-O}
\end{equation}%
and for $d(x,x^{\prime })\leq (2A_{0})^{-1}d(x,y)$, 
\begin{equation}
|\mathfrak K(x,y)-\mathfrak K(x^{\prime },y)|+|\mathfrak K(y,x)-\mathfrak K(y,x^{\prime })|\leq {\frac{C}{V(x,y)}}%
\omega \left( {\frac{d(x,x^{\prime })}{d(x,y)}}\right) ,
\label{smooth of C-Z-S-I-O}
\end{equation}%
where $V(x,y):=\mu (B(x,d(x,y)))$, $\omega :[0,1]\rightarrow \lbrack 0,\infty
)$ is continuous, increasing, subadditive, and $\omega (0)=0$. 
\end{definition}
Note that by the
doubling condition we have that $V(x,y)\approx V(y,x)$.  In the theorem below we have taken $\omega (t)=t^{\kappa }$ for some $\kappa\in(0,1)$ for the kernel
estimate \eqref{smooth of C-Z-S-I-O} and we refer to $\kappa$ as the smoothness parameter for the kernel $\mathfrak{K}(x,y)$. 

The main result of this paper provides sufficient conditions on a pair of weights $u$ and $v$ so that the following two-weight norm inequality
\begin{align}  \label{N}
& \|T(f\cdot u)\|_{L^2(v)}\leq \mathcal{N }\|f\|_{L^2(u)}
\end{align}
holds for a Calder\'{o}n--Zygmund operator $T$ on $(X,d,\mu )$, where $\mathcal N$ is the best constant (understood as the operator norm).

We also define $l(Q)$ (in Section 3.1)
 to be the side-length of a dyadic cube $Q$, and for all $Q=Q_{k}^{\alpha }\in \mathcal{D}_{k}$ where $%
\left\{ \mathcal{D}_{k}\right\} _{k\in \mathbb{Z}}$ is a system of dyadic
cubes with the parameter $\delta $ as given in Definition \ref{def dyadic
system} below. In particular we have $l(Q)= \delta ^{k}$. 

We say
that a set $Q$ is a \emph{cube}, or more precisely a $\left(
c_{1},C_{1},\delta \right) $-cube, if there is $z\in X$ such that%
\begin{equation*}
B\left( z,c_{1}\delta ^{k}\right) \subset Q_{\alpha }^{k}\left( \omega
\right) \subset B\left( z,C_{1}\delta ^{k}\right) ,
\end{equation*}%
where $c_{1}$, $C_{1}$ and $\delta $ are positive constants. The cubes
appearing in the main theorem below are those with $c_{1}$, $C_{1}$ and $%
\delta $ as in (3) of Theorem \ref{theorem dyadic} below.  Also for all cubes $Q$ and $x\in X\setminus Q$ we define $\func{dist}%
(x,Q):=\inf\limits_{q\in Q}\{d(q,x):q\in Q\}$ where $d$ is the quasi-metric
on $X$.  Recall that a measure $u$ is locally finite if for any point in $X$ there exists a neighborhood $B$ about that point so that $u(B)<\infty$.

We need two quantities that control certain quantities in the proof that control certain information uniformly over cubes $Q\subset X$.  The first is a version of an $\mathcal{A}_2$ condition.  Suppose the pair of weights $u,v$ satisfies the following $\mathcal{A}%
_{2}$ condition for all cubes $Q$: 
\begin{equation}\label{thm2 testing 1}
\left( \frac{u(Q)}{l(Q)^{n}}K(Q,v)\right) ^{\frac{1}{2}}\lesssim \mathcal{A}%
_{2}
\end{equation}%
with $K(Q,v):=\int_{X}\left(\frac{l(Q)}{l(Q)+\operatorname{dist}(y,Q)}\right)^{\kappa }\frac{1}{\mu
(B(x_{Q},l(Q)+\operatorname{dist}(y,Q)))}dv(y)$, as well as the dual condition $$\left( 
\frac{v(Q)}{l(Q)^{n}}K(Q,u)\right) ^{\frac{1}{2}}\lesssim \mathcal{A}_{2},$$
where $\mathcal A_2$ is the best constant such that the above inequalities hold.
Recall here that $\kappa$ is the smoothness parameter associated to the Calder\'on--Zygmund kernel in Definition \ref{def 1}.

The second is the pivotal condition: 
\begin{equation}\label{thm2 pivotal}
\sup\limits_{Q=\cup _{i\geq 1}S_{i}}\sum_{i\geq 1}\Phi (S_{i},1_Qu)\leq 
\mathcal{V}^{2}u(Q),
\end{equation}%
where  $\Phi (Q,1_Eu):= v(Q)K(Q,1_{E}u)^{2}$, as well as the dual version, in which $u$ and $v$ are interchanged.
Here $\mathcal V$ is the best constant, and the supremum is over all $r $-good subpartitions $%
\{S_{i}\}_{i\geq 1}$ of $Q$ where $r$ is defined in Definition \ref{d:newdef}.  An $r$-good subpartition consists of $Q$-dyadic subcubes $\{S_i\}$ of $Q$ such that $S_i$ is $r$-good in any  dyadic grid containing $Q$.

With these preliminaries, our main result is the following:
\begin{theorem}
\label{main theorem} Let $T$ be a Calder\'{o}n--Zygmund operator with smoothness parameter $\kappa$.  Let $u$ and $v$ be two locally finite, positive Borel measures on $X$. Suppose that $u(\{x\})\cdot v(\{x\})=0$ for $x\in X$ and that they satisfy the two weight condition with constant $\mathcal{A}_2$ and the pivotal condition with constant $\mathcal{V}$.  
Then $T:L^2(u)\to L^2(v)$ is bounded if and only if the following testing conditions hold:  for every cube $Q\subset X$, we have the
following testing conditions, with $1_{Q}$ taken as the indicator
of $Q$ 
\begin{equation}
\Vert T(u1_{Q})\Vert _{L^{2}(Q,v)}\leq \mathcal{T}\Vert 1_{Q}\Vert _{L^{2}(u)},
\end{equation}%
\begin{equation}
\Vert T^{\ast }(v1_{Q})\Vert _{L^{2}(Q,u)}\leq \mathcal{T}\Vert
1_{Q}\Vert _{L^{2}(v)}.
\end{equation}
Moreover, we have that $\mathcal{N}\lesssim \mathcal{A}_{2}+\mathcal{T}+\mathcal{V}$.
\end{theorem}

\begin{remark}
We would like to point out that we introduce a new version of a Poisson-type integral $K(Q,v)$ in the main Theorem that plays the role of the standard Poisson integral as in \eqref{Poisson} in Theorem A. The main reason for providing such a condition is that, for some Calder\'on--Zygmund operators in certain particular setting, the typical Poisson integral in that setting is not linked directly to the study of two weight inequality for the Calder\'on-Zygmund operator. We refer to Section \ref{sec:Bessel} for a concrete example in the Bessel setting introduced and studied by Muckenhoupt--Stein \cite{ms}.  

We also remark that the choice of $\kappa$ is flexible, but dictated by the smoothness of the Calder\'on-Zygmund kernel.  If one has a different kernel possessing a different smoothness, but satisfies the appropriate $\mathcal{A}_2$, testing and pivotal conditions, then one can have a version of Theorem \ref{main theorem}.  The choice of $\kappa$ can be dictated by the particular example at hand.
\end{remark}


It is immediate that the testing conditions are necessary and that $\mathcal{%
T}\lesssim \mathcal{N}$. The forward condition follows
by testing \eqref{N} on an indicator function of a cube and restricting the region of
integration. The dual condition follows by testing the dual inequality %
\eqref{N} (obtained by interchanging the roles of $u$ and $v$) on the
indicator of a cube and then again restricting the integration. In the
remainder of the paper we address how to show that these testing conditions
are sufficient to prove \eqref{N} under the additional $A_{2}$ and pivotal
hypothesis. In the course of the proof we will also demonstrate that $%
\mathcal{N}\lesssim \mathcal{A}_{2}+\mathcal{T}+\mathcal{%
V}$. Throughout the paper, we use the notation $X\lesssim Y$ to denote that
there is an absolute constant $C$ so that $X\leq CY$, where $C$ may change from one occurrence to another. If we write $X\approx
Y $, then we mean that $X\lesssim Y$ and $Y\lesssim X$. And, $:=$ means
equal by definition.

\subsection{Extension to Hilbert space valued operators}

Following \cite[Chapter II, Section 5]{Ste}, we consider two Hilbert spaces $%
\mathcal{H}_{1}$ and $\mathcal{H}_{2}$, and replace the scalar-valued
expressions in the definition of singular integral 
\begin{equation*}
Tf\left( x\right) =\int_{X}\mathfrak K\left( x,y\right) f\left( y\right) d\mu \left(
y\right) ,
\end{equation*}%
with the appropriate Hilbert space valued expressions, namely with $%
f:X\rightarrow \mathcal{H}_{1}$ and $K:X\times X\rightarrow B\left( \mathcal{%
H}_{1},\mathcal{H}_{2}\right) $, so that $Tf:X\rightarrow \mathcal{H}_{2}$.
Here $B\left( \mathcal{H}_{1},\mathcal{H}_{2}\right) $ is the Banach space
of bounded linear operators $L:\mathcal{H}_{1}\rightarrow \mathcal{H}_{2}$
equipped with the usual operator norm. We refer to such an operator $T$ as
an $H_{1}\rightarrow H_{2}$\emph{\ Calder\'{o}n--Zygmund operator} if its
kernel satisfies the usual size and smoothness conditions%
\begin{eqnarray}
&&\left\vert\mathfrak  K\left( x,y\right) \right\vert _{B\left( \mathcal{H}_{1},%
\mathcal{H}_{2}\right) }\leq C_{CZ} V\left( x,y\right)^{-1},  \label{size Hilbert} \\
&&\left\vert \mathfrak  K\left( x,y\right) -\mathfrak  K\left( x^{\prime },y\right)
\right\vert _{B\left( X\times X,B\left( \mathcal{H}_{1},\mathcal{H}%
_{2}\right) \right) }\leq C_{CZ}\left( \frac{d\left( x,x^{\prime }\right) }{%
d\left( x,y\right) }\right) ^{\kappa } V(x,y)^{-1},\ \ \ \ \ \frac{d\left( x,x^{\prime }\right) }{d\left( x,y\right) }%
\leq \frac{1}{2A_0},  \notag
\end{eqnarray}%
and if $T$ is bounded from unweighted $L_{\mathcal{H}_{1}}^{2}$ to
unweighted $L_{\mathcal{H}_{2}}^{2}$. In the Appendix, Section \ref{s:Appendix}, to this paper we
will fix two separable Hilbert spaces $\mathcal{H}_{1}$ and $\mathcal{H}_{2}$%
, and describe in detail the definition and interpretation of standard
fractional singular integrals, the weighted norm inequality, Poisson
integrals and Muckenhoupt conditions, Haar bases and pivotal conditions,
which are for the most part routine.

\begin{theorem}
\label{Hilbert}Let $\mathcal{H}_{1}$ and $\mathcal{H}_{2}$ be separable
Hilbert spaces. Let $T$ be a Calder\'{o}n--Zygmund operator taking $L_{%
\mathcal{H}_{1}}^{2}$ to $L_{\mathcal{H}_{2}}^{2}$.  Let $u$ and $v$ be two locally finite positive Borel measures on a space of
homogeneous type $\left( X,d,\mu \right) $. Suppose that $u(\{x\})\cdot
v(\{x\})=0$ for $x\in X$. Suppose the above $\mathcal{A}_{2}$ and pivotal
conditions hold.  Suppose the following
testing conditions hold: for every cube $Q\subset X$, we have the following
testing conditions, with $\mathbf{1}_{Q}$ taken as the indicator of $Q$: 
\begin{equation*}
\Vert T(\mathbf{e}_{1}\mathbf{1}_{Q} u)\Vert _{L_{\mathcal{H}_{2}}^{2}(v)}\leq 
\mathcal{T}\Vert \mathbf{1}_{Q}\Vert _{L^{2}(u)},\ \ \ \
\ \text{for all unit vectors }\mathbf{e}_{1}\text{ in }\mathcal{H}_{1}
\end{equation*}%
\begin{equation*}
\Vert T^{\ast }(\mathbf{e}_{2}\mathbf{1}_{Q} v)\Vert _{L_{\mathcal{H}%
_{1}}^{2}(u)}\leq \mathcal{T}\Vert \mathbf{1}_{Q}\Vert _{L^{2}(v)},\ \ \ \ \ \text{for all unit vectors }\mathbf{e}_{2}\text{
in }\mathcal{H}_{2}.
\end{equation*}%
Then there holds $\mathcal{N}\lesssim \mathcal{A}_{2}+\mathcal{T}+\mathcal{V}$.
\end{theorem}

To see how this theorem follows from the scalar-valued Theorem \ref{main
theorem}, consider the scalar operators $T_{\mathbf{e}_{1},\mathbf{e}_{2}}$
associated with $T$ for every pair of unit vectors $\left( \mathbf{e}_{1},%
\mathbf{e}_{2}\right) \in \mathcal{H}_{1}^{\func{unit}}\times \mathcal{H}%
_{2}^{\func{unit}}$ whose kernels $K_{\mathbf{e}_{1},\mathbf{e}_{2}}\left(
x,y\right) $ are given by 
\begin{equation*}
K_{\mathbf{e}_{1},\mathbf{e}_{2}}\left( x,y\right) :=\left\langle\mathfrak  K\left(
x,y\right) \mathbf{e}_{1},\mathbf{e}_{2}\right\rangle _{\mathcal{H}_{2}}.
\end{equation*}%
It is easy to see that 
\begin{equation*}
\left\Vert T\right\Vert _{L_{\mathcal{H}_{1}}^{2}(u)\rightarrow L_{\mathcal{H%
}_{2}}^{2}(v)}=\sup_{\left( \mathbf{e}_{1},\mathbf{e}_{2}\right) \in 
\mathcal{H}_{1}^{\func{unit}}\times \mathcal{H}_{2}^{\func{unit}}}\left\Vert
T_{\mathbf{e}_{1},\mathbf{e}_{2}}\right\Vert _{L^{2}(u)\rightarrow
L^{2}(v)},
\end{equation*}%
with a similar equality for the testing conditions. Theorem \ref{Hilbert}
now follows immediately from Theorem \ref{main theorem}.

\section{Applications of the Main Theorem}

In this section we provide several typical examples of Calder\'on--Zygmund operators arising from different backgrounds (including several complex variables, stratified Lie groups, and differential equations), which are not the standard Euclidean setting, but fall into the scope of spaces of homogeneous type.

\subsection{Bessel Riesz transforms}
\label{sec:Bessel}

As an application, we have a two-weight inequality for the Bessel Riesz
transform, which is a Calder\'on--Zygmund operator \cite{ms}. In 1965, Muckenhoupt and Stein in \cite{ms} introduced a notion
of conjugacy associated with this Bessel operator ${\Delta _{\lambda }}$, $\lambda>0$, 
which is defined by 
\begin{equation*}
{\Delta _{\lambda }}f(x):=-\frac{d^{2}}{dx^{2}}f(x)-\frac{2\lambda }{x}\frac{%
d}{dx}f(x),\quad x>0.
\end{equation*}%
They developed a theory in the setting of ${\Delta _{\lambda }}$ which
parallels the classical one associated to standard Laplacian. For $p\in \lbrack
1,\infty )$, ${\mathbb{R}}_{+}:=(0,\infty )$ and ${dm_{\lambda }}%
(x):=x^{2\lambda }\,dx$ results on ${L^{p}({{\mathbb{{\mathbb{R}}}}_{+}}%
,\,dm_{\lambda })}$-boundedness of conjugate functions and fractional
integrals associated with ${\Delta _{\lambda }}$ were obtained. Since then,
many problems based on the Bessel context were studied; see, for example,
\cite{ak,bcfr,bhnv,k78,v08}. In particular, the properties and $L^{p}$ boundedness $(1<p<\infty )$ of
Riesz transforms 
\begin{equation*}
R_{\Delta _{\lambda }}f:=\partial _{x}(\Delta _{\lambda })^{-\frac{1}{2}}f
\end{equation*}%
related to $\Delta _{\lambda }$ have been studied extensively (see for example \cite{ak,bcfr,bfbmt,ms,v08}).

We recall that there is a standard Poisson integral  in the Bessel setting.
Let $\left\{\mathsf{P}^{[\lambda]}_t\right\}_{t>0}$ be the Poisson semigroup $\left\{e^{-t\sqrt{\Delta_{\lambda}}}\right\}_{t>0}$ defined by
\begin{equation*}
\mathsf{P}^{[\lambda]}_tf(x):=\int_0^\infty \mathsf{P}^{[\lambda]}_t(x,y)f(y)y^{2\lambda}\,dy,
\end{equation*}
where
\begin{equation*}
\mathsf{P}^{[\lambda]}_t(x,y)=\int_0^\infty e^{-tz}(xz)^{-\lambda+\frac{1}{2}}J_{\lambda-\frac{1}{2}}(xz)(yz)^{-\lambda+\frac{1}{2}}J_{\lambda-\frac{1}{2}}(yz)z^{2\lambda}\,dz
\end{equation*}
and $J_{\nu}$ is the Bessel function of the first kind and of  order $\nu$.  Weinstein \cite{w48} established the following formula for $\mathsf{P}^{[\lambda]}_t(x,y)$: $t,\,x,\, y\in\mathbb{R}_+$,
\begin{equation}\label{poisson kernel}
\mathsf{P}^{[\lambda]}_t(x,y)={2\lambda t\over\pi} \int_0^{\pi} \frac{(\sin\theta)^{2\lambda-1}}{(x^2+y^2+t^2-2xy\cos\theta)^{\lambda+1}}\,d\theta.
\end{equation}
The two weight inequality for this Poisson operator was just established recently in \cite{LW}, that is,  for a measure $\mu$ on $\mathbb{R}^2_{+,+}:=(0,\infty)\times (0,\infty)$ and $\sigma$ on $\mathbb{R}_+$:
$$
\|\mathsf{P}^{[\lambda]}_\sigma(f)\|_{L^2(\mathbb{R}^2_{+,+};\mu)} \lesssim \|f\|_{L^2(\mathbb{R}_+;\sigma)},
$$
if and only if testing conditions hold for the Poisson operator and its adjoint. However, this two weight Poisson inequality does not relate directly to the two weight inequality for $R_{\Delta _{\lambda }}$. We have to link it to the Poisson type condition introduced in \eqref{thm2 testing 1}.

\subsection{Bergman Projection}

As another application we look at the Bergman projection. Let $\mathcal{D} :=
\{z\in \mathbb{C}^n: |z|<1\}$ be the unit ball in $\mathbb{C}^n$; $%
d\mu_{a}(\zeta):= (1-|\zeta|^2)^{a-1}d\mu(\zeta)$ where $a>0$ and $\mu$ is
Lebesgue measure on $\mathbb{C}^n$. Let us denote by $L^p(d\mu_a)$ the
Lebesgue space related to $\mu_a$, $1\leq p \leq \infty$. The Bergman
projection $T_{a}f$ for $f\in L^2(d\mu_a)$ is given by 
\begin{equation*}
T_{a}f(z) := \int_{\mathcal{D}}\frac{f(\zeta)}{(1-z\cdot\overline{\zeta})^{n+a}}%
d\mu(\zeta).
\end{equation*}
Here we have $z\cdot\overline{\zeta} := z_1\overline{\zeta}_{1}+\cdots+z_n\overline{\zeta}_{n}$
where $z:=(z_1,\ldots,z_n)$ and $\zeta :=
(\zeta_{1},\ldots,\zeta_{n})$. The operator $T_a$ extends
continuously on $L^p(d\mu_a)$ for $1<p<\infty$ and weakly continuously on $%
L^1(d\mu_a)$. If $\mathcal{D}$ is provided with the pseudo-distance $d$ then the
triple $(\mathcal{D},d,\mu_a)$ is a space of homogeneous type. Note that $%
K_a(z,\zeta) := \frac{1}{(1-z\cdot \overline{\zeta})^{n+a}}$ is the kernel associated to
the operator $T_a$. We can observe that $K_a(z,\zeta)$ satisfies the
following smoothness and size estimates: there are constants $\beta,c_B$
such that 
\begin{equation}
|K_a(z,\zeta)-K_a(z,\zeta^0)|\lesssim \frac{[d(\zeta,\zeta^0)]^{\beta} }{%
[d(z,\zeta^0)]^{n+a+\beta}}
\end{equation}
for $z,\zeta,\zeta^0$ such that $d(z,\zeta^0)>c_{B}d(\zeta,\zeta^0)$.  Here the pseudo-distance $d$ is defined by 
\begin{equation*}
d(z,\zeta):= ||z|-|\zeta||+\left|1-\frac{z\cdot\zeta}{|z||\zeta|}\right|.
\end{equation*}
Also the following size estimate holds: for $z,\zeta$ in $\mathcal{%
D}$  with $z\not=\zeta$, we have 
\begin{equation}
|K_a(z,\zeta)|\lesssim \frac{1}{d(z,\zeta)^{n+a}}.
\end{equation}

So $T_{a}$ is a singular integral operator on $(\mathcal{D},d,\mu _{a})$.
Using Theorem \ref{main theorem} we deduce a two weight inequality for the Bergman projection using $\mathcal{A}_{2}$, testing conditions and the pivotal condition associated with the kernel $K_{a}$.

\subsection{The Szeg\H{o} Projection on a Family of Unbounded Weakly
Pseudoconvex Domains}

Recall the family of weakly pseudoconvex domains $\left\{ \Omega
_{k}\right\} _{k=1}^{\infty }$ defined in Greiner and Stein \cite{GrSt}
by%
\begin{eqnarray*}
\Omega _{k} &:= &\left\{ \left( z_{1},z_{2}\right) \in \mathbb{C}^{2}:%
\func{Im}z_{2}>\left\vert z_{1}\right\vert ^{2k}\right\} , \\
\partial \Omega _{k} &:= &\left\{ \left( z_{1},z_{2}\right) \in \mathbb{C%
}^{2}:\func{Im}z_{2}=\left\vert z_{1}\right\vert ^{2k}\right\} ,
\end{eqnarray*}%
which are naturally parameterized by $z_{1}$ and $\func{Re}z_{2}$. We
consider the points $\zeta ,\omega ,\nu $ in $\partial \Omega _{k}$ given by%
\begin{eqnarray*}
\zeta &:=&\left( z_{1},\func{Re}z_{2}\right) := \left( z,t\right) ,\ \ \ z=z_{1}\in 
\mathbb{C}\text{ and }t\in \mathbb{R}\text{,} \\
\omega &:=&\left( w_{1},\func{Re}w_{2}\right) := \left( w,s\right) ,\ \ \ w=w_{1}\in 
\mathbb{C}\text{ and }s\in \mathbb{R}\text{,} \\
\nu &:=&\left( u_{1},\func{Re}u_{2}\right) := \left( u,r\right) ,\ \ \ u=u_{1}\in 
\mathbb{C}\text{ and }r\in \mathbb{R}\text{.}
\end{eqnarray*}%
The Szeg\H{o} projection $\mathcal{S}$ on $\Omega _{k}$ is the orthogonal
projection from $L^{2}\left( \partial \Omega _{k}\right) $ to the Hardy
space $H^{2}\left( \Omega _{k}\right) $ of holomorphic functions on $\Omega
_{k}$\ with $L^{2}$ boundary values. The Szeg\"{o} kernel $S\left( \zeta
,\omega \right) $ is the unique kernel satisfying%
\begin{equation*}
\mathcal{S}f\left( \zeta \right) :=\int_{\partial \Omega _{k}}S\left( \zeta
,\omega \right) f\left( \omega \right) dV\left( \omega \right) ,
\end{equation*}%
where $dV\left( \omega \right) :=dV\left( x,y,s\right) =dxdyds$ with $\omega
:=\left( w,s\right) =\left( x+iy,s\right) $ is Lebesgue measure on the
parameter space $\mathbb{R}^{3}$. Greiner and Stein \cite{GrSt} have
computed the Szeg\"{o} kernel with Lebesgue measure on the parameter space
via the formula%
\begin{eqnarray*}
S\left( \zeta ,\omega \right) &:=&\frac{1}{4\pi ^{2}}\left[ \left( \frac{i}{2}%
\left[ s-t\right] +\frac{\left\vert z_{1}\right\vert ^{2k}+\left\vert
w_{1}\right\vert ^{2k}}{2}+\frac{\mu +\eta }{2}\right) ^{\frac{1}{k}}-z_{1}%
\overline{w_{1}}\right] ^{-2} \\
&&\ \ \ \ \ \ \ \ \ \ \times \left( \frac{i}{2}\left[ s-t\right] +\frac{%
\left\vert z_{1}\right\vert ^{2k}+\left\vert w_{1}\right\vert ^{2k}}{2}+%
\frac{\mu +\eta }{2}\right) ^{\frac{1}{k}-1},
\end{eqnarray*}%
where $\mu :=\func{Im}z_{2}-\left\vert z_{1}\right\vert ^{2k}$ and $\eta :=%
\func{Im}w_{2}-\left\vert w_{1}\right\vert ^{2k}$.

In \cite{Dia}, Diaz defined and analyzed a pseudometric $d\left( \zeta
,\omega \right) $, globally suited to the complex geometry of $\partial
\Omega _{k}$, which was arrived at by a study of the Szeg\"{o} kernel. This
allows the treatment of the Szeg\"{o} kernel as a singular integral kernel:%
\begin{equation*}
d\left( \zeta ,\omega \right) :=\left\vert \left( \frac{i}{2}\left[ s-t\right]
+\frac{\left\vert z_{1}\right\vert ^{2k}+\left\vert w_{1}\right\vert ^{2k}}{2%
}+\frac{\mu +\eta }{2}\right) ^{\frac{1}{k}}-z_{1}\overline{w_{1}}%
\right\vert ^{\frac{1}{2}}.
\end{equation*}%
Then the pseudometric balls are defined as $B_{\zeta} \left( \delta \right)
=B_{\zeta}^{d}\left( \delta \right) := \left\{ \omega \in \partial \Omega
_{k}:d\left( \zeta ,\omega \right) <\delta \right\} $ and the volume of the associated 
ball is%
\begin{equation*}
V\left( B_{\zeta} \left( \delta \right) \right) =4\pi \delta ^{2}\left( \frac{%
\sin ^{2k-2}\left( \frac{\pi }{k}\right) }{4}\left\vert z\right\vert
^{2k-2}\delta ^{2}+\frac{1}{2}\delta ^{2k}\right) ,
\end{equation*}%
and it is shown that this volume measure is doubling. Thus $S$ is a standard
Calder\'{o}n--Zygmund operator in the setting of the space of homogeneous
type $\left( \partial \Omega _{k},d,V\right) $.  We can again deduce a two-weight inequality in this setting from Theorem \ref{main theorem}.

\subsection{Riesz Transforms Associated with the sub-Laplacian on Stratified Nilpotent Lie Groups}

Recall that a connected, simply connected nilpotent Lie group $\mathcal{G}$
is said to be stratified if its left-invariant Lie algebra $\mathfrak{g}$
(which is assumed real and of finite dimension) admits a direct sum
decomposition%
\begin{equation*}
\mathfrak{g}=\oplus _{i=1}^{k}V_{i}\ ,\ \ \ \ \ \text{where }\left[
V_{1},V_{i}\right] =V_{i+1}\text{ for }i\leq k-1.
\end{equation*}%
One identifies $\mathcal{G}$ and $\mathfrak{g}$ via the exponential map $%
\exp :\mathfrak{g}\rightarrow \mathcal{G}$, which is a diffeomorphism. We
fix once and for all a (bi-invariant) Haar measure $dg$ on $\mathcal{G}$
(which is just the lift of Lebesgue measure on $\mathfrak{g}$ via $\exp $).
There is a natural family of dilations on $\mathfrak{g}$ defined for $r>0$
as follows:%
\begin{equation*}
\delta _{r}\left( \sum_{i=1}^{k}v_{i}\right) =\sum_{i=1}^{k}r^{i}v_{i}\ ,\ \
\ \ \ \text{with }v_{i}\in V_{i}\ .
\end{equation*}%
This permits the definition of a dilation on $\mathcal{G}$, which we
continue to denote by $\delta _{r}$. We choose a basis $%
\left\{ X_{1},...,X_{n}\right\} $ for $V_{1}$ and consider the sub-Laplacian 
$\bigtriangleup := \sum_{j=1}^{n}X_{j}^{2}$. Observe that $X_{j}$, $%
1\leq j\leq n$, is homogeneous of degree $1$, and that $\bigtriangleup $ is
homogeneous of degree $2$, with respect to dilations in the sense that $%
X_{j}\left( f\circ \delta _{r}\right) =r\left( X_{j}f\right) \circ \delta
_{r}$, $1\leq j\leq n$, and $\delta _{\frac{1}{r}}\circ \bigtriangleup \circ
\delta _{r}=r^{2}\bigtriangleup $ for all $r>0$.

Let $Q$ denote the homogeneous dimension of $\mathcal{G}$, namely\ $%
Q=\sum_{i=1}^{k}i\dim V_{i}$. Let $p_{h}$ for $h>0$ be the heat kernel, i.e.
the integral kernel of $e^{h\bigtriangleup }$ on $\mathcal{G}$. For
convenience we set $p_{h}\left( g\right) =p_{h}\left( g,o\right) $, which
means that we identify the
integral kernel with the convolution kernel, and we set $p\left( g\right)
=p_{1}\left( g\right) $.

Recall that (c.f. Folland and Stein \cite{FoSt}) $p_{h}\left( g\right) =h^{-%
\frac{Q}{2}}p\left( \delta _{\frac{1}{\sqrt{h}}}\left( g\right) \right) $
for all $h>0$ and $g\in \mathcal{G}$. The kernel of the $j^{th}$ Riesz
transform $R_{j}=X_{j}\left( -\bigtriangleup \right) ^{-\frac{1}{2}}$, $%
i\leq j\leq n$, is written simply as $K_{j}\left( g,g^{\prime }\right)
=K_{j}\left( \left( g^{\prime }\right) ^{-1}g\right) $ where%
\begin{equation*}
K_{j}\left( g\right) =\frac{1}{\sqrt{\pi }}\int_{0}^{\infty }h^{-\frac{1}{2}%
}X_{j}p_{h}\left( g\right) dh=\frac{1}{\sqrt{\pi }}\int_{0}^{\infty }h^{-%
\frac{Q}{2}-1}\left( X_{j}p\right) \left( \delta _{\frac{1}{\sqrt{h}}%
}g\right) dh.
\end{equation*}%
The standard metric $d$ on $G$ is defined as $d\left( g,g^{\prime }\right)
:= \rho \left( \left( g^{\prime }\right) ^{-1}g\right) $ where $\rho $
is the homogeneous norm on $\mathcal{G}$ (\cite[Chapter 1, Section A]{FoSt}%
). The measure $dg$ \ is then a doubling measure. It is well known that $S$
is a standard Calder\'{o}n--Zygmund operator in the setting of the space of
homogeneous type $\left( \mathcal{G},d,dg\right) $, and once more we can deduce a two weight inequality from Theorem \ref{main theorem}.

\subsection{Area Functions\label{Subsection area}}

In the Euclidean homogeneous space $\left( \mathbb{R}^{n},\left\vert \cdot
\right\vert ,dx\right) $, the\ Littlewood-Paley $g$-function and the Lusin
area function are both examples of $\mathcal{H}_{1}\rightarrow \mathcal{H}%
_{2}$ Calder\'{o}n--Zygmund operators with $\mathcal{H}_{1}=\mathbb{C}$.
Indeed they are given by%
\begin{eqnarray*}
\mathbf{g}\left( f\right) \left( x\right) &=&\left( \int_{0}^{\infty
}\left\vert \nabla \left( P_{t}\ast f\right) \left( x\right) \right\vert
^{2}tdt\right) ^{\frac{1}{2}}=\left\vert Gf\left( x\right) \right\vert _{%
\mathcal{G}_{n}}, \\
\boldsymbol{S}f\left( x\right) &=&\left( \iint_{\Gamma _{0}}\left\vert \nabla
\left( P_{t}\ast f\right) \left( x-y\right) \right\vert ^{2}\frac{dtdy}{%
t^{n-1}}\right) ^{\frac{1}{2}}=\left\vert Sf\left( x\right) \right\vert _{%
\mathcal{H}_{n+1}},
\end{eqnarray*}%
where%
\begin{eqnarray}
Gf\left( x\right) \left( t\right) &:= &\nabla \left( P_{t}\ast f\right)
\left( x\right) ,\ \ \ \ \ t\in \left( 0,\infty \right),  \label{def S} \\
Sf\left( x\right) \left( t,y\right) &:= &\nabla \left( P_{t}\ast
f\right) \left( x-y\right) ,\ \ \ \ \ \left( t,y\right) \in \Gamma _{0}\ , 
\notag
\end{eqnarray}%
and $\mathcal{G}_{n}$ and $\mathcal{H}_{n+1}$ are Hilbert spaces with norms%
\begin{equation*}
\left\vert g\right\vert _{\mathcal{G}_{n}}:= \sqrt{\int_{0}^{\infty
}\left\vert g\left( t\right) \right\vert ^{2}tdt}\text{\ \ \ \ \  and\ \ \  }\left\vert
h\right\vert _{\mathcal{H}_{n+1}}:= \sqrt{\iint_{\Gamma _{0}}\left\vert
h\left( t,y\right) \right\vert ^{2}t^{1-n}dtdy}<\infty ,
\end{equation*}%
and $\Gamma _{0}$ is a fixed cone with vertex at the origin in $\mathbb{R}%
^{n+1}$ opening upward into the upper half space $\mathbb{R}_{+}^{n+1}$.

In a general\ space of homogeneous type $\left( X,d,\mu \right) $, the
notion of a Poisson kernel can be approached in several different ways.
First, if $\mathcal{D}$ is a dyadic grid on $X$ and $\{h_{Q}:Q\in \mathcal{D}%
\}$ is the collection of Haar wavelets constructed in \cite{KLPW}, then with
the one-dimensional projection $\bigtriangleup _{Q}$ defined by $%
\bigtriangleup _{Q}f:=\left\langle f,h_{Q}\right\rangle_{\mu} h_{Q}$, we have%
\begin{equation}
f=\sum_{Q\in \mathcal{D}}\bigtriangleup _{Q}f=\sum_{Q\in \mathcal{D}%
}\bigtriangleup _{Q}\bigtriangleup _{Q}f,\ \ \ \ \ f\in L^{2}\left( \mu
\right) ,  \label{ortho wavelet}
\end{equation}%
where the orthonormal property and the self-adjointness of Haar projections gives the second sum $\sum_{Q\in \mathcal{D}%
}\bigtriangleup _{Q}\bigtriangleup _{Q}f$, usually called a Calder\'{o}n
reproducing formula for $f$. We can then define a discrete Poisson operator
by%
\begin{equation*}
P_{k}f:= \sum_{Q\in \mathcal{D}:\ l \left( Q\right) \geq
2^{k}}\bigtriangleup _{Q}\bigtriangleup _{Q}f.
\end{equation*}%
However, the kernel of $P_{k}$ is not Lipschitz continuous, and thus fails
to be a Calder\'{o}n--Zygmund kernel as defined above. The following smoother
construction of a Poisson kernel$\ $had been introduced much earlier by R.
Coifman, see \cite{DaJoSe} where this first appears.

\subsubsection{Coifman's Construction of a Calder\'{o}n Reproducing Formula}

Start with a smooth function $h:\left( 0,\infty \right) \rightarrow \left(
0,\infty \right) $ that equals $1$ on $\left( 0,\frac{1}{2}\right] $ and $0$
on $\left[ 2,\infty \right) $. Let $T_{k}$ be the operator with kernel $%
2^{k}h\left( 2^{k}d\left( x,y\right) \right) $ so that $\frac{1}{C}\leq T_{k}%
\mathbf{1}\leq C$ for some positive constant $C$. Let $M_{k}$ be the
operator of multiplication by $\frac{1}{T_{k}\mathbf{1}}$, and let $W_{k}$
be the operator of multiplication by $\frac{1}{T_{k}\left( \frac{1}{T_{k}%
\mathbf{1}}\right) }$. Then the operator 
\begin{equation*}
S_{k}:= M_{k}T_{k}W_{k}T_{k}M_{k}\ ,
\end{equation*}%
has kernel $S_{k}\left( x,y\right) $ that satisfies%
\begin{eqnarray}
&&S_{k}\left( x,y\right) =0\text{ if }d\left( x,y\right) \geq C\frac{1}{2^{k}%
}\text{ and }\left\Vert S_{k}\right\Vert _{\infty }\leq C2^{k},  \label{sat}
\\
&&\left\vert S_{k}\left( x,y\right) -S_{k}\left( x^{\prime },y\right)
\right\vert +\left\vert S_{k}\left( y,x\right) -S_{k}\left( y,x^{\prime
}\right) \right\vert \leq C2^{k}\left[ 2^{k}d\left( x,x^{\prime }\right) %
\right] ^{\varepsilon },  \notag \\
&&\int_{X} S_{k}\left( x,y\right) d\mu \left( y\right) := 1\text{ and }\int_{X}
S_{k}\left( x,y\right) d\mu \left( x\right) := 1.  \notag
\end{eqnarray}%
Thus with $D_{k}:= S_{k+1}-S_{k}$ we have%
\begin{equation*}
f=\sum_{k\in \mathbb{Z}}D_{k}f,\ \ \ \ \ f\in L^{2}\left( \mu \right) .
\end{equation*}%
Next fix $N\in \mathbb{N}$ so large that if $T^{N}:= \sum_{k\in \mathbb{Z%
}}D_{k}^{N}D_{k}$ where $D_{k}^{N}:= \sum_{\left\vert j\right\vert \leq
N}D_{k+j}$, then $T^{N}$ is invertible on $L^{2}\left( \mu \right) $. It
follows that%
\begin{equation*}
f=T^{N}T^{-N}f=\sum_{k\in \mathbb{Z}}D_{k}^{N}D_{k}T^{-N}f=\sum_{k\in 
\mathbb{Z}}E_{k}\widetilde{E}_{k}f,\ \ \ \ \ f\in L^{2}\left( \mu \right) ,
\end{equation*}%
where $E_{k}:= D_{k}^{N}$ and $\widetilde{E}_{k}:= D_{k}T^{-N}$.
This latter formula is usually called a Calder\'{o}n reproducing formula,
and substitutes for the orthonormal wavelet formula (\ref{ortho wavelet}).

\subsubsection{Discrete $g$-functions and Area functions}

The function%
\begin{equation*}
\boldsymbol{g}\left( f\right) \left( x\right) := \sum_{k=-\infty
}^{\infty }\left\vert D_{k}f\left( x\right) \right\vert ^{2}=\left\vert
Gf\left( x\right) \right\vert _{\ell ^{2}\left( \mathbb{Z}\right) },\ \ \ \
\ Gf\left( x\right) := \left\{ D_{k}f\left( x\right) \right\} _{k\in 
\mathbb{Z}}
\end{equation*}%
plays the role of a $g$-function in a homogeneous space $X$. We could also
replace $D_{k}$ by $E_{k}:= D_{k}^{N}$ in this formula giving an
alternative $g$-function. Property (\ref{sat}), together with an application
of the Cotlar-Stein lemma, shows that in either case $\boldsymbol{g}$ is a
Hilbert space valued Calder\'{o}n--Zygmund operator on $X$. The function%
\begin{equation*}
\boldsymbol{S}f\left( x\right) := \sum_{k=-\infty }^{\infty }\left\vert 
\widetilde{E}_{k}f\left( x\right) \right\vert ^{2}=\left\vert Sf\left(
x\right) \right\vert _{\ell ^{2}\left( \mathbb{Z}\right) },\ \ \ \ \
Sf\left( x\right) := \left\{ \widetilde{E}_{k}f\left( x\right) \right\}
_{k\in \mathbb{Z}},
\end{equation*}%
plays the role of an area function in $X$, and \cite[Theorem 3.4]{HaSa}
shows that the kernel of $\boldsymbol{S}$ satisfies (\ref{size Hilbert}),
and the boundedness on $L^{2}\left( \mu \right) $ is proved in \cite{DaJoSe}%
. Thus $\boldsymbol{S}$ is a Hilbert space valued Calder\'{o}n--Zygmund
operator on $X$, and Theorem \ref{Hilbert} yields the following two weight
norm inequality - a stronger result was obtained in Euclidean space by\
Lacey and Li \cite{LaLi}, stronger in the sense that neither the dual
testing condition nor the dual pivotal condition was needed and they didn't need to test over all unit vectors.

\begin{theorem}
Let $u $ and $v $ be two locally finite positive Borel measures on 
$X$. Suppose that $u(\{x\})\cdot v(\{x\})=0$ for $x\in X$.
Suppose the above $\mathcal{A}_{2}$ and pivotal conditions hold. Suppose
also that the following testing conditions hold for the area function $%
\boldsymbol{S}$ as defined above: for every ball $B\subset X$, we have the
following testing conditions, with $\mathbf{1}_{Q}$ taken as the indicator
of $Q$: 
\begin{equation*}
\Vert \boldsymbol{S}(u \mathbf{1}_{Q})\Vert _{L_{\mathcal{H}%
_{2}}^{2}(v )}\leq \mathcal{T}\Vert \mathbf{1}_{Q}\Vert _{L^{2}(u
)},
\end{equation*}%
\begin{equation*}
\Vert \boldsymbol{S}^{\ast }(v \mathbf{e}_2\mathbf{1}_{Q})\Vert _{L^{2}(u
)}\leq \mathcal{T}\Vert 1_{Q}\Vert _{L^{2}(v)} \quad \text{for all unit vectors }\mathbf{e}_{2}\text{
in }\mathcal{H}_{2}.
\end{equation*}%
Then there holds $\mathcal{N}\lesssim \mathcal{A}_{2}+\mathcal{T}+%
\mathcal{V}$.
\end{theorem}

\subsection{Classical Riesz transforms}

As an application to known work we can look at the two weight inequality for
the Hilbert Transform.  We define our Hilbert operator as below for a signed
measure $v $ on $\mathbb{R}$: 
\begin{equation*}
Hv (x):= p.v.\int_{\mathbb{R}}\frac{1}{x-y}v (dy).
\end{equation*}%
For two weights $u ,v $, the inequality we are interested in is 
\begin{equation*}
\Vert H(u f)\Vert _{L^{2}(v )}\lesssim \Vert f\Vert _{L^{2}(u
)}.
\end{equation*}%
Along with $\mathcal{A}_{2}$ for the pair of weights $u ,v $ given
we have the following testing conditions holding uniformly over intervals $I$%
: 
\begin{align*}
& \int_{I}|H(1_{I}u )|^{2}v (dx)\leq \mathcal{H}^{2}u (I),
\\
& \int_{I}|H(1_{I}v )|^{2}u (dx)\leq \mathcal{H}^{2}v (I).
\end{align*}%
Here $\mathcal{H}$ denotes the smallest constants for which these inequalities are true uniformly over all intervals $I$.

In beautiful series of papers, Nazarov, Treil and Volberg, \cite{NTV1,NTV2,NTV} have developed a
sophisticated approach towards proving the sufficiency of these testing
conditions combined with the improvement of the two weight $\mathcal{A}_{2}$
condition. The improvement is described below using a variant of Poisson
integral. For an interval $I$ and measure $v $, 
\begin{align*}
& P(I,v ):= \int_{\mathbb{R}}\frac{|I|}{(|I|+\func{dist}(x,I))^{2}}%
\omega (dx), \\
& \sup_{I}P(I,v )\cdot P(I,u ):=\mathcal{A}_{2}^{2}<\infty .
\end{align*}%
We will refer to the last line above as the $\mathcal{A}_{2}$ condition. In \cite[Theorem 2.2]{NTV} Nazarov, Treil and Volberg proved the
sufficiency of the $\mathcal{A}_{2}$ and testing conditions above for the
two weight inequality of the Hilbert transform in the presence of the pivotal condition given by 
\begin{equation*}
\sum_{r=1}^{\infty }v (I_{r})P(I_{r},1_{I_{0}}u )^{2}\leq \mathcal{V}^{2}u (I_{0}),
\end{equation*}%
and its dual, where the inequality is required to hold for all intervals $%
I_{0}$ and decompositions $\{I_{r}:r\geq 1\}$ of $I_{0}$ into disjoint
intervals $I_{r}\subsetneq I_{0}$.  We have taken inspiration from this proof and condition in providing the proof of our main result.  In \cite{L,Lac,LaSaUr2,LaSaShUr} the pivotal condition was removed as a side condition and replaced with an energy condition, ultimately yielding a characterization of the two weight inequality for the Hilbert transform.


In higher dimension we look at two weight inequalities for Riesz transforms.   Earlier work appears in \cite{LaWi, SaShUr} where certain two weight inequalities for the Riesz transforms (and fractional versions) were studied.
Namely for two weights, nonnegative locally finite Borel measures $u,v $ on $%
\mathbb{R}^{n}$, we are interested in the following inequality for the $d$
dimensional Riesz transform 
\begin{equation*}
\left\Vert \int_{\mathbb{R}^{n}}f(y)\frac{x-y}{|x-y|^{d+1}}u
(dy)\right\Vert _{L^{2}(v )}\leq \mathcal{N}\Vert f\Vert _{L^{2}(u
)}.
\end{equation*}%
Here we take $0<d\neq n-1\leq n$ and $\mathcal{N}$ is the best constant in
the inequality above. 

The $\mathcal{A}_{2}$ type condition is expressed
in terms of a Poisson type operator. For a cube $Q\subset \mathbb{R}^{n}$, we take 
\begin{equation*}
P(u ,Q):= \int_{\mathbb{R}^{n}}\frac{|Q|^{d/n}}{|Q|^{2d/n}+\func{%
dist}(x,Q)^{2d}}u (dx).
\end{equation*}%
Using the $\mathcal{A}_{2}$, testing conditions and pivotal condition we
can obtain sufficient conditions for the two weight inequality for the $d$
dimensional Riesz transform by the above Theorem \ref{main theorem} on $(\mathbb{R}%
^{n},dx,|\cdot |_{n})$, viewed as space of homogeneous type where $|\cdot|_{n}$
is a standard metric on $\mathbb{R}^{n}$.

\subsection{Riesz Transform Associated with Certain Schr\"odinger Operators}

Consider $L:=-\Delta+\mu$, which is a Schr\"odinger operator  with a non-negative Radon measure  $\mu$ on $\mathbb{R}^n$ for $n\geq3$.
We assume that $\mu$ satisfies  the following conditions: there exists a
positive constant  $\sigma_0\in(1,\infty)$ such that
\begin{align}\label{condition1}
\mu(B(x,r))\lesssim \left(\frac{r}{R}\right)^{n-2+\sigma_0} \mu(B(x,R))
\end{align}
and
\begin{align}\label{condition2}
\mu(B(x,2r))\lesssim \left\{\mu(B(x,r))+r^{n-2}\right\}
\end{align}
for all $x\in \mathbb R^n$ and $0<r<R$, where $B(x,r)$ denotes the open ball centered at $x$ with radius $r$.
As pointed in \cite{Sh99}, condition \eqref{condition1} may be regarded as scale-invariant Kato-condition,
and \eqref{condition2} says that the measure
$\mu$ is doubling on balls satisfying $\mu(B(x,r))\geq cr^{n-2}$. We will also assume that $\mu\not\equiv 0.$
When  $d\mu=V(x)dx$ and $V\geq 0$ is in the reverse H\"older class $(RH)_{n}$,
i.e.,
\begin{equation}
\left(\frac{1}{|B(x, r)|}\int_{B(x, r)} V(y)^{n}~dy\right)^{1/n}\leq
\frac{C}{|B(x, r)|}\int_{B(x, r)}V(y)~dy,
\end{equation}
then $\mu$ satisfies the conditions \eqref{condition1} and \eqref{condition2} for some $\sigma_0>1.$ However,
in general, measures which satisfy \eqref{condition1} and \eqref{condition2} need not be absolutely continuous
with respect to the Lebesgue measure on $\mathbb R^n.$ For instance, when $d\mu=d\sigma(x_1, x_2) dx_3\cdots dx_n$,
where $\sigma$ is a doubling measure on ${\mathbb R}^2$, then $\mu$ satisfies
\eqref{condition1} and \eqref{condition2} for some $\sigma_0>1.$

It is well-known that the Riesz transform $ \nabla L^{-{1\over2}} $ is bounded on $L^2(\mathbb R^n)$. Moreover,
let $K(x,y)$ be the kernel of  $ \nabla L^{-{1\over2}} $, the Riesz transforms associated to $L$.
Then it was proved in \cite{Sh99} that 
\begin{align}\label{size}
&|K(x,y)|\lesssim { 1\over |x-y|^n } 
\end{align}
and that for $|x-x'|\leq {1\over 2 }|x-y|$,
\begin{align}\label{smooth}
&\qquad|K(x,y)-K(x',y)|\lesssim  \Big({  |x-x'|\over |x-y| } \Big)^{\sigma_0-1}  { 1\over |x-y|^n }.
\end{align}
Here the implicit constants are independent of $x$ and $y$.


\section{Preliminaries on Spaces of Homogeneous Type\label{s:Pre}}

Let $(X,d,\mu )$ be a space of homogeneous type as in Section \ref{s1}.

\subsection{A System of Dyadic Cubes}

We will recall from \cite{TA} a construction of dyadic cubes, which is a
deep elaboration of work by M.~Christ \cite{Chr}, as well as that of
Sawyer--Wheeden \cite{ER}. We summarize the dyadic construction of random
dyadic systems from \cite{HyMa} and \cite{TA} in the following theorem.
First we need to define an appropriate notion of `reference points' or `lattice points' in $X$.

\begin{definition}
A set of points $\left\{ x_{\alpha }^{k}\right\} _{k\in \mathbb{Z},\ \alpha
\in \mathcal{A}_{k}}\subset X $ is said to be a \emph{set of reference
points} if there exist constants $0<c_{0}\leq C_{0}<\infty $ and $0<\delta
<1 $ such that $12A_{0}^{3}C_{0}\delta \leq c_{0}$ and%
\begin{eqnarray*}
d\left( x_{\alpha }^{k},x_{\beta }^{k}\right) &\geq &c_{0}\delta ^{k},\ \ \
\ \ \alpha \neq \beta , \\
\min_{\alpha }d\left( x,x_{\alpha }^{k}\right) &\leq &C_{0}\delta ^{k},\ \ \
\ \ x\in X.
\end{eqnarray*}
\end{definition}

The following construction is from \cite[Theorems 5.1 and 5.6]{TA}.

\begin{theorem}
\label{t:randomgrid}
\label{theorem dyadic}Given a set of reference points $\left\{ x_{\alpha
}^{k}\right\} _{k\in \mathbb{Z},\ \alpha \in \mathcal{A}_{k}}$ with parameters $c_0, C_0$ and $\delta$,  and sufficiently small $0<\delta <1$ $($e.g. $144A_{0}^{8}\delta \leq 1$$)$, there
exists a probability space $\left( \Omega ,\mathbb{P}\right) $ such that
every $\omega \in \Omega $ defines a dyadic system $\mathcal{D}\left( \omega
\right) :=\left\{ Q_{\alpha }^{k}\left( \omega \right) \right\} _{k\in 
\mathbb{Z},\ \alpha \in \mathcal{A}_{k}}$ related to new dyadic points $%
\left\{ z_{\alpha }^{k}\left( \omega \right) \right\} _{k\in \mathbb{Z},\
\alpha \in \mathcal{A}_{k}}$ with the following geometric properties: 
for some $c_1$ and $C_1$ depending on $c_0, C_0, A_0$ and $\delta$,
\begin{enumerate}
\item If $\ell \geq k$, then either $Q_{\beta}^{\ell }\left( \omega \right) \subset Q_{\alpha }^{k}\left( \omega \right)$ or $Q_{\beta }^{\ell }\left( \omega \right) \cap Q_{\alpha}^{k}\left( \omega \right) =\emptyset$;
\item $\displaystyle X=\bigcup\limits_{\alpha }Q_{\alpha}^{k}\left( \omega \right)$ for all $k\in \mathbb{Z}$;
\item $\displaystyle B\left( z_{\alpha }^{k}\left( \omega \right)
,c_{1}\delta ^{k}\right) \subset Q_{\alpha }^{k}\left( \omega \right)
\subset B\left( z_{\alpha }^{k}\left( \omega \right) ,C_{1}\delta
^{k}\right) =: B\left( Q_{\alpha }^{k}\left( \omega \right) \right)$;
\item If $\ell \geq k$ and $Q_{\beta }^{\ell}\left( \omega \right) \subset Q_{\alpha }^{k}\left( \omega \right)$ then $B\left(Q_{\beta }^{\ell }\left( \omega \right) \right) \subset B\left( Q_{\alpha }^{k}\left( \omega \right) \right)$ ,

\end{enumerate}
and the following probabilistic property: There are positive constants $%
C_{2},\eta >0$ such that for every $x\in X,\tau >0$ and $k\in \mathbb{Z}$, 
\begin{equation}
\mathbb{P}\left( \left\{ \omega \in \Omega :x\in \bigcup\limits_{\alpha
}\partial _{\tau \delta ^{k}}Q_{\alpha }^{k}\left( \omega \right) \right\}
\right) \leq C_2\tau ^{\eta }\ .  \label{surgery}
\end{equation}
We also have the following containment property:
\begin{equation}\label{eq:contain}
    B(x^k_{\alpha},c_1\delta^k)
    \subseteq Q^k_{\alpha}\subseteq B(x^k_{\alpha},C_1\delta^k)
    =: B(Q^k_{\alpha});
\end{equation}
\end{theorem}

\begin{definition}
\label{def dyadic system}We say that $\mathcal{D}\left( \omega \right)
=\left\{ Q_{\alpha }^{k} \left( \omega \right) \right\} _{k\in \mathbb{Z},\
\alpha \in \mathcal{A}_{k}}$ is a system of dyadic cubes if \textnormal{(1)---(4)} hold in Theorem \ref{t:randomgrid}.
Given a dyadic cube $Q_{\alpha }^{k}\left( \omega \right)$, we denote the quantity $\delta^k$ by
$l(Q_{\alpha }^{k})$, by analogy with the side length of a
Euclidean cube. 
\end{definition}

\subsection{An Explicit Haar Basis on Spaces of Homogeneous Type}

Next we recall the explicit construction in \cite{KLPW} of a Haar basis $%
\{h_{Q}^{\epsilon }:Q\in \mathcal{D},\epsilon =1,\dots ,M_{Q}-1\}$ for $%
L^{p}(X,\mu )$, $1<p<\infty $, associated to the dyadic cubes $Q\in \mathcal{%
D}$ as follows. Here $M_{Q}:=\#{\mathcal{H}}(Q)=\#\{R\in \mathcal{D}%
_{k+1}\colon R\subseteq Q\}$ denotes the number of dyadic sub-cubes (which
we will refer to as the \textquotedblleft children\textquotedblright ) the
cube $Q\in \mathcal{D}_{k}$ has; namely $\mathcal{H}(Q)$ is the collection
of dyadic children of $Q$. It is known in \cite{KLPW} that $%
\sup\limits_{Q\in \mathcal{D}}M_{Q}<\infty $.

\begin{theorem}
\label{thm:convergence} Let $(X,d,\mu )$ be a space of homogeneous type and
suppose $\mu $ is a positive locally finite Borel measure on $X$. For $1<p<\infty $, for each $%
f\in L^{p}(X,\mu )$, we have 
\begin{equation*}
f(x)=\sum_{Q\in \mathcal{D}}\sum_{\epsilon =1}^{M_{Q}-1}\langle
f,h_{Q}^{\epsilon }\rangle _{\mu }h_{Q}^{\epsilon }(x),
\end{equation*}%
where the sum converges (unconditionally) both in the $L^{p}(X,\mu )$-norm
and pointwise $\mu $-almost everywhere.
\end{theorem}

The following theorem collects several basic properties of the functions $h_{Q}^{\epsilon}$.

\begin{theorem}
\label{prop:HaarFuncProp} The Haar functions $h_{Q}^{\epsilon }$, $Q\in 
\mathcal{D}$, $\epsilon =1,\ldots ,M_{Q}-1$, have the following properties:

\begin{enumerate}

\item $h_{Q}^{\epsilon }$ is a simple Borel-measurable real function on $X$;

\item $h_{Q}^{\epsilon }$ is supported on $Q$;

\item $h_{Q}^{\epsilon }$ is constant on each $R\in {\mathcal{H}}(Q)$;

\item $\int_{Q}h_{Q}^{\epsilon }\,d\mu =0$ (cancellation);

\item $\langle h_{Q}^{\epsilon },h_{Q}^{\epsilon ^{\prime }}\rangle =0$ for $%
\epsilon \neq \epsilon ^{\prime }$, $\epsilon $, $\epsilon ^{\prime }\in
\{1,\ldots ,M_{Q}-1\}$;

\item the collection $\left\{ \mu (Q)^{-1/2}1_{Q}\right\} \cup \left\{
h_{Q}^{\epsilon }:\epsilon =1,\ldots ,M_{Q}-1\right\} $ is an orthogonal
basis for the vector space~$V(Q)$ of all functions on $Q$ that are constant
on each sub-cube $R\in {\mathcal{H}}(Q)$;

\item if $h_{Q}^{\epsilon }\not= 0$ then $\Vert h_{Q}^{\epsilon }\Vert
_{L^{p}(X,\mu )}\approx \mu (Q)^{\frac{1}{p}-\frac{1}{2}}\quad 
\text{for}~1\leq p\leq \infty ;$

\item $\Vert h_{Q}^{\epsilon }\Vert _{L^{1}(X,\mu )}\cdot \Vert
h_{Q}^{\epsilon }\Vert _{L^{\infty }(X,\mu )}\approx 1$.
\end{enumerate}
\end{theorem}

We denote $h_{Q}^{0}:=\mu (Q)^{-1/2}1_{Q}$, which is a non-cancellative Haar
function. Moreover, the martingale associated with the Haar functions are as
follows: for $Q\in \mathcal{D}_{k}$, 
\begin{equation*}
\mathbf{E}_{Q}f:=\langle f,h_{Q}^{0}\rangle _{\mu }h_{Q}^{0}\quad \mathrm{and}%
\quad \mathbb{D}_{Q}f:=\sum_{\epsilon =1}^{M_{Q}-1}\mathbb{D}_{Q}^{\epsilon
}f,
\end{equation*}%
where $\mathbb{D}_{Q}^{\epsilon }f:=\langle f,h_{Q}^{\epsilon }\rangle_{\mu}
h_{Q}^{\epsilon }$ is the martingale operator associated with the $\epsilon $%
-th subcube of $Q$. Also we have 
\begin{equation*}
\mathbf{E}_{k}f=\sum_{Q\in \mathcal{D}_{k}}\mathbf{E}_{Q}f\quad \mathrm{and}%
\quad \mathbb{D}_{k}f=\mathbf{E}_{k+1}f-\mathbf{E}_{k}f.
\end{equation*}%
Hence, based on the construction of Haar system $\{h_{Q}^{\epsilon }\}$ in 
\cite{KLPW} we obtain that for each $R\in \mathcal{D}$ and $\eta=1,\ldots, M_R-1$, 
\begin{equation}
\label{e:tagged}
\sum_{Q:\ R\subset Q}\sum_{\epsilon =1}^{M_{Q}-1}\langle f,h_{Q}^{\epsilon
}\rangle _{\mu }h_{Q}^{\epsilon }h_{R}^{\eta }=\sum_{Q:\ R\subset Q}\mathbb{D%
}_{Q}f\cdot h_{R}^{\eta }=\mathbf{E}_{R}f\cdot h_{R}^{\eta }=\langle
f,h_{R}^{0}\rangle _{\mu }h_{R}^{0}h_{R}^{\eta }.
\end{equation}

\subsection{The Carleson Embedding Theorem in Spaces of Homogeneous Type}

Now we will describe the familiar Carleson Embedding theorem, which will be
crucial in the control of certain paraproduct terms.

\begin{theorem}
\label{t:2.5} Fix a weight $u$ and consider nonnegative constants $%
\{a_{Q}:Q\in \mathcal{D}\}$. The following two inequalities are equivalent: 
\begin{align*}
& \sum_{Q\in \mathcal{D}}a_{Q}|\mathbf{E}_{Q}^{u}f|^{2}\leq c\Vert
f\Vert _{L^{2}(u)}^{2}, {\quad\rm where\ \  } \mathbf{E}_{Q}^{u}f:=\langle f,h_{Q}^{0}\rangle _{u }h_{Q}^{0}; \\
& \sum_{Q\in \mathcal{D}:Q\subset S}a_{Q}\leq Cu(S).
\end{align*}%
Taking $c$ and $C$ to be the best constants in these inequalities, we
have $c\approx C$.
\end{theorem}

\section{First Reduction in the Proof of the Two-Weight Inequality}
\label{s:MainResult}

We now begin to prove the two weight inequality in our setting.   In this section we reduce to showing that it suffices to prove Theorem \ref{main theorem} under the hypothesis that $f$ and $g$ are `good' functions (as explained below).

Let $f\in L^{2}(u)$ and $g\in L^{2}(v)$ be two functions. Without loss of generality,
we can assume that these two functions have compact support. Moreover, it is
sufficient to assume that $f$ and $g$ are supported on a common (large)
cube $Q_{0}$, see for example \cite{Vol}. From Theorem \ref{thm:convergence} we have 
\begin{equation*}
f(x)=\sum_{Q\in \mathcal{D}}\sum_{\epsilon =1}^{M_{Q}-1}\langle
f,h_{Q}^{\epsilon }\rangle _{u}h_{Q}^{\epsilon }(x).
\end{equation*}
We write this sum in two parts as follows: 
\begin{equation*}
f(x)=\sum_{Q\subset Q_{0}}\sum_{\epsilon =1}^{M_{Q}-1}\langle
f,h_{Q}^{\epsilon }\rangle _{u}h_{Q}^{\epsilon }(x)+\sum_{Q:Q_{0}\subset
Q}\sum_{\epsilon =1}^{M_{Q}-1}\langle f,h_{Q}^{\epsilon }\rangle
_{u}h_{Q}^{\epsilon }(x).
\end{equation*}%
Based on Theorem \ref{prop:HaarFuncProp} and \eqref{e:tagged} we have that 
\begin{eqnarray*}
\mathbf{E}_{Q_{0}}f\cdot h_{Q_{0}}^{\eta } &=&\langle f,h_{Q_{0}}^{0}\rangle
_{u}h_{Q_{0}}^{0}h_{Q_{0}}^{\eta }=\sum_{Q:Q_{0}\subset Q}\sum_{\epsilon
=1}^{M_{Q}-1}\langle f,h_{Q}^{\epsilon }\rangle _{u}h_{Q}^{\epsilon }, \\
f &=&\mathbf{E}_{Q_{0}}f+\sum_{Q\subset Q_{0}}\sum_{\epsilon
=1}^{M_{Q}-1}\langle f,h_{Q}^{\epsilon }\rangle _{u}h_{Q}^{\epsilon
}=:f_{1}+f_{2}.
\end{eqnarray*}%
Similarly we write for $g\in L^{2}(v)$: 
\begin{equation*}
g=\mathbf{E}_{Q_{0}}g+\sum\limits_{Q\subset Q_{0}}\sum\limits_{\epsilon
=1}^{M_{Q}-1}\langle g,h_{Q}^{\epsilon }\rangle _{v}h_{Q}^{\epsilon
}=:g_{1}+g_{2}.
\end{equation*}%
Here we have $\Vert f\Vert _{L^{2}(u)}^{2}=\Vert f_{1}\Vert
_{L^{2}(u)}^{2}+\Vert f_{2}\Vert _{L^{2}(u)}^{2}$; similar formulas hold for
the function $g$.

Let $T$ be a Calder\'{o}n--Zygmund Operator on $(X,d,\mu )$. Given these
decompositions of $f$ and $g$, let us begin the proof by looking at their
inner product 
\begin{equation*}
\langle T(uf),g\rangle _{v}=\langle T(uf_{1}),g_{1}\rangle _{v}+\langle
T(uf_{1}),g_{2}\rangle _{v}+\langle T(uf_{2}),g_{1}\rangle _{v}+\langle
T(uf_{2}),g_{2}\rangle _{v}:=I_{1}+I_{2}+I_{3}+I_{4}.
\end{equation*}%
It is enough to obtain good estimates on each of the $I_{j}$ above. The first three terms are easy to control just using the testing condition assumed on the
operator $T$, the last term will then require substantial analysis.

We now show how to control $I_{1}$, $I_{2}$ and $I_{3}$ just using the
testing condition. First observe that 
\begin{equation*}
I_{1}=\langle T(uf_{1}),g_{1}\rangle _{v}=\frac{\int_{Q_{0}}fdu\cdot
\int_{Q_{0}}gdv}{u(Q_{0})\cdot v(Q_{0})}\langle T(u1_{Q_{0}}),1_{Q_{0}}\rangle
_{v}.
\end{equation*}%
By Cauchy-Schwarz, applied to the function $f$ and the function $g$, and in
the inner product, and then using the testing conditions assumed on the
operator $T$ we have: 
\begin{equation*}
|I_{1}|\leq \mathcal{T}\Vert f\Vert _{L^{2}(u)}\Vert g\Vert _{L^{2}(v)}.
\end{equation*}%
The terms $I_{2}$ and $I_{3}$ are symmetric 
\begin{equation*}
I_{2}=\langle T(uf_{1}),g_{2}\rangle _{v}=\frac{\int_{Q_{0}}fdu}{u(Q_{0})}%
\langle T(u1_{Q_{0}}),g_{2}\rangle _{v}.
\end{equation*}%
Using Cauchy-Schwarz, the testing conditions and the fact that $\Vert
g_{2}\Vert _{L^{2}(v)}\leq \Vert g\Vert _{L^{2}(v)}$ we get the following: 
\begin{equation*}
|I_{2}|\leq \mathcal{T}\Vert f\Vert _{L^{2}(u)}\Vert g\Vert _{L^{2}(v)}.
\end{equation*}%
An identical argument works for $I_{3}$.

By the above it suffices to prove 
\begin{equation*}
|\langle T(uf),g\rangle _{v}|\lesssim \Vert f\Vert _{L^{2}(u)}\Vert g\Vert
_{L^{2}(v)}
\end{equation*}%
when $f$ and $g$ have compact support in $Q_{0}$ and $\int_{Q_{0}}fdu=0$ and $%
\int_{Q_{0}}gdv=0$. We will now decompose the inner product $\langle
T(uf),g\rangle _{v}$ using good-bad decomposition.

\begin{remark}
In order to use surgery to remove weak boundedness, we will need to work in
the world of two independent systems of random grids.
\end{remark}

\subsection{The Good and Bad Parts of Functions}

We use the good-bad decomposition of test functions to simplify the proof even
further. Fix a number $\epsilon $, $0<\epsilon <1$. Later the choice of $%
\epsilon $ will be dictated by the Calder\'{o}n--Zygmund properties of the
operator $T$ and the underlying measure $\mu $. Also fix a sufficiently
large integer $r$. The choice of $r$ will be made in this section. Finally,
we consider two grids $\mathcal{D}=\left\{ Q_{\alpha }^{k}\left( \omega
\right) \right\} _{k,\alpha }$ and $\mathcal{D}^{\prime }=\left\{ Q_{\alpha
}^{k}\left( \omega ^{\prime }\right) \right\} _{k,\alpha }$ for $\omega
,\omega ^{\prime }\in \Omega $.

\begin{definition}
\label{d:3.1} Take a dyadic cube $Q\in \mathcal{D}$. We say that $Q$ is $r$-%
\textit{good} in $\mathcal{D}^{\prime }$ for an integer $r$, if for every cube $Q_{1}\in \mathcal{D}^{\prime }$ such that if $\delta ^{k}\leq \delta
^{r}\delta ^{n}$ with $k\geq n+r$, then either 
\begin{equation*}
\func{dist}(Q,Q_{1})\geq \delta ^{k\epsilon }\delta ^{n(1-\epsilon )}\text{ or 
}\func{dist}(Q,X\setminus Q_{1})\geq \delta ^{k\epsilon }\delta ^{n(1-\epsilon
)}.
\end{equation*}%
Above we are letting $l(Q):= \delta ^{k}$ and $l(Q_1):= \delta^{n}$%
. If $Q$ is not $r$-good we call it $r$-\textit{bad}.
\end{definition}

We can now decompose $f$ into good and bad parts as below. 
\begin{eqnarray*}
f &=&f_{good}+f_{bad} \\
f_{bad}:= &&\sum_{Q\in \mathcal{D},Q\text{ is bad}}\Delta _{Q}f.
\end{eqnarray*}

\begin{theorem}{\cite[Theorem 17.1]{Vol}}
\label{t:VolbergBad} There holds on $(X,d,\mu )$ for $f\in L^{2}(u)$ 
\begin{equation*}
\mathbf{E}(\Vert f_{bad}\Vert _{L^{2}(u)})\leq \varepsilon (r)\Vert f\Vert
_{L^{2}(u)}
\end{equation*}%
where $\varepsilon (r)\rightarrow 0$ as $r\rightarrow \infty $. A similar
estimate holds for $g_{bad}\in L^{2}(v)$.
\end{theorem}

\begin{proposition}
\label{prop:Reduction to Good} Consider the decompositions of $f$ and $g$
into bad and good parts on $(X,d,\mu )$, where the support cubes of the Haar
projections of $f$ are good with respect to $\mathcal{G}$, and the support
cubes of the Haar projections of $g$ are good with respect to $\mathcal{D}$.
Let $u$ and $v$ be pairs of weights and suppose there holds uniformly over
all dyadic grids $\mathcal{D}$ and $\mathcal{G}$ for some finite constant $C$
\begin{equation}
\mathbf{E}(|\langle T(uf_{good}),g_{good}\rangle _{v}|)\leq C\Vert f\Vert
_{L^{2}(u)}\Vert g\Vert _{L^{2}(v)} ,  \label{e:RedtoGood}
\end{equation}%
where $\mathbf{E}$ refers to expectation over the product probability space $%
\Omega \times \Omega $.  Then 
\begin{equation*}
|\langle T(uf),g\rangle _{v}|\leq 2C\Vert f\Vert _{L^{2}(u)}\Vert g\Vert
_{L^{2}(v)},
\end{equation*}%
that is 
\begin{equation*}
\Vert T\Vert _{L^{2}(u)\rightarrow L^{2}(v)}\leq 2C.
\end{equation*}
\end{proposition}

\begin{proof}
Note that 
\begin{equation*}
\langle T(uf),g\rangle _{v}=\langle T(uf_{good}),g_{good}\rangle _{v}+\langle
T(uf_{good}),g_{bad}\rangle _{v}+\langle T(uf_{bad}),g\rangle _{v}.
\end{equation*}%
So 
\begin{equation*}
\mathbf{E}(|\langle T(uf),g\rangle _{v}|)\leq \mathbf{E}(|\langle
T(uf_{good}),g_{good}\rangle _{v}|)+\mathbf{E}(|\langle
T(uf_{good}),g_{bad}\rangle _{v}|)+\mathbf{E}(|\langle T(uf_{bad}),g\rangle _{v}|).
\end{equation*}

\noindent Using \eqref{e:RedtoGood} and by Theorem \ref{t:VolbergBad} we
have 
\begin{equation*}
|\langle T(uf),g\rangle _{v}|\leq C\Vert f\Vert _{L^{2}(u)}\Vert g\Vert
_{L^{2}(v)}+2\Vert T\Vert _{L^{2}(u)\rightarrow L^{2}(v)}\varepsilon
(r)\Vert f\Vert _{L^{2}(u)}\Vert g\Vert _{L^{2}(v)}.
\end{equation*}%
Notice that $\Vert Tf\Vert _{L^{2}(u)\rightarrow L^{2}(v)}=\sup |\langle
Tf,g\rangle _{v}|$. Choose $f$, $g$ and $r$ sufficiently large such that 
\begin{equation}
|\langle T(uf),g\rangle _{v}|\geq \frac{1}{2}\Vert T\Vert _{L^{2}(u)\rightarrow
L^{2}(v)}\Vert f\Vert _{L^{2}(u)}\Vert g\Vert _{L^{2}(v)}\text{ and }%
\varepsilon (r)<\frac{1}{8}.  \label{e:3.3}
\end{equation}%
Then by absorbing the second term in \eqref{e:3.3} to left hand side, we get 
\begin{equation*}
\Vert T\Vert _{L^{2}(u)\rightarrow L^{2}(v)}\leq 2C.
\end{equation*}
\end{proof}

\begin{definition}
\label{d:newdef}
We fix $r>0$ with $\epsilon(r)<\frac{1}{8}$ and throughout the paper and abbreviate $r$-good as simply good.
\end{definition}

The upshot of the above is that if we manage to prove that for all $f\in
L^{2}(u)$ and $g\in L^{2}(v)$ we have 
\begin{equation*}
\mathbf{E}(|\langle T(uf_{good}),g_{good}\rangle _{v}|)\leq C\Vert f\Vert
_{L^{2}(u)}\Vert g\Vert _{L^{2}(v)}
\end{equation*}%
then we obtain \eqref{e:RedtoGood}. The remainder of the paper is devoted to
proving \eqref{e:RedtoGood}. We remind the reader that $\mathbf{E}$ refers
to expectation taken over the product probability space $\Omega \times
\Omega $.  In fact, we will prove that \eqref{e:RedtoGood} holds for all dyadic grids and so that statement regarding the expectation will then follow.

\section{Main Decomposition}
\label{s5}

Fix a cube $Q_{0}$ for the rest of the paper; this is the support of the
functions $f$ and $g$. Our goal is to demonstrate \eqref{e:RedtoGood}. This
will require several additional reductions.



Throughout the rest of this paper we assume the $\mathcal{A}_{2}$ conditions, the testing conditions and the pivotal conditions. 

\subsection{Global to Local Reduction}

We will now try to control the bilinear form:

\begin{eqnarray*}
\langle T(uf),g\rangle _{v} &=&\sum\limits_{Q\in \mathcal{D},S\in \mathcal{G}%
}\sum\limits_{\epsilon =1}^{M_{Q}-1}\sum\limits_{k=1}^{M_{S}-1}\langle
f,h_{Q}^{\epsilon }\rangle _{u}\langle T(uh_{Q}^{\epsilon }),h_{S}^{k}\rangle
_{v}\langle g,h_{S}^{k}\rangle _{v} \\
&=&\sum\limits_{\substack{ Q\in \mathcal{D},S\in \mathcal{G}  \\ l(Q)\geq
l(S) }}\sum\limits_{\epsilon =1}^{M_{Q}-1}\sum\limits_{k=1}^{M_{S}-1}\langle
f,h_{Q}^{\epsilon }\rangle _{u}\langle T(uh_{Q}^{\epsilon }),h_{S}^{k}\rangle
_{v}\langle g,h_{S}^{k}\rangle _{v} \\
&&+\sum\limits_{\substack{ Q\in \mathcal{D},S\in \mathcal{G}  \\ l(S)>l(Q)}}%
\sum\limits_{\epsilon =1}^{M_{Q}-1}\sum\limits_{k=1}^{M_{S}-1}\langle
f,h_{Q}^{\epsilon }\rangle _{u}\langle T(uh_{Q}^{\epsilon }),h_{S}^{k}\rangle
_{v}\langle g,h_{S}^{k}\rangle _{v} \\
&=: &A_{1}^{1}+A_{2}^{1}.
\end{eqnarray*}%
Set $\mathcal{A}_{1}^{1}:=\{(Q,S)\in \mathcal{D}\times \mathcal{G}:l(Q)\geq
l(S)\}$ and $\mathcal{A}_{2}^{1}:=\{(Q,S)\in \mathcal{D}\times \mathcal{G}%
:l(S)>l(Q)\}$ and 
\begin{equation*}
A_{j}^{i}:=\sum\limits_{(Q,S)\in \mathcal{A}_{j}^{i}}\sum\limits_{\epsilon
=1}^{M_{Q}-1}\sum\limits_{k=1}^{M_{S}-1}\langle f,h_{Q}^{\epsilon }\rangle
_{u}\langle T(uh_{Q}^{\epsilon }),h_{S}^{k}\rangle _{v}\langle
g,h_{S}^{k}\rangle _{v}.
\end{equation*}%
Here $A_{2}^{1}$ is complementary to the sum $A_{1}^{1}$. The sums are
estimated symmetrically. Hence it is enough to prove \eqref{e:RedtoGood} for 
$A_{1}^{1}$. We will further decompose this term into a number of other
bilinear forms $A_{j}^{i}$. The superscript $i$ denotes the generation and
subscript $j$ counts the number of decompositions. To help understand these
decompositions we can look at the flow chart where all the terms are listed.

\begin{center}    
\begin{tikzpicture}
\node[draw,rounded corners] (1) at (0,16){$\langle T(uf), g\rangle_v$};
\node[draw,rounded corners] (2) at (3,15){$A_1^1$};
\node[draw,rounded corners] (3) at (-3,15){$A_1^2$};
\node[draw,rounded corners] (4) at (0,13){$A_3^2$};
\node[draw,rounded corners] (5) at (3,13){$A_2^1$};
\node[draw,rounded corners] (6) at (6,13){$A_2^2$};
\node[draw,rounded corners] (7) at (-2,11){$A_2^3$};
\node[draw,rounded corners] (8) at (2,11){$A_1^3$};
\node[draw,rounded corners] (9) at (-4,9){$A_1^4$};
\node[draw,rounded corners] (10) at (-2,9){$A_2^4$};
\node[draw,rounded corners] (11) at (0,9){$A_3^4$};
\node[draw,rounded corners] (12) at (-4,7){$A_1^5$};
\node[draw,rounded corners] (13) at (0,7){$A_2^5$};
\node[draw,rounded corners] (14) at (-2,5){$A_1^6$};
\node[draw,rounded corners] (15) at (0,5){$A_2^6$};
\node[draw,rounded corners] (16) at (2,5){$A_3^6$};
\node[draw,rounded corners] (17) at (-0.5,3){$A_1^7$};
\node[draw,rounded corners] (18) at (2,3){$A_2^7$};
\node[draw,rounded corners] (19) at (-4,3){$A_3^7$};
\node[draw,rounded corners] (20) at (-2,3){$A_4^7$};
\draw (1) to (2);
\draw (1) to (3);
\draw (2) to (5);
\node at (3.65,14) {Surgery};
\draw (2) to (4);
\draw (2) to (6);
\node at (5.55,14){$\mathcal{A}_2$};
\draw (4) to (7);
\draw (4) to (8);
\node at (1.55,12){$\mathcal{A}_2$};
\draw (7) to (9);
\node at (-3.55,10){$\mathcal{A}_2$};
\draw (7) to (10);
\draw (7) to (11);
\node at (0.1,10){$\mathcal{V}$};
\draw (10) to (12);
\node at (-3.95,8){$\mathcal{T}+ \mathcal{V}$};
\draw (10) to (13);
\node at (0.3,8){Paraproducts};
\draw (13) to (14);
\draw (13) to (15);
\draw (13) to (16);
\node at (2.1,6){$\mathcal{V}$};
\draw (14) to (19);
\node at (-2.5,4){$\mathcal{V}$};
\draw (14) to (20);
\node at (-3.75,4){$\mathcal{V}$};
\draw (15) to (17);
\node at (-1.1,4){$\mathcal{T}+\mathcal{V}$};
\draw (15) to (18);
\node at (2.1,4){$\mathcal{V}$};
\end{tikzpicture}
\end{center}

\begin{enumerate}
\item The flow chart starts at the bilinear form in \eqref{e:RedtoGood}. 
\vspace{0.25 cm}

\item The hypothesis we used in controlling each bilinear form $\mathcal{A}%
_{2},\mathcal{T}$, surgery and/or $\mathcal{V}$ is written on the
edges of the chart. \vspace{0.25cm}

\item To control the terms $A_3^4,A_3^6,A_3^7$ and $A_4^7$ we use the
stopping cube arguments given in Section 6. \vspace{0.25 cm}

\item The edge leading to $A_{2}^{5}$ has been labelled paraproduct as all
the estimates below that use paraproduct arguments to control them.
\end{enumerate}

Now let us begin by proving the estimate in \eqref{e:RedtoGood} for the term 
$A_{1}^{1}$. We decompose $\mathcal{A}_{1}^{1}$ into the following sets.
Denote by 
\begin{equation*}
\mathcal{A}_{1}^{2}:=\{(Q,S)\in \mathcal{A}_{1}^{1}:\delta ^{r}l(Q)\leq
l(S)\leq l(Q),\func{dist}(Q,S)\leq l(Q)\};
\end{equation*}%
\begin{equation*}
\mathcal{A}_{2}^{2}:=\{(Q,S)\in \mathcal{A}_{1}^{1}:l(S)\leq l(Q),\func{dist}%
(Q,S)\geq l(Q)\};
\end{equation*}%
\begin{equation*}
\mathcal{A}_{3}^{2}:=\{(Q,S)\in \mathcal{A}_{1}^{1}:l(S)\leq \delta ^{r}l(Q),%
\func{dist}(Q,S)\leq l(Q)\}.
\end{equation*}%
We will show the following estimates below: 
\begin{eqnarray*}
|A_{1}^{2}| &\lesssim &\left( C_{\tau }\sqrt{ \mathcal{A}_{2}}+%
\mathcal{T}+\tau ^{\frac{%
\eta }{2}}\mathcal{N}\right) \Vert f\Vert _{L^{2}(u)}\Vert
g\Vert _{L^{2}(v)} \\
|A_{2}^{2}| &\lesssim &\mathcal{A}_{2}\Vert f\Vert _{L^{2}(u)}\Vert g\Vert
_{L^{2}(v)}.
\end{eqnarray*}

The term $A_{1}^{2}$ is easily controlled using a `weak boundedness
property', which we recall (even though it will not be needed):
\begin{definition}
Weak Boundedness Condition: For a constant $C_{WBP}>1$, we have $C_{WBP}$ as the
best constant in the inequality 
\begin{equation*}
\left\vert \int_{S}T(u1_{Q})dv\right\vert \leq C_{WBP}u(Q)^{\frac{1}{2}%
}v(S)^{\frac{1}{2}},
\end{equation*}%
where $Q,S$ are cubes such that $\delta ^{r}l(Q)\leq l(S)\leq l(Q)$ and $%
\func{dist}(Q,S)\leq l(Q)$.
\end{definition}

We can avoid the weak boundedness property if we instead use `surgery' to control the \emph{average over grids} in terms of only the testing and $\mathcal{A}_2$ conditions, and a small multiple of the operator norm.  It is clear that the lemma below provides control on the term $A_1^2$ as desired with a small multiple of the norm that can be absorbed by choosing the parameter $\tau$ appropriately small.

\begin{lemma}
\label{l:4.1} The following estimate holds: 
\begin{eqnarray*}
\left\vert \mathbf{E}_{\mathcal{D}\in \Omega }\sum\limits_{(Q,S)\in 
\mathcal{A}_{1}^{2}}\sum\limits_{\epsilon
=1}^{M_{Q}-1}\sum\limits_{k=1}^{M_{S}-1}\langle f,h_{Q}^{\epsilon }\rangle
_{u}\langle Th_{Q}^{\epsilon },h_{S}^{k}\rangle _{v}\langle
g,h_{S}^{k}\rangle _{v}\right\vert 
& \lesssim &\left( C_{\tau }\sqrt{\mathcal{A}_{2}}+\mathcal{T}+\tau ^{\frac{\eta }{2}}%
\mathcal{N}\right) \Vert f\Vert _{L^{2}(u)}\Vert g\Vert
_{L^{2}(v)}.
\end{eqnarray*}
\end{lemma}

 For the proof of Lemma \ref{l:4.1} we recall
the surgery estimate (\ref{surgery}) for the random dyadic systems
constructed in \cite{HyMa} and \cite{TA} with parameter $\delta >0$. There
are positive constants $C_3,\eta >0$ such that for every $x\in X,\tau >0$
and $k\in \mathbb{Z}$, 
\begin{equation}
\mathbb{P}\left( \left\{ \omega \in \Omega :x\in \bigcup\limits_{\alpha
}\partial _{\tau \delta ^{k}}Q_{\alpha }^{k}\left( \omega \right) \right\}
\right) \leq C_3\tau ^{\eta }\ .  \label{surgery'}
\end{equation}%
We can now prove an extension to spaces of homogeneous type of the surgery
lemma of Lacey and Wick in \cite[Lemma 8.5]{LaWi}, by repeating their
argument with obvious modifications. In the estimate of Lemma \ref{l:4.1},
the explicit Haar functions are used in the decompositions, whereas in the
surgery lemma it is convenient to use instead the associated Haar
projections 
\begin{equation*}
\bigtriangleup _{Q}^{u}f=\sum\limits_{\epsilon =1}^{M_{Q}-1}\langle
f,h_{Q}^{\epsilon }\rangle _{u}h_{Q}^{\epsilon }\text{ and }\bigtriangleup
_{S}^{v}g=\sum\limits_{k=1}^{M_{S}-1}\langle g,h_{S}^{k}\rangle
_{v}h_{S}^{k}\ .
\end{equation*}%
We say that two cubes $Q$ and $S$ are $\rho$-close if $\delta ^{\rho}\leq \frac{l \left( Q\right) }{l \left( S\right) }\leq
\delta ^{-\rho}$ and $d\left( Q,S\right) \leq \max\{l \left(
Q\right), l \left(
S\right)\} $.

\begin{lemma}[Surgery Lemma]
\label{l:surgery}
For $0<\tau <1$ and sufficiently large $r$ we have%
\begin{equation}
\mathbf{E}_{\mathcal{D}\in \Omega }\sum_{\substack{ \left( Q,S\right) \in 
\mathcal{D}_{\limfunc{good}}\times \mathcal{G}_{\limfunc{good}}  \\ Q\text{
and }S\text{ are }\rho\text{-close}}}\left\vert \left\langle
T\left(u \bigtriangleup _{Q}^{u }f\right) ,\bigtriangleup
_{S}^{v }g\right\rangle _{v }\right\vert  \label{more restricted}
\lesssim\left( C_{\tau }\sqrt{\mathcal{A}_{2}}+\mathcal{T}+\tau ^{\frac{\eta }{2}}%
\mathcal{N}_{T^{\alpha }}\right) \left\Vert f\right\Vert _{L^{2}\left(
u \right) }\left\Vert g\right\Vert _{L^{2}\left( v \right) }\ . 
\end{equation}
\end{lemma}
Note that we can choose $\rho$ to be $r$ so Lemma \ref{l:4.1} follows immediately from Lemma \ref{l:surgery} since we are summing on a larger collection of cubes and we have pulled the absolute value inside the sum.

\begin{proof}
In order to prove \eqref{more restricted}, we invoke the surgery estimate using \eqref{surgery'}. Given $0<\lambda <\frac{1}{2}$, define 
\begin{equation*}
S_{\lambda }:= \left\{ x\in S:\limfunc{dist}\left( x,\partial S\right)
>\lambda l \left( S\right) \right\} .
\end{equation*}%
Then we write%
\begin{eqnarray*}
\left\langle T\left(u \bigtriangleup _{Q}^{u
}f\right) ,\bigtriangleup _{S}^{v }g\right\rangle _{v }
&=&\left\langle T\left( u\sum_{Q^{\prime }\in \mathcal{H}\left( Q\right) }\mathbf{1}_{Q^{\prime }}\bigtriangleup
_{Q}^{u }f\right) ,\sum_{S^{\prime }\in \mathcal{H}\left( J\right) }\mathbf{1}_{S^{\prime }}\bigtriangleup _{S}^{v
}g\right\rangle _{v } \\
&=&\sum_{Q^{\prime }\in \mathcal{H}\left( Q\right) }\sum_{S^{\prime }\in 
\mathcal{H}\left( S\right) }\left(\mathbf{E}_{Q^{\prime }}^{u }\bigtriangleup
_{Q}^{u }f\right) \ \left\langle T(u \mathbf{1}%
_{Q^{\prime }}),\mathbf{1}_{S^{\prime }}\right\rangle _{v }\ \left(
\mathbf{E}_{S^{\prime }}^{v }\bigtriangleup _{S}^{v }g\right) ,
\end{eqnarray*}%
and so%
\begin{eqnarray*}
&&\sum_{\substack{ \left( Q,S\right) \in \mathcal{D}_{\limfunc{good}}\times 
\mathcal{G}_{\limfunc{good}}  \\ Q\text{ and }S\text{ are }\rho%
\text{-close}}}\left\vert \left\langle T\left(u \bigtriangleup
_{Q}^{u }f\right) ,\bigtriangleup _{S}^{v }g\right\rangle _{v
}\right\vert \\
&\leq &\sum_{\substack{ \left( Q,S\right) \in \mathcal{D}_{%
\limfunc{good}}\times \mathcal{G}_{\limfunc{good}}  \\ Q\text{ and }S\text{
are }\rho\text{-close}}}\sum_{Q^{\prime }\in \mathcal{H}\left(
S\right) }\sum_{S^{\prime }\in \mathcal{H}\left( S\right) }\left\vert
\left( \mathbf{E}_{Q^{\prime }}^{u }\bigtriangleup _{Q}^{u }f\right) \
\left\langle T(u\mathbf{1}_{Q^{\prime }}),\mathbf{1}%
_{S^{\prime }\cap Q^{\prime }}\right\rangle _{v }\ \left( \mathbf{E}_{S^{\prime
}}^{v }\bigtriangleup _{S}^{v }g\right) \right\vert \\
&&+\sum_{\substack{ \left( Q,S\right) \in \mathcal{D}_{\limfunc{good}}\times 
\mathcal{G}_{\limfunc{good}}  \\ Q\text{ and }S\text{ are }\rho%
\text{-close}}}\sum_{Q^{\prime }\in \mathcal{H}\left( Q\right)
}\sum_{S^{\prime }\in \mathcal{H}\left( S\right) }\left\vert \left(
\mathbf{E}_{Q^{\prime }}^{u }\bigtriangleup _{Q}^{u }f\right) \
\left\langle T(u\mathbf{1}_{Q^{\prime }}),\mathbf{1}%
_{S^{\prime }\setminus Q^{\prime }}\right\rangle _{v }\ \left(
\mathbf{E}_{S^{\prime }}^{v }\bigtriangleup _{S}^{v }g\right) \right\vert \\
&=: &\operatorname{Term}_1+\operatorname{Term}_{2}\ .
\end{eqnarray*}%
Now for convenience of notation and to shorten some displays  we write 
\begin{equation*}
\sum^{\maltese }:= \sum_{\substack{ \left( Q,S\right) \in \mathcal{D}_{%
\limfunc{good}}\times \mathcal{G}_{\limfunc{good}}  \\ Q\text{ and }S\text{
are }\rho\text{-close}}}\sum_{Q^{\prime }\in \mathcal{H}\left(
Q\right) }\sum_{S^{\prime }\in \mathcal{H}\left( S\right) },
\end{equation*}%
then the inequality 
\begin{equation*}
\left\vert \left\langle T(u\mathbf{1}_{Q^{\prime }}),%
\mathbf{1}_{S^{\prime }\cap Q^{\prime }}\right\rangle _{v }\right\vert
\leq \sqrt{\int_{Q^{\prime }}\left\vert T(u\mathbf{1}%
_{Q^{\prime }})\right\vert ^{2}dv }\sqrt{v(S^{\prime }\cap
Q^{\prime })}\leq \mathcal{T}\sqrt{u(Q^{\prime })}\sqrt{v(S^{\prime })}
\end{equation*}%
shows that%
\begin{eqnarray*}
\operatorname{Term}_1 &\leq &\mathcal{T}\sum^{\maltese }\left( \left\vert \mathbf{E}_{Q^{\prime
}}^{u }\bigtriangleup _{Q}^{u }f\right\vert \sqrt{u \left(
Q^{\prime }\right)}\right) \left( \left\vert \mathbf{E}_{S^{\prime }}^{v }\bigtriangleup
_{S}^{v }g\right\vert \sqrt{v\left( S^{\prime
}\right)}\right) \\
&\leq &\mathcal{T}\sqrt{\sum^{\maltese }\left\vert \mathbf{E}_{Q^{\prime
}}^{u }\bigtriangleup _{Q}^{u }f\right\vert ^{2}u\left(
Q^{\prime }\right)}\sqrt{\sum^{\maltese }\left\vert
\mathbf{E}_{S^{\prime }}^{v }\bigtriangleup _{S}^{v }g\right\vert
^{2}v\left( S^{\prime }\right)}\lesssim \mathcal{T}
\left\Vert f\right\Vert _{L^{2}\left( u \right) }\left\Vert
g\right\Vert _{L^{2}\left( v \right) }.
\end{eqnarray*}

To control $\operatorname{Term}_2$ we further decompose $S^{\prime }\setminus Q^{\prime }$
into small and large parts,%
\begin{equation*}
S^{\prime }\setminus Q^{\prime }:= \left\{ \left( S^{\prime }\setminus
Q^{\prime }\right) \cap \partial _{\lambda }Q^{\prime }\right\} \overset{%
\cdot }{\bigcup }\left\{ \left( S^{\prime }\setminus Q^{\prime }\right)
\setminus \partial _{\lambda }Q^{\prime }\right\} := E\,\overset{\cdot }{%
\bigcup }\,F,
\end{equation*}%
and estimate $\operatorname{Term}_2$ accordingly,%
\begin{equation*}
\operatorname{Term}_2\leq \sum^{\maltese }\left\vert \left( \mathbf{E}_{Q^{\prime }}^{u
}\bigtriangleup _{Q}^{u }f\right) \ \left\langle T(u%
\mathbf{1}_{Q^{\prime }}),\mathbf{1}_{E}\right\rangle _{v }\ \left(
\mathbf{E}_{S^{\prime }}^{v }\bigtriangleup _{S}^{v }g\right) \right\vert
+\sum^{\maltese }\left\vert \left( \mathbf{E}_{Q^{\prime }}^{u }\bigtriangleup
_{Q}^{u }f\right) \ \left\langle T(u\mathbf{1}%
_{Q^{\prime }}),\mathbf{1}_{F}\right\rangle _{v }\ \left( \mathbf{E}_{S^{\prime
}}^{v }\bigtriangleup _{S}^{v }g\right) \right\vert :=
\operatorname{Term}_{21}+\operatorname{Term}_{22}.
\end{equation*}%
Now $F=\left( S^{\prime }\setminus Q^{\prime }\right) \setminus \partial
_{\lambda }Q^{\prime }$ is contained in $S^{\prime }$ and has distance at
least $cl \left( Q\right) $ from the cube $Q^{\prime }$, so we can
control $\operatorname{Term}_{22}$ by the $\mathcal{A}_2$ condition using Cauchy-Schwarz as above,%
\begin{equation*}
\operatorname{Term}_{22}\leq \sum^{\maltese }\left\vert \mathbf{E}_{Q^{\prime }}^{u }\bigtriangleup
_{Q}^{u }f\right\vert \left( \mathcal{A}_{2}\sqrt{u\left( Q^{\prime
}\right)}\sqrt{v\left( S^{\prime }\right)}%
\right) \left\vert \mathbf{E}_{S^{\prime }}^{v }\bigtriangleup _{S}^{v
}g\right\vert \lesssim \mathcal{A}_{2}\left\Vert f\right\Vert _{L^{2}\left(
u \right) }\left\Vert g\right\Vert _{L^{2}\left( v \right)}.
\end{equation*}

Finally, we use the operator norm to control the \emph{average} of $B_{1}$ by%
\begin{eqnarray*}
\mathbf{E}_{\mathcal{D}\in \Omega }\operatorname{Term}_{21} &\leq &\mathbf{E}_{\mathcal{D}\in
\Omega }\sum^{\maltese }\left\vert \mathbf{E}_{Q^{\prime }}^{u }\bigtriangleup
_{Q}^{u }f\right\vert \ \left( \mathcal{N}\sqrt{u\left(
Q^{\prime }\right)}\sqrt{v\left( \left( S^{\prime
}\setminus Q^{\prime }\right) \cap \partial _{\lambda }Q^{\prime
}\right)}\right) \ \left\vert \mathbf{E}_{S^{\prime }}^{v
}\bigtriangleup _{S}^{v }g\right\vert \\
&\lesssim &\mathcal{N}\mathbf{E}_{\mathcal{D}\in \Omega }\left\Vert
f\right\Vert _{L^{2}\left( u \right) }\sqrt{\sum^{\maltese }\left\vert
\bigtriangleup _{S}^{v }g\right\vert ^{2}v\left( \left( J^{\prime
}\setminus Q^{\prime }\right) \cap \partial _{\lambda }Q^{\prime
}\right)} \\
&\lesssim & \mathcal{N}\left\Vert f\right\Vert _{L^{2}\left( u \right)
}\sqrt{\sum_{S\in \mathcal{G}_{\limfunc{good}}}\left\vert \bigtriangleup
_{S}^{v }g\right\vert ^{2}\mathbf{E}_{\mathcal{D}\in \Omega }\left(
\sum _{\substack{ Q\in \mathcal{D}_{\limfunc{good}},\ Q^{\prime }\in 
\mathcal{H}\left( Q\right)  \\ Q\text{ and }S\text{ are }\rho%
\text{-close}}}v\left( S\cap \partial _{\lambda }Q^{\prime }\right)
\right) }.
\end{eqnarray*}%
Now we use \eqref{surgery'} to obtain%
\begin{equation*}
\mathbf{E}_{\mathcal{D}\in \Omega }\left( \sum_{\substack{ Q\in \mathcal{D}_{%
\limfunc{good}},\ Q^{\prime }\in \mathcal{H}\left( Q\right)  \\ Q\text{ and 
}S\text{ are }\rho\text{-close}}}v \left ( S\cap \partial _{\tau
}Q^{\prime }\right)\right) \lesssim \tau ^{\eta }v\left(
S\right),
\end{equation*}%
which altogether gives%
\begin{equation*}
\mathbf{E}_{\mathcal{D}\in \Omega }\operatorname{Term}_{21}\lesssim \mathcal{N}\left\Vert
f\right\Vert _{L^{2}\left( u \right) }\sqrt{\sum_{S\in \mathcal{G}_{%
\limfunc{good}}}\left\vert \bigtriangleup _{S}^{v }g\right\vert
^{2}\tau^{\eta} v\left(S\right)}\lesssim \tau ^{\frac{\eta }{2}}%
\mathcal{N}\left\Vert f\right\Vert _{L^{2}\left( u \right)
}\left\Vert g\right\Vert _{L^{2}\left( v \right) }.
\end{equation*}
The proof of Lemma \ref{l:surgery} is complete.
\end{proof}

The next lemma controls $A_{2}^{2}$.

\begin{lemma}
\label{l:4.2} The following estimate holds: 
\begin{equation*}
\left\vert \sum\limits_{(Q,S)\in \mathcal{A}_{2}^{2}}\sum\limits_{\epsilon
=1}^{M_{Q}-1}\sum\limits_{k=1}^{M_{S}-1}\langle f,h_{Q}^{\epsilon }\rangle
_{u}\langle T(uh_{Q}^{\epsilon }),h_{S}^{k}\rangle _{v}\langle
g,h_{S}^{k}\rangle _{v}\right\vert \lesssim \mathcal{A}_{2}\Vert f\Vert
_{L^{2}(u)}\Vert g\Vert _{L^{2}(v)}.
\end{equation*}
\end{lemma}

Before we proceed to prove Lemma \ref{l:4.2}, let us collect a couple of
auxiliary lemmas.

\begin{lemma}
\label{l:4.3} Let $S\subset Q^{\prime }\subset \hat{Q}$ be three cubes with $%
\func{dist}(\partial Q^{\prime },S)\geq l(S)$. Let $H_{S}$ be a function
supported on $S$ and with $v$ integral zero. Then we have 
\begin{equation}
\label{e:4.4}
|\langle T(u1_{\hat{Q}\setminus Q^{\prime }}),H_{S}\rangle _{v}|\lesssim \Vert
H_{S}\Vert _{L^{2}(v)}\Phi (S,1_{\hat{Q}\setminus Q^{\prime}}u)^{\frac{1}{2}}.
\end{equation}%
Here $\Phi (S,1_{\hat{Q}\setminus Q^{\prime }}u):=v(S)K\left( S,1_{\hat{Q}%
\setminus Q^{\prime }}u\right) ^{2}$ where 
\begin{equation*}
K(S,1_{\hat{Q}\setminus Q^{\prime }}u):=\int_{\hat{Q}\setminus Q^{\prime
}}\left( \frac{l(S)}{l(S)+\func{dist}(y,S)}\right) ^{\kappa }\frac{1}{\mu
(B(x_{S},l(S)+\func{dist}(y,S)))}du(y),
\end{equation*}
and $\kappa$ is as in Definition \ref{def 1}.
\end{lemma}

The $L^{2}$ formulation of \eqref{e:4.4} proves useful in many estimates
below, in particular in several Carleson Embedding Theorem estimates,
Theorem \ref{t:estimate}. We will apply \eqref{e:4.4} in the dual formulation. Namely, we
have 
\begin{equation}
\Vert T(u1_{\hat{Q}\setminus Q^{\prime }})-\mathbf{E}_{S}^{v}T(1_{\hat{Q}%
\setminus Q^{\prime }}u)\Vert _{L^{2}(S,v)}\lesssim \Phi (S,1_{\hat{Q}\setminus
Q^{\prime}} u)^{\frac{1}{2}}.
\label{e:l2form}
\end{equation}

\begin{proof}
This proof uses the standard computation in the Calder\'on--Zygmund theory.
We use cancellation of a function to pull additional information onto the
kernel of the operator. Using Fubini and the fact that $H_{S}$ has $v$
integral zero we get: 
\begin{eqnarray*}
|\langle T(u 1_{\hat{Q}\setminus Q^{\prime }}),H_{S}\rangle_{v}| & = &
\left|\int_S H_S(x) T(u 1_{\hat{Q}\setminus Q^{\prime }}(x))dv(x)\right| \\
& = & \left|\int_{S}H_{S}(x)\int_{\hat{Q}\setminus Q^{\prime }}\mathfrak K(x,y)du(y)
dv(x)\right| \\
& = & \left|\int_{\hat{Q}\setminus Q^{\prime }}
\int_{S}\left(\mathfrak K(x,y)-\mathfrak K(x_S,y)\right)H_{S}(x)dv(x)du(y)\right| \\
& \leq & \int_{\hat{Q}\setminus Q^{\prime }}\int_{S}\left(\frac{\func{dist}%
(x,x_S)}{\func{dist}(y,x_S)}\right)^{\kappa}\frac{1}{\mu(B(x_S,\func{dist}%
(x_S,y)))}|H_{S}(x)|dv(x)du(y) \\
& \leq & \int_{\hat{Q}\setminus Q^{\prime }}\left(\frac{l(S)}{\func{dist}%
(x_S,y)}\right)^{\kappa}\frac{1}{\mu(B(x_S,\func{dist}(x_S,y)))}%
\int_{S}|H_{S}(x)|dv(x)du(y).
\end{eqnarray*}

Here $x_S\in S$, is the center of $S$ and $y\in \hat{Q}\setminus
Q^{\prime }$ and so $\func{dist}(x_S,y) \approx \func{dist}(y,S) + l(S)$. By
the doubling property of the measure $\mu$ we have $\mu(B(x_S,\func{dist}%
(x_S,y)))\approx \mu(B(x_S,l(S) + \func{dist}(y,S)))$. These estimates can
then be used to give: 
\begin{eqnarray*}
&&|\langle T(u1_{\hat{Q}\setminus Q^{\prime }}, H_S\rangle_{v}| \\
& \lesssim & \int_{\hat{Q}\setminus Q^{\prime }}\left(\frac{l(S)}{l(S)+\func{dist%
}(y,S)}\right)^{\kappa} \frac{1}{\mu(B(x_S,l(S) + \func{dist}(y,S)))}%
du(y)\|H_S\|_{L^2(v)}v(S)^\frac{1}{2} \\
& = &\|H_S\|_{L^2(v)}\Phi(S,u1_{\hat{Q}\setminus Q^{\prime}})^\frac{1}{2}
\end{eqnarray*}
completing the proof.
\end{proof}

The next Lemma is an extension to spaces of homogenous type of the Poisson inequality in \cite{Vol}.  This lemma plays a crucial role in obtaining geometric decay from goodness in controlling $A_3^{1}$ appearing below, and the neighbor and the stopping form appearing later in Section \ref{s6}.

\begin{lemma}
Let $0<\lambda<\frac{\kappa}{n+\kappa}$.
\label{l:4.4} If $S\subset Q\subset \hat{Q}$ and $\func{dist}(S,e(Q))\geq 
\frac{1}{2}l(S)^{\lambda }l(Q)^{1-\lambda }$ where $e(S):=\partial {S}\cup \{(%
\text{center of }S)\}$ then
\end{lemma}

\begin{equation*}
l(S)^{\sigma _{0}}K\left( S,1_{\hat{Q}\setminus Q}u\right) \leq l(Q)^{\sigma
_{0}}K\left( Q,1_{\hat{Q}\setminus Q}u\right) .
\end{equation*}%
Here $\sigma _{0}:=\lambda (n+\kappa )-\kappa $ with $\kappa$ as in Definition \ref{def 1}.

\begin{proof}
To begin with, recall that for each $Q\in \mathcal{D}$, the containment in %
\eqref{eq:contain} holds and the outer ball that contains $S$ is denoted by $%
B(S)$.

By decomposing the space $X$ into annuli based on $B(S)$, we have that 
\begin{align*}
& K\left( S,1_{\hat{Q}\setminus Q}u\right) \\
& =\int_{B(S)}\left( \frac{l(S)}{l(S)+\func{dist}(y,S)}\right) ^{\kappa }%
\frac{1}{\mu (B(x_{S},l(S)+\func{dist}(y,S)))}1_{\hat{Q}\setminus Q}(y)du(y)
\\
& \quad +\sum_{k=1}^{\infty }\int_{\delta ^{-k}B(S)\backslash \delta
^{1-k}B(S)}\left( \frac{l(S)}{l(S)+\func{dist}(y,S)}\right) ^{\kappa }\frac{1%
}{\mu (B(x_{S},l(S)+\func{dist}(y,S)))}1_{\hat{Q}\setminus Q}(y)du(y) \\
& =\sum_{k=\kappa _{0}}^{\kappa _{1}}\int_{\delta ^{-k}B(S)\backslash \delta
^{1-k}B(S)}\left( \frac{l(S)}{l(S)+\func{dist}(y,S)}\right) ^{\kappa }\frac{1%
}{\mu (B(x_{S},l(S)+\func{dist}(y,S)))}1_{\hat{Q}\setminus Q}(y)du(y).
\end{align*}%
Here $\kappa_0$ and $\kappa_1$ are determined by the conditions that for $k<\kappa_0$
\begin{equation*}
\delta ^{-k}B(S)\cap (\hat{Q}\setminus Q)=\emptyset,
\end{equation*}%
and for $k>\kappa _{1}$, we have 
\begin{equation*}
\delta ^{1-k}B(S)\cap (\hat{Q}\setminus Q)=\emptyset .
\end{equation*}%
Hence, for $\kappa_0\leq k\leq \kappa _{1}$, we have 
\begin{equation*}
\frac{1}{2}l(S)^{\lambda }l(Q)^{1-\lambda }\leq \func{dist}(S,e(Q))\lesssim
\delta ^{-k}l(S),
\end{equation*}%
and thus 
\begin{equation*}
\delta ^{k}\lesssim \left( {\frac{l(S)}{l(Q)}}\right) ^{1-\lambda }.
\end{equation*}

We now estimate $ K\left( S,1_{\hat{Q}\setminus Q}u\right)$. To begin with, we give a partition between $\kappa _{0}$ and $\kappa _{1}$ to create a new collection of integers $\tilde{\kappa}_{0},%
\tilde{\kappa}_{1},\ldots ,\tilde{\kappa}_{L}$ such that $\tilde{\kappa}%
_{0}=\kappa _{0}$, $\tilde{\kappa}_{L}=\kappa _{1}$ and that $\tilde{\kappa}%
_{1}$ satisfies $\delta ^{-\tilde{\kappa}_{1}}l(S)\approx l(Q)$, $\tilde{%
\kappa}_{2}$ satisfies $\delta ^{-\tilde{\kappa}_{2}}l(S)\approx \delta l(Q)$, $%
\ldots $ , $\tilde{\kappa}_{L}$ satisfies $\delta ^{-\tilde{\kappa}%
_{L}}l(S)\approx \delta ^{-L}l(Q)$. In fact, we have $\tilde{\kappa}%
_{0}=\kappa _{0}$, $\tilde{\kappa}_{1}>\tilde{\kappa}_{0}$, $\tilde{\kappa}%
_{2}=\tilde{\kappa}_{1}+1$, and so on. Then we get that
$$ \sum_{k=\kappa _{0}}^{\kappa _{1}}= \sum_{\ell =1}^{L}\sum_{k=\tilde{\kappa}_{\ell -1}}^{\tilde{\kappa%
}_{\ell }}.$$
Hence, 
\begin{align*}
& K\left( S,1_{\hat{Q}\setminus Q}u\right) \\
& \lesssim \sum_{\ell =1}^{L}\sum_{k=\tilde{\kappa}_{\ell -1}}^{\tilde{\kappa%
}_{\ell }}\int_{\delta ^{-k}B(S)\backslash \delta ^{1-k}B(S)}\left( \frac{%
l(S)}{l(S)+\func{dist}(y,S)}\right) ^{\kappa }\frac{1}{\mu (B(x_{S},l(S)+%
\func{dist}(y,S)))}1_{\hat{Q}\setminus Q}(y)du(y) \\
& \lesssim \sum_{\ell =1}^{L}\sum_{k=\tilde{\kappa}_{\ell -1}}^{\tilde{\kappa%
}_{\ell }}\int_{(\delta ^{-k}B(S)\backslash \delta ^{1-k}B(S))\cap (\hat{Q}%
\setminus Q)}\left( \frac{l(S)}{\delta ^{1-k}l(S)}\right) ^{\kappa }\frac{1}{%
\mu (B(x_{S},\delta ^{1-k}l(S)))}du(y) \\
& \lesssim \sum_{k=\tilde{\kappa}_{0}}^{\tilde{\kappa}_{1}}\delta ^{\kappa
k}\int_{(\delta ^{-k}B(S)\backslash \delta ^{1-k}B(S))\cap (\hat{Q}\setminus
Q)}\frac{1}{\mu (B(x_{S},\delta ^{1-k}l(S)))}du(y) \\
& \quad +\sum_{\ell =2}^{L}\delta ^{\kappa (\tilde{\kappa}_{1}+\ell
)}\int_{(\delta ^{-(\tilde{\kappa}_{1}+\ell )}B(S)\backslash \delta ^{(1-%
\tilde{\kappa}_{1}-\ell )}B(S))\cap (\hat{Q}\setminus Q)}\frac{1}{\mu
(B(x_{S},\delta ^{1-\tilde{\kappa}_{1}-\ell }l(S)))}du(y) \\
& =:\operatorname{Term}_{1}+\operatorname{Term}_{2}.
\end{align*}%
For the first term, by using the doubling property, we have that 
\begin{align*}
\operatorname{Term}_{1}& \lesssim \sum_{k=\tilde{\kappa}_{0}}^{\tilde{\kappa}_{1}}\delta
^{\kappa k}\frac{\mu (B(x_{S},l(Q)))}{\mu (B(x_{S},\delta ^{1-k}l(S)))}\frac{%
1}{\mu (B(x_{S},l(Q)))}u\left[ (\delta ^{-k}B(S)\backslash \delta
^{1-k}B(S))\cap (\hat{Q}\setminus Q)\right] \\
& \lesssim \sum_{k=\tilde{\kappa}_{0}}^{\tilde{\kappa}_{1}}\delta ^{\kappa
k}\left( \frac{l(Q)}{\delta ^{-k}l(S)}\right) ^{n}\frac{1}{\mu
(B(x_{Q},l(Q)))}u\left[ (\delta ^{-k}B(S)\backslash \delta ^{1-k}B(S))\cap (%
\hat{Q}\setminus Q)\right] \\
& \lesssim \sum_{k=\tilde{\kappa}_{0}}^{\tilde{\kappa}_{1}}\delta ^{\kappa
k}\delta ^{kn}\left( \frac{l(Q)}{l(S)}\right) ^{n}\frac{1}{\mu
(B(x_{Q},l(Q)))}u\left[ (\delta ^{-k}B(S)\backslash \delta ^{1-k}B(S))\cap (%
\hat{Q}\setminus Q)\right] \\
& \lesssim \sum_{k=\tilde{\kappa}_{0}}^{\tilde{\kappa}_{1}}\delta ^{\kappa
k}\left( \frac{l(Q)}{l(S)}\right) ^{n\lambda }\frac{1}{\mu (B(x_{Q},l(Q)))}u%
\left[ (\delta ^{-k}B(S)\backslash \delta ^{1-k}B(S))\cap (\hat{Q}\setminus
Q)\right] \\
& \lesssim \delta ^{\kappa \tilde{\kappa}_{0}}\left( \frac{l(Q)}{l(S)}%
\right) ^{n\lambda }\frac{1}{\mu (B(x_{Q},l(Q)))}u\left[ B(Q)\cap (\hat{Q}%
\setminus Q)\right] \\
& \lesssim \left( \frac{l(Q)}{l(S)}\right) ^{\lambda (n+\kappa )-\kappa
}\int_{B(Q)}\frac{1}{\mu (B(x_{Q},l(Q)))}1_{\hat{Q}\setminus Q}(y)du(y).
\end{align*}%
For the second term, we get that 
\begin{align*}
\operatorname{Term}_{2}& \lesssim \delta ^{\kappa \tilde{\kappa}_{1}}\sum_{\ell
=2}^{L}\delta ^{\kappa \ell }\frac{1}{\mu (B(x_{S},\delta ^{-\ell }l(Q)))}u%
\left[ (\delta ^{-(\tilde{\kappa}_{1}+\ell )}B(S)\backslash \delta ^{1-%
\tilde{\kappa}_{1}-\ell }B(S))\cap (\hat{Q}\setminus Q)\right] \\
& \lesssim \delta ^{\kappa \tilde{\kappa}_{1}}\sum_{\ell =2}^{L}\delta
^{\kappa \ell }\frac{1}{\mu (B(x_{Q},\delta ^{-\ell }l(Q)))}u\left[ (\delta
^{-\ell }B(Q)\backslash \delta ^{1-\ell }B(Q))\cap (\hat{Q}\setminus Q)%
\right] \\
& \lesssim \delta ^{\kappa \tilde{\kappa}_{1}}\sum_{\ell =2}^{L}\int_{\delta
^{-\ell }B(Q)\backslash \delta ^{1-\ell }B(Q)}\left( \frac{l(Q)}{\delta
^{-\ell }l(Q)}\right) ^{\kappa }\frac{1}{\mu (B(x_{Q},\delta ^{-\ell }l(Q)))}%
1_{\hat{Q}\setminus Q}(y)du(y).
\end{align*}%

Next, we note that 
\begin{align*}
& K\left( Q,1_{\hat{Q}\setminus Q}u\right) \\
& =\int_{B(Q)}\left( \frac{l(Q)}{l(Q)+\func{dist}(y,Q)}\right) ^{\kappa }%
\frac{1}{\mu (B(x_{Q},l(Q)+\func{dist}(y,Q)))}1_{\hat{Q}\setminus Q}(y)du(y)
\\
& \quad +\sum_{j=1}^{j_{1}}\int_{\delta ^{-j}B(Q)\backslash \delta
^{1-j}B(Q)}\left( \frac{l(Q)}{l(Q)+\func{dist}(y,Q)}\right) ^{\kappa }\frac{1%
}{\mu (B(x_{Q},l(Q)+\func{dist}(y,Q)))}1_{\hat{Q}\setminus Q}(y)du(y) \\
& \approx \int_{B(Q)}\frac{1}{\mu (B(x_{Q},l(Q)))}1_{\hat{Q}\setminus
Q}(y)du(y) \\
& \quad +\sum_{j=1}^{j_{1}}\int_{\delta ^{-j}B(Q)\backslash \delta
^{1-j}B(Q)}\left( \frac{l(Q)}{\delta ^{-j}l(Q)}\right) ^{\kappa }\frac{1}{%
\mu (B(x_{Q},\delta ^{-j}l(Q)))}1_{\hat{Q}\setminus Q}(y)du(y).
\end{align*}

Combining the estimates of $\operatorname{Term}_{1}$ and $\operatorname{Term}_{2}$, and the equality above, we see that 
\begin{equation*}
K\left( S,1_{\hat{Q}\setminus Q}u\right) \lesssim \left( \frac{l(Q)}{l(S)}%
\right) ^{\lambda (n+\kappa )-\kappa }K\left( Q,1_{\hat{Q}\setminus
Q}u\right) .
\end{equation*}%
The proof of Lemma \ref{l:4.4} is complete.
\end{proof}

Let us begin with the proof of Lemma \ref{l:4.2}.

\begin{proof}[Proof of Lemma \ref{l:4.2}]
Recall that the pairs of cubes $(Q,S)\in \mathcal{A}_{2}^{2}$ satisfy $%
l(S)\leq l(Q)$ and $\func{dist}(Q,S)\geq l(Q)$. We can apply Lemma \ref%
{l:4.3} to $\langle T(uh_{Q}^{\epsilon }),h_{S}^{k}\rangle _{v}$. To see
this note that $h_{Q}^{\epsilon }$ is constant on each child $Q_{\epsilon }$
where $\epsilon \in \{1,2....,M_{Q-1}\}$.

Take a child $Q_{\epsilon}$ and apply Lemma \ref{l:4.3} with the largest cube $%
\hat{Q}$ taken to be $\hat{Q} := \func{hull}\left[Q_{\epsilon},(\frac{l(Q)}{l(S)}%
)^{1-\lambda}S\right]$. Here $\rho S$ means the cube with same centre as $S$ and
length equal to $\rho l(S)$ and by $\func{hull}$ above we mean $\hat{Q}$ is
the smallest cube containing both cubes $Q_{\epsilon}$ and $\left(\frac{l(Q)}{l(S)}%
\right)^{1-\lambda}S$.

The two cubes $Q_{\epsilon}$ and $\left(\frac{l(Q)}{l(S)}%
\right)^{1-\lambda}S$ are disjoint. We take $Q^{\prime }\subset \hat{Q}$ so
that $\hat{Q}\setminus Q^{\prime }= Q_{\epsilon}$. Then using Lemma \ref{l:4.3}
we get the following estimate 
\begin{equation*}
\beta(Q,S) := \left|\sum_{\epsilon} \langle T(1_{Q_{\epsilon}}
h_Q^{\epsilon} u), h_S^k\rangle_v\right| \leq \sum\limits_{\epsilon}|\mathbf{E}%
_{Q_{\epsilon}}^u (h_Q^{\epsilon})| |\langle T(1_{Q_{\epsilon}} u),
h_S^k\rangle_v|.
\end{equation*}
Here $\mathbf{E}_{Q_{\epsilon}}^u (h_Q^{\epsilon}) := \frac{1}{u(Q_{\epsilon})}%
\int_{Q_{\epsilon}} h_Q^{\epsilon}(x) du(x).$ Observe $|\mathbf{E}%
_{Q_{\epsilon}}^u (h_Q^{\epsilon})| \leq \frac{1}{u(Q_{\epsilon})^{\frac{1}{2}}}$.

\vspace{0.5 cm}

\noindent We have $K(S,1_{Q_{\epsilon}}u) \lesssim \frac{l(S)^{\kappa}}{(l(S) +%
\func{dist}(Q_i,S))^{\kappa + n} } u(Q_{\epsilon})$, where $n$ is the upper dimension of $\mu$ and we have used the doubling property of the measure $\mu$. 
Hence we get 
\begin{equation*}
\beta(Q,S) \lesssim \sum\limits_{\epsilon} v(S)^{1/2} u(Q_{\epsilon})^{1/2} 
\frac{l(S)^{\kappa}}{(l(S) +\func{dist}(Q,S))^{\kappa + n}}.
\end{equation*}

To continue, we may assume that $\Vert f\Vert _{L^{2}(u)}=\Vert g\Vert
_{L^{2}(v)}=1$. We then estimate 
\begin{equation*}
|A_{2}^{2}|=\left\vert \sum\limits_{(Q,S)\in \mathcal{A}_{2}^{2}}\sum%
\limits_{\epsilon =1}^{M_{Q}-1}\sum\limits_{k=1}^{M_{S}-1}\langle
f,h_{Q}^{\epsilon }\rangle _{u}\langle T(uh_{Q}^{\epsilon
}),h_{S}^{k}\rangle _{v}\langle g,h_{S}^{k}\rangle _{v}\right\vert
\end{equation*}%
which can be written as follows
\begin{eqnarray*}
|A_2^2| & \leq & \sum\limits_Q \sum\limits_{\substack{ S:l(S) \leq l(Q)  \\ 
\func{dist}(Q,S)\geq l(Q) }} \sum\limits_{\epsilon =
1}^{M_{Q}-1}\sum\limits_{k = 1}^{M_{S}-1}|\langle f, h_Q^{\epsilon} \rangle_u
|\beta(Q,S) |\langle g, h_S^k\rangle_v| \\
& \lesssim & \sum\limits_Q \sum\limits_{\substack{ S:l(S) \leq l(Q)  \\ 
\func{dist}(Q,S)\geq l(Q) }} \sum\limits_{\epsilon =
1}^{M_{Q}-1}\sum\limits_{k = 1}^{M_{S}-1}|\langle f, h_Q^{\epsilon} \rangle_u
| u(Q)^{1/2} \frac{l(S)^{\kappa}}{(l(S) + \func{dist}(Q,S))^{\kappa + n}}
v(S)^{1/2}|\langle g, h_S^k\rangle_v| \\
& \lesssim & \sum\limits_Q \sum\limits_{\epsilon = 1}^{M_{Q}-1}|\langle f,
h_Q^{\epsilon} \rangle_u |^2 \sum\limits_{\substack{ S:l(S) \leq l(Q)  \\ 
\func{dist}(Q,S)\geq l(Q) }}\left(\frac{l(S)}{l(Q)}\right)^{-\sigma}
u(Q)^{1/2} \frac{l(S)^{\kappa}}{(l(S) + \func{dist}(Q,S))^{\kappa + n}}
v(S)^{1/2} \\
& & + \sum\limits_S \sum\limits_{k = 1}^{M_{S}-1}|\langle g, h_S^k \rangle_v
|^2 \sum\limits_{\substack{ Q:l(S) \leq l(Q)  \\ \func{dist}(Q,S)\geq l(Q) }}%
\left(\frac{l(S)}{l(Q)}\right)^{\sigma} u(Q)^{1/2} \frac{l(S)^{\kappa}}{%
(l(S) + \func{dist}(Q,S))^{\kappa + n}} v(S)^{1/2}\\
&=:& A_{21}^2+A_{22}^2,
\end{eqnarray*}
\noindent where in the last inequality we have inserted the gain and loss term $\left(\frac{l(S)}{l(Q)}%
\right)^{\pm \sigma}$ with $0 < \sigma < 1$. 

We first consider the term $A_{22}^2$. For each fixed $Q$ we have 
\begin{align*}
&\sum\limits_{\substack{ S:l(S) \leq l(Q)  \\ \func{dist}(Q,S)\geq l(Q) }}%
\left(\frac{l(S)}{l(Q)}\right)^{\sigma} u(Q)^{1/2} \frac{l(S)^{\kappa}}{%
(l(S) + \func{dist}(Q,S))^{\kappa + n}} v(S)^{1/2} \\
& \lesssim u(Q)^{1/2}\sum\limits_{i = 0}^{\infty} \delta^{i\sigma}\left(\sum 
_{\substack{ S:l(S) = \delta^i l(Q)  \\ \func{dist}(Q,S)\geq l(Q) }} \frac{%
l(S)^{\kappa}}{( \func{dist}(Q,S))^{\kappa + n}} v(S)\right)^{1/2}\left(\sum%
\limits_{\substack{ S:l(S) = \delta^i l(Q)  \\ \func{dist}(Q,S)\geq l(Q) }}%
\frac{l(S)^{\kappa}}{( \func{dist}(Q,S))^{\kappa+ n}} \right)^\frac{1}{2} \\
& \lesssim \sum\limits_{i= 0}^{\infty}\delta^{i\sigma}\left(\frac{u(Q)}{l(Q)^n}%
K(Q,v)\right)^\frac{1}{2} \\
&\lesssim \mathcal{A}_{2},
\end{align*}
\vspace{0.25 cm} where the last inequality follows from the fact that $\sigma > 0$.  Consider the term $A_{21}^2$. For each fixed $S$ we have

\begin{align*}
&\sum\limits_{\substack{ Q:l(S) \leq l(Q)  \\ \func{dist}(Q,S)\geq l(Q) }}%
\left(\frac{l(S)}{l(Q)}\right)^{-\sigma} u(Q)^{1/2} \frac{l(S)^{\kappa}}{%
(l(S) + \func{dist}(Q,S))^{\kappa + n}} v(S)^{1/2} \\
&\lesssim v(S)^\frac{1}{2}\sum\limits_{j=0}^{\infty}\delta^{j(1 -\sigma)}
\sum\limits_{\substack{ Q:l(S) = \delta^j l(Q)  \\ \func{dist}(Q,S)\geq l(Q) 
}}\frac{l(Q)^{\kappa}}{( \func{dist}(Q,S))^{\kappa + n}} u(Q)^\frac{1}{2} \\
&\lesssim v(S)^\frac{1}{2}\sum\limits_{j=0}^{\infty}\delta^{j(1 -\sigma)}
\left(\sum\limits_{\substack{ Q:l(S) = \delta^j l(Q)  \\ \func{dist}%
(Q,S)\geq l(Q) }}\frac{l(Q)^{\kappa}}{( \func{dist}(Q,S))^{\kappa + n}}
u(Q)\right)^\frac{1}{2}\left(\sum\limits_{\substack{ Q:l(S) = \delta^j l(Q) 
\\ \func{dist}(Q,S)\geq l(Q) }}\frac{l(Q)^{\kappa}}{( \func{dist}%
(Q,S))^{\kappa + n}}\right)^\frac{1}{2},
\end{align*}
which is bounded by 
\begin{equation*}
v(S)^{\frac{1}{2}}\sum\limits_{j=0}^{\infty }\delta ^{j(1-\sigma )}K(\delta
^{j/n}S,u)\left( \frac{1}{l(\delta ^{j/n}S)^{n}}\right) ^{\frac{1}{2}%
}\lesssim \mathcal{A}_{2},
\end{equation*}%
where the last inequality follows from the fact that $\sigma <1$. Thus, with any fixed $0<\sigma <1$ we have from the above inequalities
that 
\begin{eqnarray*}
|A_{2}^{2}| &\lesssim &\mathcal{A}_{2}\sum\limits_{Q}\sum\limits_{\epsilon
=1}^{M_{Q}-1}|\langle f,h_{Q}^{\epsilon }\rangle _{u}|^{2}+\mathcal{A}%
_{2}\sum\limits_{S}\sum_{k=1}^{M_{S}-1}|\langle g,h_{S}^{k}\rangle _{v}|^{2}
\\
&=&\left( \Vert f\Vert _{L^{2}(u)}^{2}+\Vert g\Vert _{L^{2}(v)}^{2}\right) 
\mathcal{A}_{2}=2\mathcal{A}_{2}.
\end{eqnarray*}

The proof of Lemma \ref{l:4.2} is complete.
\end{proof}

We have now reduced matters to the case of considering $\mathcal{A}_{3}^{2}$. We
further decompose $\mathcal{A}_{3}^{2}$ into 
\begin{align*}
\mathcal{A}_{1}^{3}& :=\{(Q,S)\in \mathcal{A}_{1}^{1}:l(S)\leq \delta
^{r}l(Q),Q\cap S=\emptyset ,\func{dist}(Q,S)\leq l(Q)\}, \\
\mathcal{A}_{2}^{3}& :=\{(Q,S)\in \mathcal{A}_{1}^{1}:l(S)\leq \delta
^{r}l(Q),Q\cap S\neq \emptyset ,\func{dist}(Q,S)\leq l(Q)\}.
\end{align*}%
We are now going to prove in this section that 
\begin{equation*}
|A_{1}^{3}|\leq \mathcal{A}_{2}\Vert f\Vert _{L^{2}(u)}\Vert g\Vert
_{L^{2}(v)}.
\end{equation*}

\vspace{0.5cm}

\noindent Fix $p\geq r$. Observe that 
\begin{eqnarray*}
A_{1}^{3}(p) &=&\left\vert \sum\limits_{Q}\sum\limits_{\substack{ S:{\delta }%
^{-p}l(S)=l(Q)  \\ \func{dist}(Q,S)\leq l(Q)  \\ Q\cap S=\emptyset }}%
\sum\limits_{\epsilon =1}^{M_{Q}-1}\sum\limits_{k=1}^{M_{S}-1}\langle
f,h_{Q}^{\epsilon }\rangle _{u}\langle T(uh_{Q}^{\epsilon
}),h_{S}^{k}\rangle _{v}\langle g,h_{S}^{k}\rangle _{v} \right\vert\\
&\lesssim &\left[ \sum_{Q}\sum_{\epsilon =1}^{M_{Q}-1}\left\vert \langle
f,h_{Q}^{\epsilon }\rangle _{u}\right\vert^{2}\right] ^{\frac{1}{2}}\left[
\sum_{Q}\sum_{\epsilon =1}^{M_{Q}-1}\left[ \sum\limits_{\substack{ S:{\delta 
}^{-p}l(S)=l(Q)  \\ \func{dist}(Q,S)\leq l(Q)  \\ Q\cap S=\emptyset }}%
\sum_{k=1}^{M_{S}-1}\left\vert \langle T(uh_{Q}^{\epsilon }),h_{S}^{k}\rangle _{v}\right\vert \left\vert \langle
g,h_{S}^{k}\rangle _{v}\right\vert\right] ^{2}\right] ^{\frac{1}{2}} \\
&\lesssim &\Lambda (p)\Vert f\Vert _{L^{2}(u)}\Vert g\Vert _{L^{2}(v)}.
\end{eqnarray*}%
The last two inequalities follow from the Cauchy-Schwarz inequality and we
use Fubini to get $\Vert g\Vert _{L^{2}(v)}$ above and $\delta ^{-p}$ in the
expression below. Here

\begin{equation*}
\Lambda (p)^{2}:=\delta ^{-p}\sup_{Q}\sum\limits_{\substack{ S:{\delta }%
^{-p}l(S)=l(Q)  \\ \func{dist}(Q,S)\leq l(Q)  \\ Q\cap S=\emptyset }}%
\sum_{\epsilon =1}^{M_{Q}-1}\sum_{k=1}^{M_{S}-1}|\langle T(uh_{Q}^{\epsilon
}),h_{S}^{k}\rangle _{v}|^{2}.
\end{equation*}%
But $S$ is good, so that Lemma \ref{l:4.3} applies to each child of $Q$ and $%
S$ yielding 
\begin{equation*}
\Lambda (p)^{2}\lesssim \sup\limits_{Q}\delta ^{-p}\sum\limits_{\epsilon
}\sum\limits_{k}\sum\limits_{\substack{ S:{\delta }^{-p}l(S)=l(Q)  \\ \func{%
dist}(Q,S)\leq l(Q)  \\ Q\cap S=\emptyset }}\frac{v(S_{k})}{u(Q_{\epsilon })}%
K(S_{k},1_{Q_{\epsilon }}u)^{2}.
\end{equation*}

Hence we have the following, using Lemma \ref{l:4.4} 
\begin{eqnarray*}
\Lambda (p)^{2} &\lesssim &\sup\limits_{Q}\delta ^{-p}\sum\limits_{\epsilon
}\sum\limits_{k}\sum\limits_{\substack{ S:{\delta }^{-p}l(S)=l(Q)  \\ \func{%
dist}(Q,S)\leq l(Q)  \\ Q\cap S=\emptyset }}\frac{v(S_{k})}{u(Q_{\epsilon })}%
.\left( \frac{l(S)}{l(Q)}\right) ^{-2\sigma _{0}}\cdot K(Q,1_{Q_{\epsilon
}}u)^{2} \\
&\lesssim &\sup\limits_{Q}\delta ^{-p}\delta ^{-2p\sigma
_{0}}\sum\limits_{\epsilon }\frac{u(Q_{\epsilon })}{l(Q)^{2n}}\sum\limits 
_{\substack{ S:{\delta }^{-p}l(S)=l(Q)  \\ \func{dist}(Q,S)\leq l(Q)  \\ %
Q\cap S=\emptyset }}\sum_{k}v(S_{k}) \\
&\lesssim &\delta ^{-p(1+2\sigma _{0})}\mathcal{A}_{2}.
\end{eqnarray*}%
Above, we have used that 
\begin{equation*}
K(S,1_{Q_{i}}u)\leq \frac{l(S)^{\kappa }}{(l(S)+\func{dist}%
(Q_{i},S))^{\kappa +n}}u(Q_{i}).
\end{equation*}%
This is clearly summable in $p\geq r$ as $-\sigma_{0} >\frac{1}{2}$ (as we can fix $\lambda \in (0,1)$ such that $\lambda< \frac{\kappa}{n+\kappa}$) so the proof  is complete.

\section{Stopping Cubes and Corona Decompositions}
\label{s6}

Our focus is now on the \textquotedblleft short range terms" given by $%
A_{2}^{3}$. Here we will use our pivotal condition. In this section we will
decompose $A_{2}^{3}$ further and estimate each piece.

Recall

\begin{equation*}
\mathcal{A}_{2}^{3}:=\{(Q,S)\in \mathcal{A}_{1}^{1}:l(S)\leq \delta
^{r}l(Q),Q\cap S\neq \emptyset ,\func{dist}(Q,S)\leq l(Q)\}
\end{equation*}%
and 
\begin{equation*}
A_{2}^{3}:=\sum\limits_{(Q,S)\in \mathcal{A}_{2}^{3}}\sum\limits_{\epsilon
=1}^{M_{Q}-1}\sum\limits_{k=1}^{M_{S}-1}\langle f,h_{Q}^{\epsilon }\rangle
_{u}\langle T(uh_{Q}^{\epsilon }),h_{S}^{k}\rangle _{v}\langle
g,h_{S}^{k}\rangle _{v}.
\end{equation*}%
Denote $\Delta _{Q}^{u}f=\sum\limits_{i=1}^{M_{Q-1}}\langle f,h_{Q}^{i}\rangle _{u}h_{Q}^{i}$. Then observe that: 
  
\begin{align}
\notag\langle T(\Delta_Q^u f), \Delta_S^u g\rangle_v &= \langle T((1_{Q}-1_{Q_S})\Delta_Q^u f), \Delta_S^u g\rangle_v + \langle T(1_{Q_S}\Delta_Q^u f), \Delta_S^u g\rangle_v\\ 
\label{e:neigh}& = \langle T((1_{Q}\setminus1_{Q_S})\Delta_Q^u f), \Delta_S^u g\rangle_v\\
\label{e:para}&+ \mathbf{E}_{Q_S}^u(\Delta_Q^u f)\langle T(1_{\widetilde{Q}}), \Delta_S^u g\rangle _v\\
\label{e:stop}&- \mathbf{E}_{Q_S}^u(\Delta_Q^u f)\langle T(1_{\widetilde{Q}\setminus Q_{S}}), \Delta_S^u g\rangle_v .\end{align}
First observe that $1_{\widetilde{Q}}-1_{\widetilde{Q}\setminus
Q_{S}}=1_{Q_{S}}$ and secondly 
\begin{equation*}
1_{Q_{S}}\Delta _{Q}^{u}f=\sum\limits_{i=1}^{M_{Q-1}}\langle
f,h_{Q}^{i}\rangle _{u}h_{Q}^{i}\cdot 1_{Q_{S}}
\end{equation*}%
and so 
\begin{align*}
\mathbf{E}_{Q_{S}}^{u}(\Delta _{Q}^{u}f)& =\frac{1}{u(Q_{S})}%
\int_{Q_{S}}\Delta _{Q}^{u}f(x)du(x) \\
& =\frac{1}{u(Q_{S})}\int_{X}\sum\limits_{i=1}^{M_{Q-1}}\langle
f,h_{Q}^{i}\rangle _{u}h_{Q}^{i}(x)\cdot 1_{Q_{S}}(x)du(x) \\
& =\frac{1}{u(Q_{S})}\sum\limits_{i=1}^{M_{Q-1}}\langle f,h_{Q}^{i}\rangle
_{u}\int_{X}h_{Q}^{i}(x)\cdot 1_{Q_{S}}(x)du(x).
\end{align*}%
Here $Q_{S}$ is child of $Q$ containing $S$ and $\widetilde{Q}$ the parent of $Q_{S}$.

\subsection{The Decomposition of the Short Range Term}

To estimate $A_{2}^{3}$ and conclude this section, we combine the splitting
into \eqref{e:neigh}, \eqref{e:para}, \eqref{e:stop} and the following Corona
decomposition. Namely select the cubes $\widetilde{Q}$ that appear in %
\eqref{e:neigh}-\eqref{e:stop} according to the stopping rule below. Recall the set $%
\mathcal{A}_{2}^{3}$. Using the fact that $S$ is good we can make this set
more explicit, i.e. $S\subset Q$ and $l(S)<\delta ^{r}l(Q)$ 
\begin{equation*}
\mathcal{A}_{2}^{3}:= \{(Q,S)\in \mathcal{A}_{1}^{1}:l(S)\leq \delta
^{r}l(Q),S\subset Q\}.
\end{equation*}

\subsubsection{The Corona Decomposition}

We are going to define the `stopping cubes' and the `Corona decomposition'.
Let us first define the following functionals:
\begin{align*}
\Phi(Q,1_E u) &:= v(Q)K(Q,1_E u)^2, \\
\Psi(Q,1_Eu) &:= \sup\limits_{Q= \cup_{i\geq 1}
Q_i} \sum \limits_{i\geq 1}\Phi(Q_i,1_E u),
\end{align*}
where $Q_i$ are dyadic subcubes of $Q$, hence lie in some dyadic grid as $Q$ and where the supremum is over all $r$-good dyadic subpartitions $%
\{Q_{i}\}_{i\geq 1}$ of $Q$.  We have the following pivotal condition 
\begin{equation}
\label{e:nostop}
\sum_{r\geq 1}\Phi(Q_{r},1_{Q}u)\leq \mathcal{V}^{2}u(Q).
\end{equation}%
We will use certain key properties of $\Psi$ to estimate the term $%
A_{3}^{4}$ defined below.

\begin{definition}
\label{def:5.1} Given any cube $Q_{o}$, we will set $\mathcal{S}(Q_{o})$ to
be the maximal $\mathcal{D}^{u}$ strict subcubes $S\subset Q_{o}$ such that 
\begin{equation}
\label{e:stopcube}
\Psi(S,1_{Q_{o}}u)\geq 4\mathcal{V}^{2}u(S).
\end{equation}%
The collection $\mathcal{S}(Q_{o})$ can be empty.
\end{definition}

We will be able to now recursively define $\mathcal{S}_1 := \{Q_o\}$ and 
$\mathcal{S}_{j+1}:= \cup_{S\in \mathcal{S}_j}\mathcal{S}(S)$. The
collection of $\mathcal{S}:= \cup_{j=1}^{\infty}\mathcal{S}_j$ is the
collection of stopping cubes. Let us define $\rho:\mathcal{S}\mapsto 
\mathbb{N}$ by $\rho(S) := j$ for all $S\in \mathcal{S}_j$, so that $\rho(S)$
denotes the generation in which $S$ occurs in the construction of $\mathcal{S%
}$.

Let us now discuss the associated Corona Decomposition.

\begin{definition}
\label{def:5.2} \noindent For $S^{\prime }\in \mathcal{S}$, we are going to
set $\mathcal{P}(S^{\prime })$ to be all the pairs of cubes $(Q,S)$ such
that

\begin{enumerate}
\item $Q\in \mathcal{D}^{u}$, $S\in \mathcal{D}^{v}$, $S\subset Q$ and $%
l(S)\leq \delta ^{r}l(Q)$;

\item $S^{\prime }$ is the $\mathcal{S}$ parent of $Q_{S}$ which is the child
of $Q$ containing $S$.
\end{enumerate}
\end{definition}

Observe that we can write $\mathcal{A}_{2}^{3}=\cup _{S^{\prime }\in 
\mathcal{S}}\mathcal{P}(S^{\prime })$, where $\mathcal{A}_{2}^{3}$ is as
defined above. We next define the Coronas associated to $f$ and $g$.

\begin{definition}
Let $\mathcal{C}^{u}(S^{\prime })$ be all those $Q\in 
\mathcal{D}^{u}$ such that $S^{\prime }$ is a minimal member of $\mathcal{S}$
that contains a $\mathcal{D}^{u}$ child of $Q$. The definition of $C^{v}(%
\mathcal{S^{\prime }})$ is similar but not symmetric: all those $S\in 
\mathcal{D}^{v}$ such that $S^{\prime }$ is the smallest member of $\mathcal{%
S}$ that contains $S$ and satisfies $l(S)\leq \delta ^{r}l(S^{\prime })$ together with all those $S\in \mathcal{D}_v$ such that for some $S''\in\mathcal{S}(S)$ we have $S\in \mathcal{C}^v(S'')$ with $S\subset S''$ with $l(S)\geq \delta^r l(S'')$.  The collections $\{\mathcal{C}^u (S^{\prime }):S^{\prime }\in \mathcal{S}\}$
and $\{\mathcal{C}^v (S^{\prime }):S^{\prime }\in \mathcal{S}\}$ are
referred to as the Corona Decompositions (the collection $\mathcal{C}^v$ is called the shifted Corona in the literature).
\end{definition}

We will now define the projection operators associated to these Coronas 
\begin{equation*}
\mathcal{P}_{S^{\prime }}^{u}f:=\sum_{Q\in \mathcal{C}^{u}(S^{\prime
})}\sum_{\epsilon =1}^{M_{Q}-1}\langle f,h_{Q}^{\epsilon }\rangle
_{u}h_{Q}^{\epsilon }.
\end{equation*}%
Similarly we can define $\mathcal{P}_{S^{\prime }}^{v}g$. Observe that $%
\mathcal{P}_{S^{\prime }}^{v}g$ projects only cubes $S$ with $l(S)\leq
\delta ^{r}l(S^{\prime })$.

We have the estimate below which we will use in the proofs below 
\begin{equation*}
\sum_{S^{\prime }\in \mathcal{S}}\Vert \mathcal{P}_{S^{\prime }}^{u}f\Vert
_{L^{2}(u)}^{2}\leq \sup_{Q\in \mathcal{D}^{u}}M_{Q}\Vert f\Vert
_{L^{2}(u)}^{2}
\end{equation*}%
where $M_Q$ is the number of children of $Q$ in the grid $\mathcal{D}^u$.  Recall that we also assume throughout the paper $\sup\limits_{Q\in \mathcal{D}%
^{u}}M_{Q}<\infty $. We have a similar inequality for $\mathcal{P}%
_{S^{\prime }}^{v}$.

Observe in the definition of stopping cubes we are using the functional $%
\Psi$ associated with hypothesis \eqref{e:nostop}. So the
stopping cubes can be viewed as the enemy of verifying \eqref{e:nostop}.

\begin{definition}
\label{def:5.3} Given a pair $(Q,S)\in \mathcal{A}_{2}^{3}$, choose $%
\widetilde{Q}\in \mathcal{S}$ to be the unique stopping cube such that $%
Q_{S}\in \mathcal{C}^{u}(\widetilde{Q})$, where $Q_{S}$ is the $\mathcal{S}$
child of $\widetilde{Q}$. Equivalently, $\widetilde{Q}\in \mathcal{S}$ is
determined by the requirement $(Q,S)\in \mathcal{P}(\widetilde{Q}).$
\end{definition}

Note that if $Q_{S}\notin \mathcal{S}$, then $Q\subset \widetilde{Q}$, while
if $Q_{S}\in \mathcal{S}$, then $\widetilde{Q}$ is the child of $Q$
containing $S$. With the choice of $\widetilde{Q}$ in the splitting of %
\eqref{e:neigh}-\eqref{e:stop} we obtain $|A_{2}^{3}|\leq
\sum_{j=1}^{3}|A_{j}^{4}|$ where 
\begin{align}
A_{1}^{4}& :=\sum_{(Q,S)\in \mathcal{A}_{2}^{3}}T(1_{Q\setminus
Q_{S}}u\Delta _{Q}^{u}f),\Delta _{S}^{u}g\rangle _{v}  \label{e:a14}, \\
A_{2}^{4}& :=\sum_{S^{\prime }\in \mathcal{S}}\sum_{(Q,S)\in \mathcal{P}%
(S^{\prime })}\mathbf{E}_{Q_{S}}^{u}(\Delta _{Q}^{u}f)\langle T(1_{S^{\prime
}}u),\Delta _{S}^{u}g\rangle _{v}\label{e:a24}, \\
A_{3}^{4}& :=\sum_{S^{\prime }\in \mathcal{S}}\sum_{(Q,S)\in \mathcal{P}%
(S^{\prime })}\mathbf{E}_{Q_{S}}^{u}(\Delta _{Q}^{u}f)\langle T(1_{S^{\prime
}\setminus Q_{S}}u),\Delta _{S}^{u}g\rangle _{v}.\label{e:a34}
\end{align}%
The three terms above are referred to as the \textit{neighbor}, \textit{%
paraproduct} and \textit{stopping term} respectively.

The paraproduct term $A_{2}^{4}$ is further decomposed and estimates on this
term will be handled in Section \ref{s:Paraproduct}, while in the remainder
of this section we will prove: 
\begin{align}
|A_{1}^{4}|& \lesssim \mathcal{A}_{2}\Vert f\Vert _{L^{2}(u)}\Vert g\Vert
_{L^{2}(v)}  \label{e:a14estimate}, \\
|A_{3}^{4}|& \lesssim \mathcal{V}\Vert f\Vert
_{L^{2}(u)}\Vert g\Vert _{L^{2}(v)}.\label{e:a34estimate}
\end{align}

\subsection{Control of the Neighbor Term $A_1^4$}

The neighbor terms are defined in \eqref{e:a14} and we are to prove %
\eqref{e:a14estimate}. Recall we have $Q\in \mathcal{D}^{u}$, $S\in \mathcal{G}%
^{v}$ contained in $Q$, with $l(S)\leq \delta ^{r}l(Q)$ and $Q_{S}$ is the
child of $Q$ containing $S$.

Fix a child $\theta \in \{1,2,...,M_{Q}-1\}$ and an integer $s^{\prime }\geq
r$. Here we use that $Q\setminus Q_{\theta} = \cup_{\substack{ i=1  \\ i\neq
\theta}}^{M_Q-1} Q_i$.

We are now going to estimate the inner product as that in \eqref{e:neigh}: 
\begin{align*}
\langle T(1_{Q}\setminus1_{Q_{\theta}} u\Delta_Q^u f), \Delta_S^u g\rangle_v
&= \sum\limits_{\substack{ i =0  \\ i\neq \theta}}^{M_{Q}-1}\langle
T(1_{Q_i}u\Delta_Q^u f), \Delta_S^u g\rangle_v \\
&= \sum\limits_{\substack{ i=1  \\ i\neq \theta}}^{M_Q-1}\mathbf{E}%
_{Q_i}^{u}(\Delta_Q^u f)\langle T(1_{Q_{i}} u),\Delta_S^u g\rangle_v.
\end{align*}
Here we will use $\|\Delta_S^v g\|_{L^2(v)} = \left(\sum\limits_{k=1}^{M_S -
1}|\langle g,h_S^k\rangle|_{v}^2\right)^\frac{1}{2}$ and $\frac{l(S)}{%
l(Q_{\theta})} = \delta^{s^{\prime }}$ in Lemma \ref{l:4.3} with $S\subset
Q_{\theta} \subset Q$ to obtain 
\begin{align*}
|\langle T(1_{Q_{i}} u), \Delta_S^u g\rangle_v |&\lesssim v(S)^\frac{1}{2}%
\left(\sum\limits_{k=1}^{M_{S} - 1}|\langle g,h_S^k\rangle|_{v}^2\right)^%
\frac{1}{2}\sum\limits_{\substack{ i=0  \\ i\neq\theta}}%
^{M_{Q}-1}K(S,1_{Q_i}u) \\
&\lesssim v(S)^\frac{1}{2}\left(\sum\limits_{k=1}^{M_{S} - 1}|\langle
g,h_S^k\rangle|_{v}^2\right)^\frac{1}{2}\delta^{-(s^{\prime
}\sigma_{0})}\sum\limits_{\substack{ i=0  \\ i\neq\theta}}%
^{M_{Q}-1}K(Q_{\theta},1_{Q_i}u).
\end{align*}

Here we applied Lemma \ref{l:4.3} and Lemma \ref{l:4.4} to $S\subset
Q\setminus \cup_{\substack{ i=1  \\ i\neq \theta}}^{M_Q-1} Q_i$.

\vspace{0.5cm}

For the sum below we keep the lengths of the cubes $S$ fixed and we are under the
assumption that $S\subset Q_{\theta }$. Define 
\begin{equation*}
\Lambda (Q,\theta ,s^{\prime })^2:= \sum\limits_{\substack{ S:l(S)=\delta
^{s^{\prime }}l(Q)  \\ S\subset Q_{\theta }}}\left(
\sum\limits_{k=1}^{M_{S}-1}|\langle g,h_{S}^{k}\rangle |_{v}^{2}\right) .
\end{equation*}%
Then we have the following estimate using Cauchy-Schwarz

\begin{align*}
A_1^4(Q,\theta,s^{\prime }) &:= \sum\limits_{\substack{ S:l(S)=\delta^{s^{%
\prime }}l(Q)  \\ S\subset Q_{\theta}}}|\langle T(1_{Q}\setminus1_{Q_{\theta}}
u\Delta_Q^u f), \Delta_S^u g\rangle_v| \\
& \leq \delta^{-(s^{\prime }\sigma_{0})}\sum\limits_{\substack{ i=1  \\ %
i\neq \theta}}^{M_{Q}-1}|\mathbf{E}_{Q_i}^{u}(\Delta_Q^u
f)|K(Q_{\theta},\sum\limits_{\substack{ i=1  \\ i\neq \theta}}^{M_{Q}-1}
1_{Q_i}u)\sum\limits_{\substack{ S:l(S)=\delta^{s^{\prime }}l(Q)  \\ %
S\subset Q_{\theta}}}v(S)^\frac{1}{2}\left(\sum\limits_{k=1}^{M_S - 1}|\langle
g,h_S^k\rangle|_{v}^2\right)^\frac{1}{2} \\
&\leq \delta^{-(s^{\prime }\sigma_{0})}\sum\limits_{\substack{ i=1  \\ i\neq
\theta}}^{M_{Q}-1}|\mathbf{E}_{Q_i}^{u}(\Delta_Q^u
f)|K\left(Q_{\theta},\sum\limits_{\substack{ i=1  \\ i\neq \theta}}%
^{M_{Q}-1} 1_{Q_i}u\right)v(Q_{\theta})^\frac{1}{2}\Lambda(Q,\theta,s^{\prime
}).
\end{align*}
We will now use the following to estimate $A_1^4(Q,\theta,s^{\prime })$: 
\begin{equation*}
|\mathbf{E}_{Q_i}^{u}(\Delta_Q^u f)| \leq \left(\sum\limits_{j=1}^{M_Q -
1}|\langle f,h_Q^j\rangle_{u}|^2\right)^\frac{1}{2} u(Q_i)^{-\frac{1}{2}}.
\end{equation*}
Substituting this into the above we find:

\begin{align*}
A_{1}^{4}(Q,\theta ,s^{\prime})&\lesssim \delta^{-(s^{\prime }\sigma _{0})}\left(
\sum\limits_{i=1}^{M_{Q}-1}|\langle f,h_{Q}^{i}\rangle |_{u}^{2}\right) ^{%
\frac{1}{2}}\Lambda (Q,\theta ,s^{\prime })\sum\limits_{\substack{ i=1  \\ %
i\neq \theta }}^{M_{Q}-1}\sum\limits_{\substack{ j=1  \\ j\neq \theta }}%
^{M_{Q}-1}u(Q_{i})^{-\frac{1}{2}}K(Q_{\theta },1_{Q_{j}}u)v(Q_{\theta })^{%
\frac{1}{2}} \\
& \lesssim \mathcal{A}_{2}\delta ^{-(s^{\prime }\sigma _{0})}\left(
\sum\limits_{i=1}^{M_{Q}-1}|\langle f,h_{Q}^{i}\rangle |_{u}^{2}\right) ^{%
\frac{1}{2}}\Lambda (Q,\theta ,s^{\prime }).
\end{align*}%
We can then sum $A_{1}^{4}(Q,\theta ,s^{\prime })$ in $Q$ and apply
Cauchy-Schwarz to show that 
\begin{equation*}
\sum_{Q}A_{1}^{4}(Q,\theta ,s^{\prime })\lesssim \mathcal{A}_{2}\delta
^{-(s^{\prime }\sigma _{0})}\Vert f\Vert _{L^{2}(u)}\Vert g\Vert _{L^{2}(v)};
\end{equation*}%
this last estimate is then summable in the parameter $s^{\prime }$ as we have $-\sigma_0>\frac{1}{2}$ as mentioned in Section \ref{s5}. This
completes the proof of the estimate \eqref{e:a14estimate}.

\subsection{Control of the Stopping Term $A_3^4$}

To control \eqref{e:a34} we want to prove \eqref{e:a34estimate}. Here we will use
the hypothesis \eqref{e:stopcube}.

We first define for $S^{\prime }\in \mathcal{S}$ and $s\geq 0$ an
integer 
\begin{equation*}
A_{3}^{4}(S^{\prime },s):=\sum_{\substack{ (Q,S)\in \mathcal{P}(S^{\prime }) 
\\ l(S)=\delta ^{s}l(Q)}}|\mathbf{E}_{Q_{S}}^{u}(\Delta _{Q}^{u}f)\langle
T(1_{S^{\prime }\setminus Q_{S}}u),\Delta _{S}^{u}g\rangle_{v} |\lesssim
\delta ^{-\sigma_0 s}\mathcal{V}\mathcal{F}(S^{\prime
})\Lambda (S^{\prime },s),
\end{equation*}%
where 
\begin{align*}
\mathcal{F}\left( S^{\prime }\right) ^{2}& :=\sum\limits_{Q\in \mathcal{C}%
^{u}(S^{\prime })}\sum\limits_{\epsilon =1}^{M_{Q}-1}|\langle
f,h_{Q}^{\epsilon }\rangle _{u}|^{2},\\
\Lambda \left( S^{\prime },s\right) ^{2}& :=\sum\limits_{Q\in \mathcal{C}%
^{u}(S^{\prime })}\sum\limits_{\substack{ S:(Q,S)\in \mathcal{P}(S^{\prime
})  \\ l(S)=\delta ^{s}l(Q)}}\sum\limits_{k=1}^{M_{S}-1}|\langle
g,h_{S}^{k}\rangle _{v}|^{2}.
\end{align*}%
Using Cauchy-Schwarz in the variable $Q$ variable above, and appealing to
the following inequality 
\begin{equation*}
|\mathbf{E}_{Q_{i}}^{u}(\Delta _{Q}^{u}f)|\leq \left(
\sum\limits_{j=1}^{M_{Q}-1}|\langle f,h_{Q}^{j}\rangle_{u} |^{2}\right) ^{%
\frac{1}{2}}u(Q_{i})^{-\frac{1}{2}},
\end{equation*}%
to continue the estimate for $A_{3}^{4}(S^{\prime },s)$,
\begin{equation*}
A_{3}^{4}(S^{\prime },s)\leq \mathcal{F}(S^{\prime })\left[
\sum\limits_{Q\in \mathcal{C}^{u}(S^{\prime })}\left( \sum\limits_{\substack{
S:(Q,S)\in \mathcal{P}(S^{\prime })  \\ l(S)=\delta ^{s}l(Q)}}\frac{1}{%
u(Q_{S})^{\frac{1}{2}}}|\langle T(1_{S^{\prime }\setminus Q_{S}}u),\Delta
_{S}^{u}g\rangle_{v} |\right) ^{2}\right] ^{\frac{1}{2}}.
\end{equation*}%
We can now estimate the terms in the square bracket above by 
\begin{align*}
\sum\limits_{Q\in \mathcal{C}^{u}(S^{\prime })}\sum\limits_{\substack{ %
(Q,S)\in \mathcal{P}(S^{\prime })  \\ l(S)=\delta ^{s}l(Q)}}%
\sum\limits_{k=1}^{M_{S}-1}|\langle g,h_{S}^{k}\rangle _{v}|^{2}& \times
\sum\limits_{\substack{ (Q,S)\in \mathcal{P}(S^{\prime })  \\ l(S)=\delta
^{s}l(Q)}}\sum\limits_{k=1}^{M_{S}-1}\frac{1}{u(Q_{S})}|\langle
T(1_{S^{\prime }\setminus Q_{S}}u),h_{S}^{k}\rangle_{v} |^{2} \\
& \lesssim \Lambda \left( S^{\prime },s\right) ^{2}A(S^{\prime },s),
\end{align*}%
where 
\begin{equation*}
A(S^{\prime },s):=\sup\limits_{Q\in \mathcal{C}^{u}(S^{\prime })}\sum\limits 
_{\substack{ (Q,S)\in \mathcal{P}(S^{\prime })  \\ l(S)=\delta ^{s}l(Q)}}%
\sum\limits_{k=1}^{M_{S}-1}\frac{1}{u(Q_{S})}|\langle T(1_{S^{\prime
}\setminus Q_{S}}u),h_{S}^{k}\rangle_{v} |^{2}.
\end{equation*}%
We will now estimate the term $A(S^{\prime },s)$. We will denote the
children of $Q$ by $Q_{\theta }$ for $\theta \in \{1,2....,M_{Q}-1\}$ and
denote the children of $S$ by $S_{k}$ for $k\in \{1,2,...,M_{S}-1\}$. Also
observe that $h_{S}^{k}$ is supported on $S$ and $h_{S}^{k}=\sum%
\limits_{k=1}^{M_{S}-1}C_{S_{k}}1_{S_{k}}$ such that $\Vert h_{S}^{k}\Vert
_{{L}^{2}(u)}^2=\sum\limits_{k=1}^{M_{S}-1}C_{S_{k}}^{2}=1$. So using %
\eqref{e:4.4} we have the following: 
\begin{align*}
A(S^{\prime },s)& \lesssim \sup\limits_{Q\in \mathcal{C}^{u}(S^{\prime
})}\sup\limits_{\theta \in \{1,2...M_{Q}-1\}}\sum\limits_{\substack{ %
(Q,S)\in \mathcal{P}(S^{\prime }):Q_{S}=Q_{\theta }  \\ l(S)=\delta ^{s}l(Q) 
}}\sum\limits_{k=1}^{M_{S}-1}\frac{1}{u(Q_{\theta })}\Phi (S_{k},1_{S^{\prime}\setminus Q_{\theta }}u) \\
&\lesssim \sup\limits_{Q\in \mathcal{C}^{u}(S^{\prime
})}\sup\limits_{\theta \in \{1,2...M_{Q}-1\}}\sum\limits_{\substack{ %
(Q,S)\in \mathcal{P}(S^{\prime }):Q_{S}=Q_{\theta }  \\ l(S)=\delta ^{s}l(Q) 
}}\sum\limits_{k=1}^{M_{S}-1}\frac{1}{u(Q_{\theta })}v(S_k)K(S_k,1_{S^{\prime}\setminus Q_{\theta}}u)^2\\
&\lesssim \sup\limits_{Q\in \mathcal{C}^{u}(S^{\prime
})}\sup\limits_{\theta \in \{1,2...M_{Q}-1\}}\sum\limits_{\substack{ %
(Q,S)\in \mathcal{P}(S^{\prime }):Q_{S}=Q_{\theta }  \\ l(S)=\delta ^{s}l(Q) 
}}\sum\limits_{k=1}^{M_{S}-1}\frac{1}{u(Q_{\theta })}v(Q_{\theta})\left(\frac{l(Q_{\theta})}{l(S_{k})}\right)^{\sigma_0}K(Q_{\theta},1_{S^{\prime}\setminus Q_{\theta}}u)^2\\
&\lesssim \sup\limits_{Q\in \mathcal{C}^{u}(S^{\prime
})}\sup\limits_{\theta \in \{1,2...M_{Q}-1\}}\sum\limits_{\substack{ %
(Q,S)\in \mathcal{P}(S^{\prime }):Q_{S}=Q_{\theta }  \\ l(S)=\delta ^{s}l(Q) 
}}\sum\limits_{k=1}^{M_{S}-1}\frac{1}{u(Q_{\theta })}\delta^{-\sigma_{0}(s+1)}\Phi(Q_{\theta},1_{S^{\prime}\setminus Q_{\theta}}u)\\
& \lesssim \sup\limits_{Q\in \mathcal{C}^{u}(S^{\prime
})}\sup\limits_{\theta \in \{1,2...M_{Q}-1\}}\frac{1}{u(Q_{\theta })}%
\sum\limits_{\substack{ (Q,S)\in \mathcal{P}(S^{\prime }):Q_{S}=Q_{\theta } 
\\ l(S)=\delta ^{s}l(Q)}}\sum\limits_{k=1}^{M_{S}-1}\delta^{-\sigma_{0}(s+1)}\Phi (Q_{\theta},1_{S^{\prime}}u) \\
& \lesssim \sup\limits_{Q\in \mathcal{C}^{u}(S^{\prime
})}\sup\limits_{\theta \in \{1,2...M_{Q}-1\}}\frac{1}{u(Q_{\theta })}\delta
^{-\sigma_{0}(s+1)} u(S^{\prime})\mathcal{V}^2 \\
& \lesssim \delta ^{-\sigma_{0}(s+1)}\mathcal{V}^2.
\end{align*}%
Here we used Lemma \ref{l:4.4} in the third inequality, used \eqref{e:nostop} in the second to last line and also that the number of children of any cube is uniformly finite.  Here we have also used that $(Q,S)\in \mathcal{P}(S^{\prime })$, so we have that $%
S^{\prime }$ is the $\mathcal{S}$-parent of $Q_{S}$ hence $Q_{S}$ is not a
stopping cube, so \eqref{e:stopcube} does not hold, hence giving the estimate above.

We can observe that 
\begin{equation*}
\sum\limits_{S^{\prime }\in \mathcal{S}}\mathcal{F}\left( S^{\prime }\right)
^{2}\lesssim \Vert f\Vert _{L^{2}(u)}^{2}.
\end{equation*}%
And we have the following
\begin{equation*}
|A_{3}^{4}|\leq \sum\limits_{S^{\prime }\in \mathcal{S}}\sum\limits_{s=0}^{%
\infty }A_{3}^{4}(S^{\prime },s)\lesssim \mathcal{V}\Vert
f\Vert _{L^{2}(u)}^{2}\Vert g\Vert _{L^{2}(v)}^{2}.
\end{equation*}

\section{The Carleson Measure Estimates}

In this section we will prove Carleson measure estimates useful for the
analysis on the paraproduct term $A_2^4$.  In the lemma below we use the stopping time definition. For a cube $S\in 
\mathcal{D}^{v}$ let 
\begin{equation*}
\Tilde{\mathcal{P}}_{S}^{v}(g):=\sum_{\substack{ S^{\prime }\in \mathcal{D}%
^{v}:S^{\prime }\subset S}}\sum_{k=1}^{M_{S^{\prime }}-1}\langle
g,h_{S^{\prime }}^{k}\rangle h_{S^{\prime }}^{k}.
\end{equation*}

\noindent Observe that the projection $\Tilde{\mathcal{P}_{S}^{v}}$ is onto
the span of all Haar functions $h_{S}^{\prime }$ supported in the $\mathcal{D%
}^{v}$ cube $S$. In contrast $\mathcal{P}_{S}^{v}$ projects onto the span of
all Haar functions $h_{S}$ with $S$ in the corona $\mathcal{C}^{v}(S)$ where 
$S$ is a stopping cube in the $\mathcal{D}^{u}$ grid.

\begin{lemma}
\label{l:6.1} Fix a cube $Q_{0}\in \mathcal{D}^{u}$ and let $\hat{Q}_{0}\in 
\mathcal{S}$ be its $\mathcal{S}$-parent. Let $\{Q_{m}:m\geq 1\}\subset 
\mathcal{D}^{u}$ be a strict subpartition of $Q_{0}$.  Suppose that $Q_m$ is good for $m\geq 1$.
Let $\{S_{m,s^{\prime }}:s^{\prime }\geq 1\}\subset \mathcal{D}^{v}$ be a subpartition of $Q_{m}$ with $l(S_{m,s^{\prime }})<\delta ^{r}l(Q_{m})$ for all $m,s^{\prime }\geq 1$. 
We then have the following 
\begin{equation}
\sum_{m,s^{\prime }\geq 1}\Vert \Tilde{\mathcal{P}}_{S_{m,s^{\prime
}}}^{v}T(1_{\hat{Q}_{0}\setminus Q_{m}}u)\Vert _{L^{2}(v)}^{2}\lesssim \mathcal{V}^2u(Q_{0}).
\label{e:estimate1}
\end{equation}
\end{lemma}

\begin{proof}
We are now going to use the $L^{2}$ formulation \eqref{e:4.4} of Lemma \ref%
{l:4.3} to deduce \eqref{e:estimate1}. We begin with the $L^{2}$ formulation to
obtain 
\begin{align*}
\sum\limits_{m,s^{\prime }\geq 1}\Vert \Tilde{\mathcal{P}}_{S_{m,s^{\prime
}}}^{v}T(1_{\hat{Q}_{0}\setminus Q_{m}}u)\Vert _{L^{2}(v)}^{2}& \lesssim
\sum\limits_{m,s^{\prime }\geq 1}\Phi (S_{m,s^{\prime }},1_{\hat{Q}_{0}\setminus Q_{m}}u) \\
& \lesssim\sum\limits_{m,s^{\prime }\geq 1}\Phi (S_{m,s^{\prime }},1_{\hat{Q}%
_{0}\setminus Q_{0}}u)+\sum\limits_{m,s^{\prime }\geq 1}\Phi (S_{m,s^{\prime
}},1_{Q_{0}\setminus Q_{m}}u).
\end{align*}%
The last inequality follows from the definition of $\Phi $ and 
\begin{equation*}
K(S,1_{\hat{Q}_{0}\setminus Q_{m}}u)=K(S,1_{\hat{Q}_{0}\setminus
Q_{0}}u)+K(S,1_{Q_{0}\setminus Q_{m}}u).
\end{equation*}%
If $Q_{0}\neq \hat{Q}_{0}$, we estimate the sum involving $\hat{Q}%
_{0}\setminus Q_{0}$ using the fact that $\{S_{m,s^{\prime }}\}_{m,s^{\prime
}\geq 1}$ is a $m$-good subpartition of $Q_{0}$. We use Lemma \ref{l:4.4} to get the first inequality below and given the
fact that \eqref{e:stopcube} fails and using \eqref{e:nostop} we obtain the following when $0<t<r$
\begin{eqnarray*}
\sum\limits_{m,s^{\prime }\geq 1}\Phi (S_{m,s^{\prime }},1_{\hat{Q}%
_{0}\setminus Q_{0}}u)&\lesssim& \sum\limits_{m,s^{\prime }\geq 1}\left[ \frac{l(Q_{m})}{l(S_{m,s^{\prime }})}\right] ^{\sigma_{0}}\Phi
(Q_{m},1_{\hat{Q}_{0}}u)\\&\lesssim & \sum\limits_{m,b^{\prime}\geq 1}\sum\limits_{\substack{S_{m,b^{\prime }}\in \mathcal{D}^{v} \\l(S_{m,b^{\prime }})= \delta^{t}l(Q_m)}}\left[ \frac{l(Q_{m})}{l(S_{m,b^{\prime }})}\right] ^{\sigma_{0}}\Phi
(Q_{m},1_{\hat{Q}_{0}}u) \\&\lesssim&\sum\limits_{m,b^{\prime}\geq 1}\sum_{t<r}\delta ^{-\sigma_{0} t}\Phi(Q_{m},1_{\hat{Q}_{0}}u)
\\&\lesssim & \mathcal{V}^{2}u(Q_{0}).
\end{eqnarray*}

\vspace{0.5 cm}

\noindent Then to estimate the sum involving $Q_{0}\setminus Q_{m}$, we use
the fact that $\{S_{m,s^{\prime }}\}_{m,s^{\prime }\geq 1}$ is a $r $%
-good subpartition of $Q_{m}$ for each $m$. Then using the same steps as above and using \eqref{e:nostop} we
obtain 
\begin{equation*}
\sum\limits_{m,s^{\prime }\geq 1}\Phi (S_{m,s^{\prime }},1_{Q_{0}\setminus
Q_{m}}u)\lesssim \sum\limits_{m,s^{\prime}\geq 1}\left[ 
\frac{l(Q_{m})}{l(S_{m,s^{\prime }})}\right] ^{\sigma_{0} }\Phi (Q_{m},1_{Q_{0}}u)\lesssim \mathcal{V}^{2}u(Q_{0}).
\end{equation*}%
This last estimate can be also used to prove when $Q_{0}=\hat{Q}_{0}\in 
\mathcal{S}$.
\end{proof}

\begin{theorem}
\label{t:carltheorem1} We have the following Carleson measure estimates for $%
S^{\prime }\in \mathcal{S}$ and $K\in \mathcal{D}^{u}$:
\begin{align}
\label{e:estimate2}\sum_{J'\in \mathcal{S}(S')}u(J')\leq \frac{1}{4}u(S') \quad \textnormal{ and } \quad \sum_{S' \in \mathcal{S}:S'\subsetneq K} u(S')\lesssim u(K);\\
\label{e:estimate3}\sum_{J\in \mathcal{C}^v (S'):J\subset K, l(J)< \delta^r l(K)}|\langle T(1_{S'} u),h_{J}^{v} \rangle_v|^2 \lesssim (\mathcal{V}^{2} + \mathcal{T}^2) u(K).
\end{align}
\end{theorem}

\begin{proof}
For the second part of the inequality in $(\ref{e:estimate2})$, it suffices to
verify it for $K=S_{0}\in \mathcal{S}$. And then the case we are interested
in follows from recursive application of the estimate from the first half of
the inequality \eqref{e:estimate2} to the cube $S_{0}$ and all of its children
in $\mathcal{S}$.

We now prove the first half of \eqref{e:estimate2}. The cubes in the collection $%
\mathcal{S}(S_{0})=\{S_{m}^{\prime }:m\geq 1\}$ given in the Definition \ref%
{def:5.1} are pairwise disjoint and strictly contained in $Q_{0}$. Each of
them satisfies \eqref{e:stopcube}, so we can apply that estimate along with %
\eqref{e:nostop} to get 
\begin{equation*}
\sum_{J^{\prime }\in \mathcal{S}(S_{0})}u(J^{\prime })=\sum_{r\geq
1}u(S_{m}^{\prime })\leq \frac{1}{4\mathcal{V}^{2}}%
\sum_{m\geq 1}\Psi (S_{m}^{\prime },1_{S_{0}}u)\leq \frac{1}{4}%
u(S_{0}).
\end{equation*}

We will now prove \eqref{e:estimate3}. First we will fix $S^{\prime }\in \mathcal{S%
}$ and $K$, which can be assumed to be a subset of $S^{\prime }$. We will
apply the operator $T$ to $u 1_K$ as opposed to $u 1_{S^{\prime }}$, and we can
then use the testing condition for $T$ to get 
\begin{equation*}
\sum_{J\in \mathcal{C}^v (S'):J\subset K, l(J)< \delta^r l(K)}|\langle T(1_{K}
u),h_{J}^{v} \rangle_v|^2 \leq \int_K |T(1_K u)|^2 dv \leq \mathcal{T}^2
u(K).
\end{equation*}

Now we will apply $T$ to $u1_{S^{\prime }\setminus K}$ and show 
\begin{equation*}
\sum_{J\in \mathcal{C}^v (S'):J\subset K, l(J)< \delta^r l(K)}|\langle T(1_{S^{\prime
}\setminus K}u),h_{J}^{v}\rangle _{v}|^{2}\lesssim \mathcal{V}^{2}u(K).
\end{equation*}%
We can assume that $K\subsetneq S^{\prime }$ and there is some $J\in 
\mathcal{C}^{v}(S^{\prime })$ with $J\subset K$. From this we can say that $%
K $ is not a stopping cube. Therefore the cube $K$ must fail %
\eqref{e:stopcube}.

Let $\mathcal{J}$ denote the maximal cubes $J\in \mathcal{C}%
^{v}(S^{\prime })$ with $J\subset K$ and $l(J)\leq \delta ^{r}l(K)$. Using
the definition of $\Tilde{\mathcal{P}}_{S}^{v}(g)$, we can use \eqref{e:estimate2}, with $\hat{Q}%
=S^{\prime }$ and $J\in \mathcal{J}$. It gives 
\begin{equation*}
\sum_{J\in \mathcal{J}}\Vert \Tilde{\mathcal{P}}_{J}^{v}T(1_{S^{\prime
}\setminus K}u)\Vert _{L^{2}(v)}^{2}\lesssim \sum_{J\in \mathcal{J}}\Phi
(J,1_{S^{\prime }\setminus K}u)\lesssim \mathcal{V}^{2}u(K).
\end{equation*}%
The second inequality uses the fact $K$ fails \eqref{e:nostop}. This proves %
\eqref{e:estimate3}.
\end{proof}

The following Carleson measure estimate uses hypothesis \eqref{e:nostop} in the
proof. It will provide the decay in the parameter $t$ in Theorem \ref{t:estimate}. For
all integers $t\geq 0$, we define for $S\in \mathcal{S}$, which are not
maximal 
\begin{equation*}
\alpha _{t}(S):=\sum_{S^{\prime }:\pi _{S}^{t}(S^{\prime })=S}\Vert \mathcal{%
P}_{S^{\prime }}^{v}T(u1_{\pi _{S}^{1}(S)\setminus S})\Vert _{L^{2}(v)}^{2}.
\end{equation*}%
Here $\pi _{\mathcal{D}^{u}}^{t}(S)$ is the $t$-ancestor of $S$ in $\mathcal{D}^{u}$. Also we are
taking the projection $T(u1_{\pi _{S}^{1}(S)\setminus S})$ associated to
parts of the corona decomposition which are `far below' $S$. We have the
following off-diagonal estimate.

\begin{theorem}\label{thm off-dia}
\label{t:estimate} The following Carleson measure estimate holds: 
\begin{equation}
\label{e:estimate4}
\sum_{S:\pi _{\mathcal{D}^{u}}^{1}(S)\subset K}\alpha _{t}(S)\lesssim \delta
^{-\sigma_{0} t}\mathcal{V}^{2}u(K),\hspace{1cm}K\in \mathcal{D%
}^{u}.
\end{equation}%
The implicit constant is independent of the choices of the cube $K$ and $t\geq 1$.
\end{theorem}

In the estimate \eqref{e:estimate4}, we need to observe the fact that the dyadic
parent $\pi _{\mathcal{D}^{v}}^{1}(S)$ of $S$ appears. In fact the role of
dyadic parents is revealed in the next proof. We use the negation of %
\eqref{e:stopcube} when $\pi _{\mathcal{D}^{u}}^{1}(S)\notin \mathcal{S}$, and
otherwise we use \eqref{e:nostop}.

\begin{proof}
We will first show that 
\begin{equation*}
\sum\limits_{S\in \mathcal{S}(\hat{S})}\alpha _{t}(S)\leq \delta ^{-\sigma_{0} t}%
\mathcal{V}^{2}u(\hat{S}),\quad \hat{S}\in \mathcal{S}.
\end{equation*}%
For this proof, we will set $\mathcal{S}_{t}(S):=\{S^{\prime }\in \mathcal{S}%
:\pi _{S}^{t}(S^{\prime })=S\}$, using this notation for $S\in \mathcal{S}(%
\hat{S})$. We apply the $L^{2}$ formulation estimate \eqref{e:l2form} of
Lemma \ref{l:4.3} to the expression $\alpha _{t}$.
\begin{align}
\label{e:maximalset}
\mathcal{S}(S') := \{J \in \mathcal{C}^v (S):\textnormal{$J$ is maximal with}\hspace{0.1 cm} J\subset S',l(J) < \delta^r l(S')\}.
\end{align}

From the definition above we have $l(J)<\delta ^{r}l(S^{\prime })$ for all $%
J\in \mathcal{S}(S^{\prime })$ and as all Haar functions have mean zero, we
can apply the $L^{2}$ formulation \eqref{e:l2form} of Lemma \ref{l:4.3}. Using this, we see
that 
\begin{equation*}
\alpha _{t}(S)\lesssim \sum\limits_{S^{\prime }\in \mathcal{S}%
_{t}(S)}\sum\limits_{J\in \mathcal{S}(S^{\prime })}\Phi (J,1_{\hat{S}\setminus
S}u).
\end{equation*}%
And by using \eqref{e:nostop} we get 
\begin{equation}
\label{e:decay}
\hspace{0.5cm}\sum_{S\in \mathcal{S}(\hat{S})}\alpha _{t}(S)\lesssim
\sum_{S\in \mathcal{S}(\hat{S})}\sum_{S^{\prime }\in \mathcal{S}%
_{t}(S)}\sum_{J\in \mathcal{S}(S^{\prime })}\Phi (J,1_{\hat{S}\setminus
S}u)\lesssim \delta ^{-\sigma_{0}t}\mathcal{V}^{2} \sum_{S\in \mathcal{S}(\hat{S})} u(S)
\lesssim \delta ^{-\sigma_{0} t}\mathcal{V}^{2}u(\hat{S}).
\end{equation}%
The last inequality follows from hypothesis \eqref{e:nostop}.

Now fix $K$ as in \eqref{e:estimate4} and let $\hat{S}\in \mathcal{S}$ be the
stopping cube such that $K\in \mathcal{C}^{u}(\hat{S}).$ Let $\mathcal{G}%
_{1}:=\{S_{i}\}_{i}$ be the maximal cubes from $\mathcal{S}$ that are
strictly contained in $K$. Inductively we define the $(k+1)^{st}$ generation 
$\mathcal{G}_{k+1}$ to consist of the maximal cubes from $\mathcal{S}$
that are strictly contained in some $k^{th}$ generation cube $S\in 
\mathcal{G}_{k}$. Inequality \eqref{e:decay} shows that 
\begin{equation*}
\sum\limits_{S\in \mathcal{G}_{k+1}}\alpha _{t}(S)\lesssim \delta ^{-\sigma_{0} t}%
\mathcal{V}^{2}\sum\limits_{S\in \mathcal{G}_{k+1}}u(S).
\end{equation*}%
We have from \eqref{e:estimate2} that 
\begin{equation*}
\sum\limits_{k=1}^{\infty }\sum\limits_{S\in \mathcal{G}_{k}}u(S)\lesssim
\sum\limits_{S\in \mathcal{G}_{1}}u(S)\lesssim u(K).
\end{equation*}%
This will be all we need for the case $K=\hat{S}$. For the case $K\neq \hat{S%
}$ we will use Lemma \ref{l:6.1} to control the first generation of
cubes $S$ in $\mathcal{G}_{1}$: 
\begin{equation*}
\sum\limits_{S\in \mathcal{G}_{1}}\alpha _{t}(S)\lesssim \delta ^{-\sigma_{0}
t}u(K).
\end{equation*}%
Indeed we will apply Lemma \ref{l:6.1} with $\hat{Q}_{0}=\hat{S}$, $Q_{0}=K$, 
$\{Q_{r}\}_{r\geq 1}=\mathcal{G}_{1}$ and $\displaystyle\{J_{r,s}\}_{s\geq 1}=\bigcup
_{S^{\prime }\in \mathcal{S}_{t}(S^{\prime })}\mathcal{S}(S^{\prime })$.

When $K\neq \hat{S}$ we have 
\begin{align*}
\sum\limits_{S\in \mathcal{S}:\pi _{\mathcal{D}^{u}}^{1}(S)\subset K}\alpha
_{t}(S)& =\sum\limits_{S\in \mathcal{G}_{1}}\alpha
_{t}(S)+\sum\limits_{k=1}^{\infty }\sum\limits_{S\in \mathcal{G}%
_{k+1}}\alpha _{t}(S)\lesssim \delta ^{-\sigma_{0} t}u(K)+\delta ^{-\sigma_{0} t}%
\mathcal{V}\sum\limits_{k=1}^{\infty }\sum\limits_{S\in 
\mathcal{G}_{k}}u(S) \\
& \lesssim \delta ^{\gamma t}\mathcal{V}^{2}u(K).
\end{align*}%
If $K=\hat{S}$ we have $\mathcal{G}_{0}=\{\hat{S}\}$ and we get the
estimate 
\begin{align*}
\sum\limits_{S\in \mathcal{S}:\pi _{\mathcal{D}^{u}}^{1}(S)\subset \hat{S}%
}\alpha _{t}(S)& =\sum\limits_{S\in \mathcal{G}_{1}}\alpha
_{t}(S)+\sum\limits_{k=1}^{\infty }\sum_{S\in \mathcal{G}_{k+1}}\alpha
_{t}(S) \\
& \lesssim \delta ^{-\sigma_{0} t}\mathcal{V}^{2}\sum\limits_{k=1}^{\infty }\sum\limits_{S\in \mathcal{G}_{k}}u(S) \\
& \lesssim \delta ^{-\sigma_{0} t}\mathcal{V}^{2}u(\hat{S}).
\end{align*}

The proof of Theorem \ref{thm off-dia} is complete.
\end{proof}

We need a Carleson measure estimate that is a common variant of \eqref{e:estimate2}
and \eqref{e:estimate4}. We define 
\begin{equation*}
\beta (S):=\Vert P_{S}^{v}T(u1_{\pi _{\mathcal{D}^{u}}^{1}(S)})\Vert
_{L^{2}(v)}^{2}.
\end{equation*}

\begin{theorem}
\label{t:bt} We have the following Carleson measure estimate 
\begin{equation*}
\sum\limits_{S\in \mathcal{S}:\pi _{\mathcal{D}^{u}}^{1}(S)\subset K}\beta
(S)\lesssim (\mathcal{T}^{2}+\mathcal{V}^{2})u(K).
\end{equation*}
\end{theorem}

\begin{proof}
Using the decomposition $\pi _{\mathcal{D}^{u}}^{1}(S)=S\cup \{\pi _{%
\mathcal{D}^{u}}^{1}(S)\setminus S\}$, we write $\beta (S)\leq 2(\beta
_{1}(S)+\beta _{2}(S))$ where 
\begin{equation*}
\beta _{1}(S):=\Vert P_{S}^{v}T(u1_{S})\Vert _{L^{2}(v)}^{2}\quad\textnormal{ and }\quad \beta _{2}(S):=\Vert P_{S}^{v}T(u1_{\pi _{\mathcal{D}^{u}}^{1}(S)\setminus
S})\Vert _{L^{2}(v)}^{2}.
\end{equation*}%
We have by the testing condition $\beta _{1}(S)\leq \mathcal{T}^{2}u(S)$, so now
by \eqref{e:estimate2}, we need only consider the Carleson measure norm of the
terms $\beta _{2}(S)$.

Now we will fix an cube $K$ of the form $K=\pi _{\mathcal{D}%
^{u}}^{1}(S_{0})$ for some $S_{0}\in \mathcal{S}$. Let $\mathcal{R}$ be the
maximal cubes of the form $\pi _{\mathcal{D}^{u}}^{1}(S)\subsetneq K$
and for $R\in \mathcal{R}$, let $\mathcal{S}(R)$ be all cubes $S\in 
\mathcal{S}$ with $S\subset R$ and $S$ is maximal. Now by using the definition of $\Tilde{\mathcal{P}}_{S}^{v}(g)$
and \eqref{e:maximalset}, we can estimate 
\begin{equation*}
\sum\limits_{R\in \mathcal{R}}\sum\limits_{S\in \mathcal{S}(R)}\beta
_{2}(S)\lesssim \sum\limits_{R\in \mathcal{R}}\sum\limits_{S\in \mathcal{S}%
(R)}\sum\limits_{J\in \mathcal{S}(S)}\Vert \tilde{P}_{J}^{v}T(u1_{\pi _{%
\mathcal{D}^{u}}^{1}(S)\setminus S})\Vert _{L^{2}(v)}^{2}\lesssim \mathcal{V}^{2}u(K).
\end{equation*}%
By careful arrangement of the collections $\mathcal{R},\mathcal{S}(R)$ and $%
\mathcal{S}(S)$ we have applied \eqref{e:estimate1} in the last step. Here we use
the same strategy as we used in the proof of Theorem \ref{t:estimate}.

We argue this inequality is enough to conclude the Theorem. Suppose that $%
S^{\prime }\in \mathcal{S}$, with $S^{\prime }\subset K$, but $S^{\prime }$
is not in any collection $\mathcal{S}(R)$ for $R\in \mathcal{R}$. It follows
that $S^{\prime }\subsetneq S$ for some $S\in \mathcal{S}(R)$ and $R\in 
\mathcal{R}$. This implies that the Carleson measure estimate \eqref{e:estimate2}
completes the proof.
\end{proof}

We collect one last Carleson measure estimate. Define 
\begin{equation*}
\gamma (S):=\Vert P_{S}^{v}T(u1_{\pi _{\mathcal{S}}^{1}(S)\setminus \pi _{%
\mathcal{D}^{u}}^{1}(S})\Vert _{L^{2}(v)}^{2}.
\end{equation*}

\begin{theorem}
\label{t:6.5} We have the estimate 
\begin{equation*}
\sum\limits_{S\in \mathcal{S}:\pi _{\mathcal{D}^{u}}^{1}(S)\subset K}\gamma
(S)\lesssim \mathcal{V}^{2}u(K).
\end{equation*}
\end{theorem}

\begin{proof}
We can take $K=\pi _{\mathcal{D}^{u}}^{1}(S_{0})$ for some $S_{0}\in 
\mathcal{S}$, and we can assume that $K\notin \mathcal{S}$ as otherwise we
are applying the $T$ to the zero function. We then repeat the argument as
in the previous proof. We use a similar construction for the proof here as in
Theorem \ref{t:bt}.
\end{proof}

\section{The Paraproduct Terms}
\label{s:Paraproduct}

We are going to prove bounds on the paraproduct term $A_2^4$ now. Before we
prove the bounds, we will reorganize the sum in \eqref{e:para} according to
the corona decomposition. We need to observe that for $S\in \mathcal{C}%
^v(S^{\prime })$ and $S \subset Q$, we need not have $Q \in \mathcal{C}%
^u(S^{\prime })$. It could be the case that $Q\in \mathcal{C}^u(\pi_{%
\mathcal{S}}^{t}(S^{\prime }))$ for some ancestor $\pi_{\mathcal{S}%
}^{t}(S^{\prime })$ of $S^{\prime }$. Remember the ancestor $\pi_{\mathcal{S}%
}^{t}(S^{\prime })$ is defined only for $1\leq t \leq \rho(S^{\prime })$.

Now we will split the sum into two parts $A_2^4 = A_1^5 +A_2^5$ where 
\begin{align}  \label{e:7.2}
A_1^5 &:= \sum\limits_{S^{\prime }\in \mathcal{S}} \sum\limits_{\substack{ %
(Q,S)\in \mathcal{P}(S^{\prime })  \\ S \in \mathcal{C}^v(S^{\prime })}} 
\mathbf{E}_{Q_S}^u(\Delta_Q^u f)\langle T(1_{S^{\prime }}u), \Delta_S^u
g\rangle_v; \\
A_2^5 &:= \sum\limits_{S^{\prime }\in \mathcal{S}\setminus
Q_0}\sum\limits_{t=1}^{\rho(S^{\prime })} \sum\limits_{\substack{ (Q,S)\in 
\mathcal{P}(\pi_{\mathcal{S}}^{t}(S^{\prime }))  \\ S\in \mathcal{C}%
^v(S^{\prime })}} \mathbf{E}_{Q_S}^u(\Delta_Q^u f)\langle T(1_{\pi_{\mathcal{%
S}}^{t}(S^{\prime })}u), \Delta_S^u g\rangle_v.
\end{align}
Observe that in $A_1^5$ we consider the case where both $Q$ and $S$ are
controlled by the same stopping cube. Whereas in $A_2^5$, $(Q,S) \in 
\mathcal{P}(\pi_{\mathcal{S}}^{t}(S^{\prime })) $, where $\pi_{\mathcal{S}%
}^{t}(S^{\prime })$ is $t$-fold parent of $S^{\prime }$ in the grid $%
\mathcal{S}$.

We will now show 
\begin{equation*}
|A_{1}^{5}|\lesssim (\mathcal{T}+\mathcal{V})\Vert f\Vert
_{L^{2}(u)}\Vert g\Vert _{L^{2}(v)}.
\end{equation*}%
We will have to further decompose $A_{2}^{5}$.

\subsection{The first Paraproduct Term $A_{1}^{5}$.}

Fix $S^{\prime }\in \mathcal{S}$ and $S\in \mathcal{C}^{v}(S^{\prime })$.
Observe that we have the following telescoping indentity:
\begin{equation}
\label{e:telescope}
\sum_{Q:(Q,S)\in \mathcal{P}(S^{\prime })}\mathbf{E}_{Q_{S}}^{u}(\Delta
_{Q}^{u}f)=\mathbf{E}_{Q_{S,\ast }}^{u}f-\mathbf{E}_{\pi _{_{\mathcal{D}%
}}(S^{\prime })}^{u}f.
\end{equation}%
Here $Q_{S,\ast }$ is the minimal member of $\mathcal{C}^{u}(S^{\prime })$
that contains $S$ and $l(S)<\delta ^{r}l(Q)$. As $S$ is good, such cubes
exist. So we have 
\begin{align}
A_{1}^{5}& =\sum\limits_{S^{\prime }\in \mathcal{S}}A_{1}^{5}(S^{\prime });
\label{e:7.6} \\
A_{1}^{5}(S^{\prime })& :=\sum\limits_{S\in \mathcal{C}^{v}(S^{\prime })}(%
\mathbf{E}_{Q_{S,\ast }}^{u}f-\mathbf{E}_{\pi _{\mathcal{D}}(S^{\prime
})}^{u}f)\langle T(1_{S^{\prime }}u),\Delta _{S}^{u}g\rangle _{v}.
\end{align}

We now give our first paraproduct estimate.

\begin{lemma}
\label{l:7.1} We have the following estimate 
\begin{equation*}
A_{1}^{5}(S^{\prime })\lesssim (\mathcal{T}+\mathcal{V})\Vert \mathcal{P}_{S^{\prime }}^{u}f-1_{S^{\prime }}\mathbf{E}_{\pi _{%
\mathcal{D}}(S^{\prime })}^{u}f\Vert _{L^{2}(u)}\Vert \mathcal{P}_{S^{\prime
}}^{v}g\Vert _{L^{2}(v)},\hspace{1cm}S^{\prime }\in \mathcal{S}.
\end{equation*}%
We have defined the projections appearing on the righthand side in Section \ref{s6}.
\end{lemma}

\begin{proof}
For $Q\in \mathcal{C}^{u}(S^{\prime })$, let us define 
\begin{equation*}
L_{Q}^{v}g:=\sum\limits_{S\in \mathcal{C}^{v}(S^{\prime }):Q_{S,\ast
}=Q}\Delta _{S}^{u}g.
\end{equation*}%
Now using Cauchy-Schwarz and the fact that $\mathbf{E}_{Q}^{u}f=\mathbf{E}%
_{Q}^{u}\mathcal{P}_{S^{\prime }}^{u}f$ and $L_{Q}^{v}g=L_{Q}^{v}\mathcal{P}%
_{S^{\prime }}^{v}g$, then we have by re-indexing

\begin{align*}
|A_{1}^{5}(S^{\prime })|& =\left\vert \sum\limits_{Q\in \mathcal{C}%
^{u}(S^{\prime })}(\mathbf{E}_{Q}^{u}\mathcal{P}_{S^{\prime }}^{u}f-\mathbf{E%
}_{\pi _{\mathcal{D}}(S^{\prime })}^{u}f)\langle T(1_{S^{\prime
}}u),L_{Q}^{v}\mathcal{P}_{S^{\prime }}^{v}g\rangle _{v}\right\vert \\
& \leq \left[ \sum\limits_{Q\in \mathcal{C}^{u}(S^{\prime })}|\mathbf{E}%
_{Q}^{u}\mathcal{P}_{S^{\prime }}^{u}f-\mathbf{E}_{\pi _{\mathcal{D}%
}(S^{\prime })}^{u}f|^{2}\Vert L_{Q}^{v}T(1_{S^{\prime }}u)\Vert
_{L^{2}(v)}^{2}\sum\limits_{Q\in \mathcal{C}^{u}(S^{\prime })}\Vert L_{Q}^{v}%
\mathcal{P}_{S^{\prime }}^{v}g\Vert _{L^{2}(v)}^{2}\right] ^{\frac{1}{2}} \\
& \leq \Vert \mathcal{P}_{S^{\prime }}^{v}g\Vert _{L^{2}(v)}\left[
\sum\limits_{Q\in \mathcal{C}^{u}(S^{\prime })}|\mathbf{E}_{Q}^{u}(\mathcal{P%
}_{S^{\prime }}^{u}f-\mathbf{E}_{\pi _{\mathcal{D}}(S^{\prime
})}^{u}f)|^{2}\Vert L_{Q}^{v}T(1_{S^{\prime }}u)\Vert _{L^{2}(v)}^{2}\right]
^{\frac{1}{2}}.
\end{align*}%
By the Carleson Embedding Theorem, Theorem \ref{t:2.5}, this last factor is at most $%
\Vert (\mathcal{P}_{S^{\prime }}^{u}f-\mathbf{E}_{\pi _{\mathcal{D}%
}(S^{\prime })}^{u}f)\Vert _{L^{2}(u)}$ times the Carleson measure norm of
the coefficients 
\begin{equation*}
\{\Vert L_{Q}^{v}T(1_{S^{\prime }}u)\Vert _{L^{2}(v)}^{2}:Q\in \mathcal{C}%
^{u}(S^{\prime })\}.
\end{equation*}%
Using \eqref{e:estimate3} in Theorem \ref{t:carltheorem1} we know this at most a constant
multiple of $\mathcal{T}+\mathcal{V}$, so the proof is
complete.
\end{proof}

Now using Lemma \ref{l:7.1} we get the desired estimate on $A_1^5$
\begin{align*}
|A_{1}^{5}|& \lesssim (\mathcal{T}+\mathcal{V})\sum\limits_{S^{\prime }\in \mathcal{S}}\Vert \mathcal{P}_{S^{\prime
}}^{u}f-1_{S^{\prime }}\mathbf{E}_{\pi _{\mathcal{D}}(S^{\prime
})}^{u}f\Vert _{L^{2}(u)}\Vert \mathcal{P}_{S^{\prime }}^{v}g\Vert
_{L^{2}(v)} \\
& \lesssim (\mathcal{T}+\mathcal{V})\left(
\sum\limits_{S^{\prime }\in \mathcal{S}}\Vert \mathcal{P}_{S^{\prime
}}^{u}f-1_{S^{\prime }}\mathbf{E}_{\pi _{\mathcal{D}}(S^{\prime
})}^{u}f\Vert _{L^{2}(u)}^{2}\sum\limits_{S^{\prime }\in \mathcal{S}}\Vert 
\mathcal{P}_{S^{\prime }}^{v}g\Vert _{L^{2}(v)}^{2}\right) ^{\frac{1}{2}} \\
& \lesssim (\mathcal{T}+\mathcal{V})\Vert f\Vert
_{L^{2}(u)}\Vert g\Vert _{L^{2}(v)}.
\end{align*}
That is because we have 
\begin{equation*}
\mathcal{P}_{S^{\prime }}^{u}f=\sum\limits_{Q\in \mathcal{C}^{u}(S^{\prime
})}\sum\limits_{\epsilon =1}^{M_{Q}-1}\langle f,h_{Q}^{\epsilon }\rangle
_{u}h_{Q}^{\epsilon }\quad \textnormal{ and }\quad \mathcal{P}_{S^{\prime
}}^{v}g=\sum\limits_{S\in \mathcal{C}^{v}(S^{\prime
})}\sum\limits_{k=1}^{M_{S}-1}\langle g,h_{S}^{k}\rangle _{v}h_{S}^{k},
\end{equation*}%
\begin{equation*}
\Vert \mathcal{P}_{S^{\prime }}^{v}g\Vert _{L^{2}(v)}\leq \Vert g\Vert
_{L^{2}(v)}\quad \textnormal{ and } \quad\hspace{0.2cm}\Vert \mathcal{P}_{S^{\prime
}}^{u}f\Vert _{L^{2}(u)}\leq \Vert f\Vert _{L^{2}(u)}.
\end{equation*}%
Also we have 
\begin{equation*}
\sum\limits_{S^{\prime }\in \mathcal{S}}\Vert \mathcal{P}_{S^{\prime
}}^{u}f-1_{S^{\prime }}\mathbf{E}_{\pi _{\mathcal{D}}(S^{\prime
})}^{u}f\Vert _{L^{2}(u)}^{2}\leq \sum\limits_{S^{\prime }\in \mathcal{S}%
}\Vert \mathcal{P}_{S^{\prime }}^{u}f\Vert
_{L^{2}(u)}^{2}+\sum\limits_{S^{\prime }\in \mathcal{S}}\Vert 1_{S^{\prime }}%
\mathbf{E}_{\pi _{\mathcal{D}}(S^{\prime })}^{u}f\Vert _{L^{2}(u)}^{2}
\end{equation*}%
and
\begin{align*}
\sum\limits_{S^{\prime }\in \mathcal{S}}\Vert 1_{S^{\prime }}\mathbf{E}_{\pi
_{\mathcal{D}}(S^{\prime })}^{u}f\Vert _{L^{2}(u)}^{2}&
=\sum\limits_{S^{\prime }\in \mathcal{S}}u(S^{\prime })|\mathbf{E}_{\pi _{%
\mathcal{D}}(S^{\prime })}^{u}f|^{2} \\
& \leq \sum\limits_{S^{\prime }\in \mathcal{S}}u(S^{\prime })(\mathbf{E}%
_{\pi _{S^{\prime }}}^{u}|f|)^{2} \\
& \lesssim \Vert M_{u}f\Vert _{L^{2}(u)}^{2}\\
&\lesssim \Vert f\Vert
_{L^{2}(u)}^{2}.
\end{align*}%
Here we are using \eqref{e:estimate2} to conclude that the maximal function $M_{u}$
dominates the sum above. This completes the proof.

\subsection{Telescoping Arguments}

We will now use telescoping arguments as in \eqref{e:telescope} for $A_{2}^{5}$.
For $S^{\prime }\in \mathcal{S}\setminus \{Q_{0}\}$, fixing a $S\in \mathcal{C}%
^{v}(S^{\prime })$ then summing over $Q$, we get 
\begin{align}
A_{2}^{5}(S^{\prime })& :=\sum\limits_{t=1}^{\rho (S^{\prime })}\sum\limits 
_{\substack{ (Q,S)\in \mathcal{P}(\pi _{\mathcal{S}}^{t}(S^{\prime }))  \\ %
S\in \mathcal{C}^{v}(S^{\prime })}}\mathbf{E}_{Q_{S}}^{u}(\Delta
_{Q}^{u}f)\langle T(1_{\pi _{\mathcal{S}}^{t}(S^{\prime })}u),\Delta
_{S}^{u}g\rangle _{v}  \notag  \label{e:7.10} \\
& =\sum\limits_{t=1}^{\rho (S^{\prime })}\sum\limits_{S\in \mathcal{C}%
^{v}(S^{\prime })}(\mathbf{E}_{\pi _{\mathcal{D}^{u}}^{2}(\pi _{\mathcal{S}%
}^{t-1}(S^{\prime }))}^{u}(f)-\mathbf{E}_{\pi _{\mathcal{D}^{u}}^{1}(\pi _{%
\mathcal{S}}^{t}(S^{\prime }))}^{u}(f))\langle T(1_{\pi _{\mathcal{S}%
}^{t}(S^{\prime })}u),\Delta _{S}^{u}g\rangle _{v}.
\end{align}%
This is because with $S\in \mathcal{C}^{v}(S^{\prime })$ fixed, the sum over 
$Q$ such that $(Q,S)\in \mathcal{P}(\pi _{\mathcal{S}}^{t}(S^{\prime }))$ is
a function of $S^{\prime }$ and $t$. The smallest cube that contributes to
the sum is $\pi _{\mathcal{D}^{u}}^{2}(\pi _{\mathcal{S}}^{t-1}(S^{\prime
})) $ which is the second parent of $\pi _{\mathcal{S}}^{t-1}(S^{\prime })$
and the largest cube that contributes to the sum is $\pi _{\mathcal{D}%
^{u}}^{1}(\pi _{\mathcal{S}}^{t}(S^{\prime }))$. Observe the sum over $S$ is
independent of the sum over $t$ in \eqref{e:7.10}. Below we will further
decompose $A_{2}^{5}(S^{\prime })$ by adding and subtracting a cancellative
term
\begin{align*}
A_{2}^{5}(S^{\prime })& =\sum\limits_{t=1}^{\rho (S^{\prime })}(%
\mathbf{E}_{\pi _{\mathcal{D}^{u}}^{2}(\pi _{\mathcal{S}}^{t-1}(S^{\prime
}))}^{u}(f)-\mathbf{E}_{\pi _{\mathcal{D}^{u}}^{1}(\pi _{\mathcal{S}%
}^{t-1}(S^{\prime }))}^{u}(f) \\
& +\mathbf{E}_{\pi _{\mathcal{D}^{u}}^{1}(\pi _{\mathcal{S}}^{t-1}(S^{\prime
}))}^{u}(f)-\mathbf{E}_{\pi _{\mathcal{D}^{u}}^{1}(\pi _{\mathcal{S}%
}^{t}(S^{\prime }))}^{u}(f))\langle T(1_{\pi _{\mathcal{S}%
}^{t}(S^{\prime })}u),\Delta _{S}^{u}g\rangle _{v} \\
& = A_{1}^{6}(S^{\prime })+ A_{22}^{5}(S^{\prime }),
\end{align*}%
where we have defined: 
\begin{align}
A_{1}^{6}(S^{\prime })& :=\sum\limits_{t=1}^{\rho (S^{\prime })}(%
\mathbf{E}_{\pi _{\mathcal{D}^{u}}^{2}(\pi _{\mathcal{S}}^{t-1}(S^{\prime
}))}^{u}(f)-\mathbf{E}_{\pi _{\mathcal{D}^{u}}^{1}(\pi _{\mathcal{S}%
}^{t-1}(S^{\prime }))}^{u}(f))\langle T(1_{\pi _{\mathcal{S}%
}^{t}(S^{\prime })}u),\Delta _{S}^{u}g\rangle _{v},
\label{e:7.12} \\
A_{22}^{5}(S^{\prime })& :=\sum\limits_{t=1}^{\rho (S^{\prime })}(\mathbf{E}_{\pi
_{\mathcal{D}^{u}}^{1}(\pi _{\mathcal{S}}^{t-1}(S^{\prime }))}^{u}(f)-%
\mathbf{E}_{\pi _{\mathcal{D}^{u}}^{1}(\pi _{\mathcal{S}}^{t}(S^{\prime
}))}^{u}(f))\langle T(1_{\pi _{\mathcal{S}%
}^{t}(S^{\prime })}u),\Delta _{S}^{u}g\rangle _{v}.
\end{align}%
Observe here that $ A_{22}^{5}$ is a telescoping term in itself, so
we can sum by parts to get the following 
\begin{equation*}
A_{22}^{5}(S^{\prime })=\mathbf{E}_{\pi _{\mathcal{D}%
^{u}}^{1}(S^{\prime })}^{u}(f)\langle T(1_{\pi _{\mathcal{S}%
}^{t}(S^{\prime })}u),\Delta _{S}^{u}g\rangle _{v} +\sum\limits_{t=1}^{\rho (S^{\prime })}\mathbf{E}_{\pi _{\mathcal{D}%
^{u}}^{1}(\pi _{\mathcal{S}}^{t}(S^{\prime }))}^{u}(f)T(1_{\pi _{\mathcal{S}}^{t+1}(S^{\prime })\setminus \pi _{%
\mathcal{S}}^{t}(S^{\prime })}u),\mathcal{P}_{S^{\prime }}^{v}g\rangle _{v} = A_{2}^{6}(S^{\prime}) + A_{3}^{6}(S^{\prime}).
\end{equation*}%
In the sum above for the missing term $\mathbf{E}_{\pi_{\mathcal{S}}^{\rho(S')}(S')}^u f = \mathbf{E}_{Q_0}^u f$ , where $Q_{0}$ is
the largest cube that was fixed, we are going to assume the expectation
is zero.

Combining the steps above we can now decompose $A_{2}^{5}$ as 
\begin{equation*}
A_{2}^{5}=A_{1}^{6}+A_{2}^{6}+A_{3}^{6}\textnormal{ where }%
A_{i}^{6}=\sum\limits_{S^{\prime }\in \mathcal{S}\setminus
Q_{0}}A_{i}^{6}(S^{\prime })\textnormal{ for }i=1,2,3,
\end{equation*}%
and 
\begin{align}
A_{1}^{6}(S^{\prime })& :=\sum\limits_{t=1}^{\rho (S^{\prime })}(\mathbf{E}%
_{\pi _{\mathcal{D}^{u}}^{2}(\pi _{\mathcal{S}}^{t-1}(S^{\prime }))}^{u}(f)-%
\mathbf{E}_{\pi _{\mathcal{D}^{u}}^{1}(\pi _{\mathcal{S}}^{t-1}(S^{\prime
}))}^{u}(f))\langle T(1_{\pi _{\mathcal{S}}^{t}(S^{\prime })}u),\mathcal{P}%
_{S^{\prime }}^{v}g\rangle _{v};  \label{e:7.16} \\
A_{2}^{6}(S^{\prime })& :=(\mathbf{E}_{\pi _{\mathcal{D}^{u}}^{1}(S^{\prime
})}^{u}f)\langle T(1_{\pi _{\mathcal{S}}^{t}(S^{\prime })}u),\mathcal{P}%
_{S^{\prime }}^{v}g\rangle _{v}; \\
A_{3}^{6}(S^{\prime })& :=\sum\limits_{t=1}^{\rho (S^{\prime })}(\mathbf{E}%
_{\pi _{\mathcal{D}^{u}}^{1}(\pi _{\mathcal{S}}^{t}(S^{\prime
}))}^{u}f)\langle T(1_{\pi _{\mathcal{S}}^{t+1}(S^{\prime })\setminus \pi _{%
\mathcal{S}}^{t}(S^{\prime })}u),\mathcal{P}_{S^{\prime }}^{v}g\rangle _{v}.
\end{align}%
Observe the expression $A_{1}^{6}$ has cancellative terms in both $f$ and $g$%
, so it is not a paraproduct term, while $A_{2}^{6}$ is a paraproduct
similar to $A_{1}^{5}$. The third term is also a paraproduct term. We will
now prove the following estimates for each of these terms: 
\begin{align}
|A_{1}^{6}|& \lesssim (\mathcal{T}+\mathcal{V})\Vert
f\Vert _{L^{2}(u)}\Vert g\Vert _{L^{2}(v)};  \label{e:a16} \\
|A_{2}^{6}|& \lesssim (\mathcal{T}+\mathcal{V})\Vert
f\Vert _{L^{2}(u)}\Vert g\Vert _{L^{2}(v)};\label{e:a26} \\
|A_{3}^{6}|& \lesssim \mathcal{V}\Vert f\Vert
_{L^{2}(u)}\Vert g\Vert _{L^{2}(v)}.\label{e:a36}
\end{align}%
We will begin the proof of these estimates above starting with the term $%
A_{3}^{6}$.

\subsection{The Paraproduct Term $A_{3}^{6}$}

Let us fix $t$ and define 
\begin{align*}
A_{3}^{6}(S^{\prime },t)& :=(\mathbf{E}_{\pi _{\mathcal{D}^{u}}^{1}(\pi _{%
\mathcal{S}}^{t}(S^{\prime }))}^{u}f)\langle T(1_{\pi _{\mathcal{S}%
}^{t+1}(S^{\prime })\setminus \pi _{\mathcal{S}}^{t}(S^{\prime })}u),%
\mathcal{P}_{S^{\prime }}^{v}g\rangle _{v},\hspace{1cm}S^{\prime }\in 
\mathcal{S},\rho (S^{\prime })\geq t; \\
A_{3}^{6}(t)& :=\sum\limits_{S^{\prime }\in \mathcal{S}:\rho (S^{\prime
})\geq t}A_{3}^{6}(S^{\prime },t).
\end{align*}%
We can see that the $t$-fold parent of $S^{\prime }$ is defined by imposing
the restriction $\rho (S^{\prime })\geq t$. We want to show 
\begin{equation*}
|A_{3}^{6}(t)|\lesssim \delta ^{\sigma t}\mathcal{V}\Vert
f\Vert _{L^{2}(u)}\Vert g\Vert _{L^{2}(v)}\hspace{1cm}t\geq 1.
\end{equation*}%
Here the constant $\sigma =-2\sigma_{0} >0$, hence we get \eqref{e:a36} when we
will sum over $t\geq 1$.

Since the $\mathcal{P}_{S^{\prime }}^{v}$ are orthogonal, we get the
following using Cauchy-Schwarz 
\begin{equation*}
|A_{3}^{6}(t)|\leq \Vert g\Vert _{L^{2}(v)}\left[ \sum\limits_{S^{\prime
}\in \mathcal{S}:\rho (S^{\prime })\geq t}|\mathbf{E}_{\pi _{\mathcal{D}%
^{u}}^{1}(\pi _{\mathcal{S}}^{t}(S^{\prime }))}^{u}f|^{2}\Vert \mathcal{P}%
_{S^{\prime }}^{v}T(1_{\pi _{\mathcal{S}}^{t+1}(S^{\prime })\setminus \pi _{%
\mathcal{S}}^{t}(S^{\prime })}u)\Vert _{L^{2}(v)}^{2}\right] ^{\frac{1}{2}}.
\end{equation*}

Using the definition of $\alpha_{t}(S)$ given in Section \ref{s6}, the sum above is 
\begin{equation*}
\left[ \sum_{S^{\prime }\in \mathcal{S}}\alpha _{t}(S^{\prime })|\mathbf{E}%
_{\pi _{\mathcal{D}^{u}}^{1}(S^{\prime })}^{u}f|^{2}\right] ^{\frac{1}{2}}.
\end{equation*}%
Now using Theorem \ref{t:estimate} we have our desired estimate on $A_3^6(t)$ as the
Carleson measure norm of the coefficients $\{\alpha _{t}(S^{\prime
}):S^{\prime }\in \mathcal{S}\}$ is at most $C\delta ^{\sigma t}\mathcal{V}$.

\subsection{The paraproduct term $A_{2}^{6}$}

We have $\pi _{\mathcal{D}^{u}}^{1}(S^{\prime })\subset \pi _{\mathcal{S}%
}^{1}(S^{\prime })$, so we will now decompose term $A_{2}^{6}$ into two
terms by writing $\pi _{\mathcal{S}}^{1}(S^{\prime })=\pi _{\mathcal{D}%
^{u}}^{1}(S^{\prime })\cup \{\pi _{\mathcal{S}}^{1}(S^{\prime })\setminus
\pi _{\mathcal{D}^{u}}^{1}(S^{\prime })\}$ to give us, 
\begin{align}
|A_{1}^{7}|& :=\left\vert \sum\limits_{S^{\prime }\in \mathcal{S}}(\mathbf{E}%
_{\pi _{\mathcal{D}^{u}}^{1}(S^{\prime })}^{u}f)\langle T(1_{\pi _{\mathcal{D%
}^{u}}^{1}(S^{\prime })}u),\mathcal{P}_{S^{\prime }}^{v}g\rangle
_{v}\right\vert \lesssim (\mathcal{T}+\mathcal{V})\Vert
f\Vert _{L^{2}(u)}\Vert g\Vert _{L^{2}(v)};  \label{e:7.23} \\
|A_{2}^{7}|& :=\left\vert \sum\limits_{S^{\prime }\in \mathcal{S}}(\mathbf{E}%
_{\pi _{\mathcal{D}^{u}}^{1}(S^{\prime })}^{u}f)\langle T(1_{\pi _{\mathcal{S%
}}^{1}(S^{\prime })\setminus \pi _{\mathcal{D}^{u}}^{1}(S^{\prime })}u),%
\mathcal{P}_{S^{\prime }}^{v}g\rangle _{v}\right\vert \lesssim \mathcal{V}\Vert f\Vert _{L^{2}(u)}\Vert g\Vert _{L^{2}(v)}.
\end{align}%
Using these we get \eqref{e:a26}. Now we will prove these inequalities. Remembering
the definition of $\beta(S)$, we can now estimate%
\begin{equation*}
|\langle T(1_{\pi _{\mathcal{D}^{u}}^{1}(S^{\prime })}u),\mathcal{P}%
_{S^{\prime }}^{v}g\rangle _{v}|=|\langle \mathcal{P}_{S^{\prime
}}^{v}T(1_{\pi _{\mathcal{D}^{u}}^{1}(S^{\prime })}u),\mathcal{P}_{S^{\prime
}}^{v}g\rangle _{v}|\leq \beta (S^{\prime})^{\frac{1}{2}}\Vert \mathcal{P}%
_{S^{\prime }}^{v}g\Vert _{L^{2}(v)}.
\end{equation*}

As the projections are mutually orthogonal we have, summing over $S^{\prime
} $ 
\begin{equation*}
|A_{1}^{7}|\leq \left[ \sum\limits_{S^{\prime }\in \mathcal{S}}\beta
(S^{\prime })|(\mathbf{E}_{\pi _{\mathcal{D}^{u}}^{1}(S^{\prime
})}^{u}f)|^{2}\right] ^{\frac{1}{2}}\Vert g\Vert _{L^{2}(v)}\lesssim (%
\mathcal{T}+\mathcal{V})\Vert f\Vert _{L^{2}(u)}\Vert
g\Vert _{L^{2}(v)}.
\end{equation*}%
We have used the estimate on $\beta(S)$ in Theorem \ref{t:bt} to get the estimate in
the last step.

We can use a similar approach to prove \eqref{e:7.23}. Using the definition of $\gamma(S)$, we have 
\begin{equation*}
|\langle T(1_{\pi _{\mathcal{S}}^{1}(S^{\prime })\setminus \pi _{\mathcal{D}%
^{u}}^{1}(S^{\prime })}u),\mathcal{P}_{S^{\prime }}^{v}g\rangle
_{v}|=|\langle \mathcal{P}_{S^{\prime }}^{v}T(1_{\pi _{\mathcal{S}%
}^{1}(S^{\prime })\setminus \pi _{\mathcal{D}^{u}}^{1}(S^{\prime })}u),%
\mathcal{P}_{S^{\prime }}^{v}g\rangle _{v}|\leq \gamma (S^{\prime})^\frac{1}{2}\Vert \mathcal{P}_{S^{\prime }}^{v}g\Vert _{L^{2}(v)}.
\end{equation*}%
So we have the following estimate using Theorem \ref{t:6.5}, 
\begin{equation*}
|A_{2}^{7}|\leq \left[ \sum\limits_{S^{\prime }\in \mathcal{S}}\gamma
(S^{\prime })|(\mathbf{E}_{\pi _{\mathcal{D}^{u}}^{1}(S^{\prime
})}^{u}f)|^{2}\right] ^{\frac{1}{2}}\Vert g\Vert _{L^{2}(v)}\lesssim 
\mathcal{V}\Vert f\Vert _{L^{2}(u)}\Vert g\Vert
_{L^{2}(v)}.
\end{equation*}

\subsection{The term $A_{1}^{6}$}

Observe in the definition of $A_{1}^{6}(S^{\prime })$ we can write the
following for the difference of expectations, 
\begin{equation*}
\mathbf{E}_{\pi _{\mathcal{D}^{u}}^{2}(\pi _{\mathcal{S}}^{t-1}(S^{\prime
}))}^{u}(f)-\mathbf{E}_{\pi _{\mathcal{D}^{u}}^{1}(\pi _{\mathcal{S}%
}^{t-1}(S^{\prime }))}^{u}(f)=-\mathbf{E}_{\pi _{\mathcal{D}^{u}}^{1}(\pi _{%
\mathcal{S}}^{t-1}(S^{\prime }))}^{u}\Delta _{\pi _{\mathcal{D}^{u}}^{2}(\pi
_{\mathcal{S}}^{t-1}(S^{\prime }))}^{u}f,\quad \pi _{\mathcal{S}%
}^{t-1}(S^{\prime })\in \mathcal{S}.
\end{equation*}%
By re-indexing the sum $A_1^6(S')$ defined above we get 
\begin{equation*}
A_{1}^{6}(S^{\prime })=\sum\limits_{t=1}^{\rho (S^{\prime })}(\mathbf{E}%
_{\pi _{\mathcal{D}^{u}}^{1}(\pi _{\mathcal{S}}^{t-1}(S^{\prime
}))}^{u}\Delta _{\pi _{\mathcal{D}^{u}}^{2}(\pi _{\mathcal{S}%
}^{t-1}(S^{\prime }))}^{u}f)\langle T(1_{\pi _{\mathcal{S}}^{t}(S^{\prime
})}u),\mathcal{P}_{S^{\prime }}^{v}g\rangle _{v}=A_{3}^{7}(S^{\prime
})+A_{4}^{7}(S^{\prime });
\end{equation*}%
where 
\begin{align}
A_{3}^{7}(S^{\prime })& :=\sum_{t=1}^{\rho (S^{\prime })}(\mathbf{E}_{\pi _{%
\mathcal{D}^{u}}^{1}(\pi _{\mathcal{S}}^{t-1}(S^{\prime }))}^{u}\Delta _{\pi
_{\mathcal{D}^{u}}^{2}(\pi _{\mathcal{S}}^{t-1}(S^{\prime }))}^{u}f)\langle
T(1_{\pi _{\mathcal{S}}^{t}(S^{\prime })\setminus \pi _{\mathcal{S}%
}^{t-1}(S^{\prime })}u),\mathcal{P}_{S^{\prime }}^{v}g\rangle _{v},
\label{e:ortho} \\
A_{4}^{7}(S^{\prime })& :=\sum_{t=1}^{\rho (S^{\prime })}(\mathbf{E}_{\pi _{%
\mathcal{D}^{u}}^{1}(\pi _{\mathcal{S}}^{t-1}(S^{\prime }))}^{u}\Delta _{\pi
_{\mathcal{D}^{u}}^{2}(\pi _{\mathcal{S}}^{t-1}(S^{\prime }))}^{u}f)\langle
T(1_{\pi _{\mathcal{S}}^{t-1}(S^{\prime })}u),\mathcal{P}_{S^{\prime
}}^{v}g\rangle _{v}.
\end{align}%
We will now show that 
\begin{align}
\left\vert \sum\limits_{S^{\prime}\in \mathcal{S}\setminus
\{S_{0}\}}A_{3}^{7}(S^{\prime })\right\vert & \lesssim \mathcal{V}\Vert f\Vert _{L^{2}(u)}\Vert g\Vert _{L^{2}(v)};  \label{e:a37}
\\
\left\vert \sum_{S^{\prime}\in \mathcal{S}\setminus \{S_{0}\}}A_{4}^{7}(S^{\prime
})\right\vert & \lesssim \mathcal{V}\Vert f\Vert
_{L^{2}(u)}\Vert g\Vert _{L^{2}(v)}.\label{e:a47}
\end{align}%
We can prove \eqref{e:a37} using a similar approach as we did for %
\eqref{e:a36} as there is orthogonality present with Haar differences
applied to $f$ in \eqref{e:ortho}.

We will start the proof of \eqref{e:a47} by
re-indexing the sum. We have  
\begin{equation*}
A_{4}^{7}(S^{\prime },t):=(\mathbf{E}_{\pi _{\mathcal{D}^{u}}^{1}(S^{\prime
})}^{u}\Delta _{\pi _{\mathcal{D}^{u}}^{2}(S^{\prime })}^{u}f)\sum\limits 
_{\substack{ J\in \mathcal{S}  \\ \pi ^{t-1}(J)=S^{\prime }}}\langle
T(1_{S^{\prime }}u),\mathcal{P}_{S^{\prime }}^{v}g\rangle _{v}
\end{equation*}%
and we will prove that: 
\begin{equation}
\label{e:finalestimate}
\left\vert \sum\limits_{S^{\prime}\in \mathcal{S}\setminus
\{S_{0}\}}A_{3}^{7}(S^{\prime })\right\vert =
\left\vert \sum_{S^{\prime }\in \mathcal{S}}A_{4}^{7}(S^{\prime
},t)\right\vert \lesssim \delta ^{\frac{-\sigma_{0} t}{2}}\mathcal{V}\Vert f\Vert _{L^{2}(u)}\Vert g\Vert _{L^{2}(v)}\hspace{1cm}t\geq
1.
\end{equation}%
(Here the decay in $t$ is slightly worse in comparison to the previous
estimates.) We will exploit the implicit orthogonality in the sum above.
Note that we have 
\begin{align*}
\sum_{S^{\prime }\in \mathcal{S}}\sum\limits_{i=1}^{M_{\pi _{\mathcal{D}%
^{u}}^{1}(S^{\prime })}-1}|\langle f,h_{\pi _{\mathcal{S}}^{2}(S^{\prime
}))}^{i}\rangle _{v}|^{2}& \leq \Vert f\Vert _{L^{2}(u)}^{2}, \\
\sum\limits_{S^{\prime }\in \mathcal{S}}\left\Vert \sum\limits_{J\in 
\mathcal{S}:\pi ^{t-1}(J)=S^{\prime }}\mathcal{P}_{J}^{v}g\right\Vert
_{L^{2}(v)}^{2}& \leq \Vert g\Vert _{L^{2}(v)}^{2},
\end{align*}%
and combining these facts we get \eqref{e:finalestimate} from the estimate 
\begin{equation}
\label{e:projectionestimate}
\left\Vert \sum\limits_{J\in \mathcal{S}:\pi ^{t-1}(J)=S^{\prime }}\mathcal{P%
}_{J}^{v}T(1_{\pi _{\mathcal{S}}^{t-1}(S^{\prime })}u)\right\Vert
_{L^{2}(v)}^{2}\lesssim \delta ^{-\sigma_{0} t}(\mathcal{T}^{2}+\mathcal{V}^{2})u(\pi _{\mathcal{D}^{u}}^{1}(S^{\prime }))\hspace{%
0.2cm}S^{\prime }\in \mathcal{S},t\geq 1.
\end{equation}

Let us now prove \eqref{e:projectionestimate}. We use the geometric decay in %
\eqref{e:estimate2} and apply hypothesis \eqref{e:nostop}. Fix a $S^{\prime }\in 
\mathcal{S}$ and an integer $w:=\frac{t-1}{2}$. Let us denote by $%
S_{w}^{^{\prime }}$, all the cubes $J\in \mathcal{S}$ with $\pi _{\mathcal{S}%
}^{w}(J)=S^{\prime }$. We have 
\begin{equation*}
\left\Vert \sum\limits_{J\in \mathcal{S}:\pi ^{t-1}(J)=S^{\prime }}\mathcal{P%
}_{J}^{v}T(1_{\pi _{\mathcal{S}}^{t-1}(S^{\prime })}u)\right\Vert
_{L^{2}(v)}^{2}=\sum\limits_{J\in \mathcal{S}_{w}}B(J),
\end{equation*}%
where 
\begin{equation*}
B(J):=\left\Vert \sum\limits_{J^{\prime }\in \mathcal{S}:\pi
^{t-1-w}(J^{\prime })=S^{\prime }}\mathcal{P}_{J^{\prime }}^{v}T(1_{\pi _{%
\mathcal{S}}^{t-1}(S^{\prime })}u)\right\Vert _{L^{2}(v)}^{2}.
\end{equation*}%
We now decompose $B(J)$ as $B(J)=B_{1}(J)+B_{2}(J)$, where
\begin{align*}
B_{1}(J)& :=\left\Vert \sum\limits_{J^{\prime }\in \mathcal{S}:\pi
^{t-1-w}(J^{\prime })=S^{\prime }}\mathcal{P}_{J^{\prime }}^{v}T(1_{\pi _{%
\mathcal{S}}^{t-1}(S^{\prime })\setminus J}u)\right\Vert _{L^{2}(v)}^{2}, \\
B_{2}(J)& :=\left\Vert \sum\limits_{J^{\prime }\in \mathcal{S}:\pi
^{t-1-w}(J^{\prime })=S^{\prime }}\mathcal{P}_{J^{\prime
}}^{v}T(1_{J}u)\right\Vert _{L^{2}(v)}^{2}.
\end{align*}%
Using the testing condition we have 
\begin{equation*}
\sum\limits_{J\in \mathcal{S}_{w}}B_{2}(J)\leq \mathcal{T}%
^{2}\sum\limits_{J\in \mathcal{S}_{w}}u(J)\leq \delta ^{\frac{w}{2}}\mathcal{%
T}^{2}u(\pi _{\mathcal{D}^{u}}^{1}(S^{\prime })).
\end{equation*}%
Here we used the Carleson measure property of $u$ on stopping cubes \eqref{e:estimate2}
to deduce the last line. Now using the notation of \eqref{e:maximalset} and
applying \eqref{e:estimate1} we can see that 
\begin{equation*}
\sum\limits_{J\in \mathcal{S}_{w}}B_{1}(J)=\sum\limits_{J\in \mathcal{S}%
_{w}}\sum\limits_{J^{\prime }\in \mathcal{S}:\pi ^{t-1-w}(J^{\prime
})=J}\sum\limits_{I\in \mathcal{J}(J^{\prime })}\Vert \mathcal{P}%
_{I}^{v}T(1_{\pi _{\mathcal{S}}^{t-1}(S^{\prime })\setminus J}u)\Vert
_{L^{2}(v)}^{2}\leq \mathcal{V}^{2}\delta ^{\frac{-\sigma_{0} t%
}{2}}u(\pi _{\mathcal{D}^{u}}^{1}(S^{\prime })).
\end{equation*}%
This completes the proof of \eqref{e:projectionestimate}.

\section{Appendix: Hilbert Space Valued Operators}
\label{s:Appendix}

Here we make precise the definitions arising in the setting of weighted norm
inequalities for Hilbert space valued singular integrals, beginning with a
Calder\'{o}n--Zygmund kernel. We define a standard $B\left( \mathcal{H}_{1},%
\mathcal{H}_{2}\right) $-valued Calder\'on--Zygmund kernel $\mathfrak K(x,y)$ to
be a function $\mathfrak  K:X\times X\rightarrow B\left( \mathcal{H}_{1},\mathcal{H}%
_{2}\right) $ satisfying the following fractional size and smoothness
conditions of order $\delta $ for some $\delta >0$: For $x\neq y$,%
\begin{eqnarray}
&&\left\vert\mathfrak  K\left( x,y\right) \right\vert _{B\left( \mathcal{H}_{1},%
\mathcal{H}_{2}\right) }\leq \frac{C_{CZ}}{V\left( x,y\right)} ,
\label{sizeandsmoothness'} \\
&&\left\vert \nabla\mathfrak  K\left( x,y\right) -\nabla\mathfrak  K\left( x^{\prime
},y\right) \right\vert _{B\left( X\times X,B\left( \mathcal{H}_{1},\mathcal{H%
}_{2}\right) \right) }\leq C_{CZ}\left( \frac{d\left( x,x^{\prime }\right) }{%
d\left( x,y\right) }\right) ^{\delta }\frac{1}{V\left( x,y\right) }, 
\ \ \ \ \ \frac{d\left( x,x^{\prime }\right) }{d\left( x,y\right) }%
\leq \frac{1}{2A_0},  \notag
\end{eqnarray}%
and the last inequality also holds for the adjoint kernel in which $x$ and $%
y $ are interchanged.

We now turn to a precise definition of the weighted norm inequality%
\begin{equation}
\left\Vert T_{\sigma }f\right\Vert _{L_{\mathcal{H}_{2}}^{2}\left(
\omega \right) }\leq \mathcal{N}\left\Vert
f\right\Vert _{L_{\mathcal{H}_{1}}^{2}\left( \sigma \right) },\ \ \ \ \ f\in
L^{2}\left( \sigma \right) ,  \label{two weight'}
\end{equation}%
where $\sigma $ and $\omega $ are locally finite positive Borel measures on $%
X$, and $L_{\mathcal{H}_{1}}^{2}\left( \sigma \right) $ is the Hilbert space
consisting of those functions $f:X\rightarrow \mathcal{H}_{1}$ for which 
\begin{equation*}
\left\Vert f\right\Vert _{L_{\mathcal{H}_{1}}^{2}\left( \sigma \right)
}:= \sqrt{\int_{X}\left\vert f\left( x\right) \right\vert _{\mathcal{H}%
_{1}}^{2}d\sigma \left( x\right) }<\infty ,
\end{equation*}%
equipped with the usual inner product. A similar definition holds for $L_{%
\mathcal{H}_{2}}^{2}\left( \omega \right) $. For a precise definition of (%
\ref{two weight'}) we suppose that $K$ is a standard $B\left( 
\mathcal{H}_{1},\mathcal{H}_{2}\right) $-valued Calder\'on--Zygmund
kernel, and we introduce a family $\left\{ \eta _{\delta ,R}^{\alpha
}\right\} _{0<\delta <R<\infty }$ of nonnegative functions on $\left[
0,\infty \right) $ so that the truncated kernels $\mathfrak K_{\delta ,R}\left(
x,y\right) :=\eta _{\delta ,R}\left( d\left( x,y\right) \right)\mathfrak  K\left(
x,y\right) $ are bounded with compact support for fixed $x$ or $y$. Then the
truncated operators 
\begin{equation*}
T_{\sigma ,\delta ,R}f\left( x\right) := \int_{X}\mathfrak K_{\delta ,R}\left(
x,y\right) f\left( y\right) d\sigma \left( y\right) ,\ \ \ \ \ x\in X,
\end{equation*}%
are pointwise well-defined, and we will refer to the pair $\left(\mathfrak  K,\left\{
\eta _{\delta ,R}\right\} _{0<\delta <R<\infty }\right) $ as a singular
integral operator, which we typically denote by $T$, suppressing the
dependence on the truncations.

\begin{definition}
We say that a singular integral operator $T=\left(\mathfrak  K,\left\{ \eta _{\delta
,R}\right\} _{0<\delta <R<\infty }\right) $ satisfies the norm inequality \eqref{two weight'} provided%
\begin{equation*}
\left\Vert T_{\sigma ,\delta ,R}f\right\Vert _{L_{\mathcal{H}_{2}}^{2}\left(
\omega \right) }\leq \mathcal{N}\left\Vert f\right\Vert _{L_{%
\mathcal{H}_{1}}^{2}\left( \sigma \right) },\ \ \ \ \ f\in L^{2}\left(
\sigma \right) ,0<\delta <R<\infty .
\end{equation*}
\end{definition}

It turns out that, in the presence of the Muckenhoupt conditions, the norm
inequality (\ref{two weight'}) is essentially independent of the choice of
truncations used, which justifies suppressing the dependance on the truncations.

The following cube testing conditions, dual to each other and referred to as 
$T1$ conditions, are necessary for the boundedness of $T$ from $%
L^{2}\left( \sigma \right) $ to $L^{2}\left( \omega \right) $,%
\begin{eqnarray*}
\mathcal{T}^{2} &:= &\sup_{Q\in \mathcal{P}}\sup_{\mathbf{e}_{1}\in 
\func{unit}\mathcal{H}_{1}} \frac{1}{\sigma(Q)} 
\int_{Q}\left\vert T\left( \mathbf{1}_{Q}\mathbf{e}_{1}\sigma \right)
\right\vert _{\mathcal{H}_{2}}^{2}d\omega <\infty , \\
\left( \mathcal{T}^{\ast }\right) ^{2} &:= &\sup_{Q\in \mathcal{P}%
}\sup_{\mathbf{e}_{2}\in \func{unit}\mathcal{H}_{2}} \frac{1}{\omega(Q)}
\int_{Q}\left\vert T^{\ast }\left( \mathbf{1}_{Q}%
\mathbf{e}_{2}\omega \right) \right\vert _{\mathcal{H}_{1}}^{2}d\sigma
<\infty ,
\end{eqnarray*}%
and where we interpret the right sides as holding uniformly over all
truncations of $T$.

\subsection{Weighted Haar bases for $L_{\mathcal{H}}^{2}\left( \protect\mu %
\right) \label{Subsection Haar}$}

Now we turn to the definition of weighted Haar bases of $L_{\mathcal{H}%
}^{2}\left( \mu \right) $ where $\mathcal{H}$ is a separable Hilbert space
and $\mu $ is a locally finite positive Borel measure on $X$. We will use a
construction of a Haar basis for $L_{\mathcal{H}}^{2}\left( \mu \right) $ in 
$X$ that is adapted to a measure $\mu $ (c.f. \cite{NTV2} for the scalar
case). Given a dyadic cube $Q\in \mathcal{D}$, where $\mathcal{D}$ is a
dyadic grid of cubes from $\mathcal{P}^{n}$, let $\bigtriangleup _{\mathcal{H%
};Q}^{\mu }$ denote orthogonal projection onto the subspace $L_{\mathcal{H}%
;Q}^{2}\left( \mu \right) $ of $L_{\mathcal{H}}^{2}\left( \mu \right) $ that
consists of $\mathcal{H}$-linear combinations of the indicators of\ the
children $\mathcal{H}\left( Q\right) $ of $Q$ that have $\mu $-mean zero
over $Q$:%
\begin{equation*}
L_{\mathcal{H};Q}^{2}\left( \mu \right) := \left\{
f=\sum\limits_{Q^{\prime }\in \mathcal{H}\left( Q\right) }b_{Q^{\prime }}%
\mathbf{1}_{Q^{\prime }}:b_{Q^{\prime }}\in \mathcal{H},\int_{Q}fd\mu
=0\right\} ,
\end{equation*}%
where the expression $b_{Q^{\prime }}\mathbf{1}_{Q^{\prime }}$ refers to the 
$\mathcal{H}$-valued function $b_{Q^{\prime }}\mathbf{1}_{Q^{\prime }}\left(
x\right) =\left\{ 
\begin{array}{ccc}
b_{Q^{\prime }}\ \left( \in \mathcal{H}\right) & \text{ if } & x\in
Q^{\prime } \\ 
0\ \left( \in \mathcal{H}\right) & \text{ if } & x\not\in Q^{\prime }%
\end{array}%
\right. $ that is constant on $Q^{\prime }$ and vanishes off $Q^{\prime }$.

If $\left\{ b_{m}\right\} _{m=1}^{\infty }$ is any orthonormal basis for $%
\mathcal{H}$ then we define the finite-dimensional projections $%
\bigtriangleup _{\mathcal{H};Q}^{\mu ,b_{m}}$ onto%
\begin{equation*}
L_{\mathcal{H};Q}^{2}\left( \mu ;b_{m}\right) := \left\{
f=\sum\limits_{Q^{\prime }\in \mathcal{H}\left( Q\right) }a_{Q^{\prime
}}b_{m}\mathbf{1}_{Q^{\prime }}=b_{m}\sum\limits_{Q^{\prime }\in \mathfrak{C%
}\left( Q\right) }a_{Q^{\prime }}\mathbf{1}_{Q^{\prime }}:a_{Q^{\prime }}\in 
\mathbb{R},\int_{Q}fd\mu =0\right\} ,
\end{equation*}%
so that $\bigtriangleup _{\mathcal{H};Q}^{\mu }=\sum_{m=1}^{\infty
}\bigtriangleup _{\mathcal{H};Q}^{\mu ,b_{m}}$. Then we have the important
telescoping property for dyadic cubes $Q_{1}\subset Q_{2}$ that arises
from the martingale differences associated with the projections $%
\bigtriangleup _{Q}^{\mu ,b_{m}}$ for a fixed basis vector $b_{m}$:%
\begin{equation}
\mathbf{1}_{Q_{0}}\left( x\right) \left( \sum\limits_{Q\in \left[
Q_{1},Q_{2}\right] }\bigtriangleup _{\mathcal{H};Q}^{\mu ,b_{m}}f\left(
x\right) \right) =\mathbf{1}_{Q_{0}}\left( x\right) \left( \mathbb{E}%
_{Q_{0}}^{\mu ,b_{m}}f-\mathbb{E}_{Q_{2}}^{\mu ,b_{m}}f\right) ,\ \ \ \ \
Q_{0}\in \mathcal{H}\left( Q_{1}\right) ,\ f\in L^{2}\left( \mu \right) ,
\label{telescope}
\end{equation}%
where 
\begin{eqnarray*}
\mathbb{E}_{Q}^{\mu ,b_{m}}f\left( x\right) &:= &\left\{ 
\begin{array}{ccc}
E_{Q}^{\mu ,b_{m}}f & \text{ if } & x\in Q \\ 
0 & \text{ if } & x\not\in Q%
\end{array}%
\right. =\mathsf{P}_{b_{m}}^{\mathcal{H}}\mathbb{E}_{Q}^{\mu }f, \\
E_{Q}^{\mu ,b_{m}}f &:= &\int_{Q}\left\langle f\left( x\right)
,b_{m}\right\rangle _{\mathcal{H}}b_{m}d\mu \left( x\right) =\left\langle
\int_{Q}f\left( x\right) d\mu \left( x\right) ,b_{m}\right\rangle _{\mathcal{%
H}}b_{m}=\left\langle E_{Q}^{\mu }f,b_{m}\right\rangle _{\mathcal{H}}b_{m}=%
\mathsf{P}_{b_{m}}^{\mathcal{H}}E_{Q}^{\mu }f, \\
E_{Q}^{\mu }f &:= &\int_{Q}f\left( x\right) d\mu \left( x\right) \in 
\mathcal{H}\ .
\end{eqnarray*}%
Taking sums over projections we obtain the more general telescoping property
for any bounded projection $\mathsf{P}$ on $\mathcal{H}$: 
\begin{equation*}
\mathbf{1}_{Q_{0}}\left( x\right) \left( \sum\limits_{Q\in \left[
Q_{1},Q_{2}\right] }\mathsf{P}\bigtriangleup _{\mathcal{H};Q}^{\mu }f\left(
x\right) \right) =\mathbf{1}_{Q_{0}}\left( x\right) \left( \mathsf{P}\mathbb{%
E}_{Q_{0}}^{\mu }f-\mathsf{P}\mathbb{E}_{Q_{2}}^{\mu }f\right) ,\ \ \ \ \
Q_{0}\in \mathcal{H}\left( Q_{1}\right) ,\ f\in L^{2}\left( \mu \right) .
\end{equation*}%
It is sometimes convenient to use a fixed orthonormal basis $\left\{
h_{Q}^{\mu ,a,b_{m}}\right\} _{\substack{ a\in \Gamma _{n}  \\ 1\leq
m<\infty }}$ of $L_{Q}^{2}\left( \mu \right) $ where $\Gamma _{n}:=
\left\{ 0,1\right\} ^{n}\setminus \left\{ \mathbf{1}\right\} $ is a
convenient index set with $\mathbf{1}:=\left( 1,1,\ldots,1\right) $. Then $%
\left\{ h_{Q}^{\mu ,a,b_{m}}\right\} _{a\in \Gamma _{n}\text{ and }Q\in 
\mathcal{D}\text{ and }m\geq 1}$ is an orthonormal basis for $L^{2}\left(
\mu \right) $, with the understanding that we add the constant function $%
\mathbf{1}$ if $\mu $ is a finite measure. In particular we have%
\begin{eqnarray*}
\left\Vert f\right\Vert _{L_{\mathcal{H}}^{2}\left( \mu \right) }^{2}
&=&\sum_{Q\in \mathcal{D}}\left\Vert \bigtriangleup _{\mathcal{H};Q}^{\mu
}f\right\Vert _{L^{2}\left( \mu \right) }^{2}=\sum_{Q\in \mathcal{D}%
}\left\vert \widehat{f}\left(
Q\right) \right\vert ^{2},\\
\left\vert \widehat{f}\left( Q\right) \right\vert ^{2} &:= &\sum_{a\in
\Gamma _{n}\text{ }}\left\vert \left\langle f,h_{Q}^{\mu ,a}\right\rangle
_{\mu }\right\vert ^{2}=\sum_{a\in \Gamma _{n}\text{ }}\sum_{m=1}^{\infty
}\left\vert \left\langle f,h_{Q}^{\mu ,a,b_{m}}\right\rangle _{\mu
}\right\vert ^{2},
\end{eqnarray*}%
where the measure is suppressed in the notation $\widehat{f}$, along with
the parameters $a\in \Gamma _{n}$ and $m\geq 1$. Indeed, this follows from (%
\ref{telescope}) and Lebesgue's differentiation theorem for cubes. We also
have the following useful estimate. If $I^{\prime }$ is any of the $2^{n}$ $%
\mathcal{D}$-children of $I$, and $a\in \Gamma _{n}$, then 
\begin{equation}
\left\vert E_{I^{\prime }}^{\mu }h_{I}^{\mu ,a}\right\vert \leq \sqrt{%
E_{I^{\prime }}^{\mu }\left\vert h_{I}^{\mu ,a}\right\vert ^{2}}\leq \frac{1%
}{\sqrt{\mu\left( I^{\prime }\right)}}.  \label{useful Haar}
\end{equation}

Finally, let%
\begin{equation*}
\mathbb{E}_{Q}^{\mu }f\left( x\right) :=\left\{ 
\begin{array}{ccc}
E_{Q}^{\mu }f & \text{ if } & x\in Q \\ 
0 & \text{ if } & x\notin Q%
\end{array}%
\right.
\end{equation*}%
be projection onto the subspace $\mathcal{H}\mathbf{1}_{Q}$ of constant $%
\mathcal{H}$-valued functions on $Q$, and note that we have%
\begin{equation*}
f=\sum_{Q\in \mathcal{D}}\bigtriangleup _{\mathcal{H};Q}^{\mu }f,
\end{equation*}%
with convergence in $L_{\mathcal{H}}^{2}\left( \mu \right) $ since%
\begin{equation*}
\bigtriangleup _{\mathcal{H};Q}^{\mu }f=\left( \sum_{Q^{\prime }\in 
\mathcal{H}\left( Q\right) }\mathbb{E}_{Q^{\prime }}^{\mu }f\right) -%
\mathbb{E}_{Q}^{\mu }f=\sum_{Q^{\prime }\in \mathcal{H}\left( Q\right) }%
\mathbf{1}_{Q^{\prime }}\left( \mathbb{E}_{Q^{\prime }}^{\mu }f-\mathbb{E}%
_{Q}^{\mu }f\right) ,
\end{equation*}%
and the Hilbert space valued version of the dyadic Lebesgue differentiation
theorem gives%
\begin{equation*}
\lim_{Q\searrow x}\mathbb{E}_{Q}^{\mu }f=f\left( x\right) ,\ \ \ \ \ \text{%
for }\mu \text{-a.e. }x\in X.
\end{equation*}

\begin{description}
\item[Caution] While the scalar identity $\mathbf{1}_{Q^{\prime
}}\bigtriangleup _{\mathcal{H};Q}^{\mu }f=\mathbf{1}_{Q^{\prime
}}E_{Q^{\prime }}^{\mu }\bigtriangleup _{Q}^{\mu }f$ extends readily to the
Hilbert space setting, the operator identity%
\begin{equation*}
T_{\sigma }\left( \mathbf{1}_{Q^{\prime }}E_{Q^{\prime }}^{\mu
}\bigtriangleup _{\mathcal{H}_{1};Q}^{\mu }f\right) =\left( E_{Q^{\prime
}}^{\mu }\bigtriangleup _{\mathcal{H}_{1};Q}^{\mu }f\right) T_{\sigma
}\left( \mathbf{1}_{Q^{\prime }}\right) \ \ \ \ \ \text{in the scalar setting%
},
\end{equation*}%
\textbf{fails} in the Hilbert space setting where $T_{\sigma }:L_{\mathcal{H}%
_{1}}^{2}\rightarrow L_{\mathcal{H}_{2}}^{1,\limfunc{loc}}$ since $T_{\sigma
}\left( \mathbf{1}_{Q^{\prime }}E_{Q^{\prime }}^{\sigma }\bigtriangleup _{%
\mathcal{H}_{1};Q}^{\sigma }f\right) $ is a vector in $\mathcal{H}_{2}$,
while $E_{Q^{\prime }}^{\sigma }\bigtriangleup _{\mathcal{H}_{1};Q}^{\sigma
}f$ is a vector $\mathcal{H}_{1}$. Even if $\mathcal{H}=\mathcal{H}_{1}=%
\mathcal{H}_{2}$, an operator $T_{\sigma }$ does not typically commute with
an element in $\mathcal{H}$ unless $\mathcal{H}$ is the scalar field.
\end{description}
Nevertheless, we can indeed take the $\mathcal{H}$-norm of the element $%
E_{Q^{\prime }}^{\mu }\bigtriangleup _{\mathcal{H};Q}^{\mu }f$ outside the
operator $T_{\sigma }$, i.e. 
\begin{equation*}
T_{\sigma }\left( \mathbf{1}_{Q^{\prime }}E_{Q^{\prime }}^{\mu
}\bigtriangleup _{\mathcal{H};Q}^{\mu }f\right) =\left\vert E_{Q^{\prime
}}^{\mu }\bigtriangleup _{\mathcal{H};Q}^{\mu }f\right\vert _{\mathcal{H}%
}T_{\sigma }\left( \mathbf{1}_{Q^{\prime }}\mathbf{e}\right) ,\ \ \ \ \ 
\mathbf{e}:=\frac{E_{Q^{\prime }}^{\mu }\bigtriangleup _{\mathcal{H};Q}^{\mu
}f}{\left\vert E_{Q^{\prime }}^{\mu }\bigtriangleup _{\mathcal{H};Q}^{\mu
}f\right\vert _{\mathcal{H}}}.
\end{equation*}%
Note that this $\mathbf{e}$ is a unit vector in $\mathcal{H}
$, and this motivates our use of testing conditions of the form%
\begin{equation*}
\int_{Q}\left\vert T_{\sigma }\left( \mathbf{1}_{Q}\mathbf{e}\right)
\right\vert ^{2}d\omega \leq \mathcal{T}^2\sigma\left( Q\right),\ \ \ \ \ Q\in \mathcal{P}^{n}\text{ and }\mathbf{e}\in \func{unit%
}\mathcal{H}\ .
\end{equation*}

\begin{remark}
A stronger form of the testing condition, analogous to the
indicator/cube testing condition arising in connection with the two
weight norm inequality for the Hilbert transform \cite{LaSaShUr}, is this:%
\begin{equation*}
\int_{Q}\left\vert T_{\sigma }\left( \mathbf{1}_{Q}\mathbf{e}\right)
\right\vert ^{2}d\omega \leq \mathcal{T}^2\sigma\left( Q\right),\ \ \ \ \ Q\in \mathcal{P}^{n}\text{ and }\mathbf{e}:Q\rightarrow 
\func{unit}\mathcal{H}\text{ measurable}.
\end{equation*}
\end{remark}

\end{document}